\documentclass[10pt]{article}
\usepackage{amsthm,amsfonts,amsmath,amscd,amssymb,epsfig,verbatim}
\theoremstyle{plain}
\newtheorem{theorem}{Theorem}[section]
\newtheorem{thm}[theorem]{Theorem}

\newtheorem{cor}[theorem]{Corollary}

\newtheorem{prop}[theorem]{Proposition}
\newtheorem{lemma}[theorem]{Lemma}

\theoremstyle{remark}
\newtheorem{conjecture}[theorem]{Conjecture}

\newtheorem{remark}[theorem]{Remark}
\newtheorem{example}[theorem]{Example}

\newtheorem*{definition}{Definition}

\newcommand{\id}{\mathrm{id}}
\newcommand{\Id}{{{\mathchoice {\rm 1\mskip-4mu l} {\rm 1\mskip-4mu l}
{\rm 1\mskip-4.5mu l} {\rm 1\mskip-5mu l}}}}

\newcommand{\ol}{\overline}
\newcommand{\ul}{\underline}
\newcommand{\p}{\partial}
\newcommand{\pb}{\bar\partial}
\newcommand{\om}{\omega}
\newcommand{\Om}{\Omega}
\newcommand{\eps}{\varepsilon}
\newcommand{\into}{\hookrightarrow}
\newcommand{\la}{\langle}
\newcommand{\ra}{\rangle}
\newcommand{\wh}{\widehat}

\newcommand{\N}{{\mathbb{N}}}
\newcommand{\Z}{{\mathbb{Z}}}
\newcommand{\R}{{\mathbb{R}}}
\newcommand{\C}{{\mathbb{C}}}
\newcommand{\Q}{{\mathbb{Q}}}
\renewcommand{\P}{{\mathbb{P}}}

\newcommand{\J}{{\bf J}}
\newcommand{\z}{{\bf z}}
\newcommand{\f}{{\bf f}}

\newcommand{\I}{{\bf I}}
\newcommand{\GW}{{\rm GW}}
\newcommand{\Aut}{{\rm Aut}}
\newcommand{\ind}{{\rm ind}}
\newcommand{\im}{{\rm im }}        
\newcommand{\st}{{\rm st}}
\newcommand{\vol}{{\rm vol}}

\renewcommand{\min}{{\rm min}}
\renewcommand{\max}{{\rm max}}

\newcommand{\CZ}{{\rm CZ}}
\newcommand{\RS}{{\rm RS}}

\newcommand{\codim}{{\rm codim}}
\newcommand{\loc}{{\rm loc}}
\newcommand{\Ev}{{\rm Ev}}

\newcommand{\pt}{{\rm pt}}

\newcommand{\ev}{{\rm ev}}

\newcommand{\Tor}{{\rm Tor}}

\newcommand{\ord}{{\rm ord}}
\newcommand{\Hom}{{\rm Hom}}

\newcommand{\EE}{\mathcal{E}}
\newcommand{\BB}{\mathcal{B}}
\newcommand{\JJ}{\mathcal{J}}
\newcommand{\MM}{\mathcal{M}}

\renewcommand{\AA}{\mathcal{A}}
\newcommand{\UU}{\mathcal{U}}
\newcommand{\ZZ}{\mathcal{Z}}

\newcommand{\KK}{\mathcal{K}}
\title{Punctured holomorphic curves and Lagrangian embeddings}
\author{K.~Cieliebak and K.~Mohnke}
\date{}
\begin{document}
\maketitle

\begin{abstract}
We use a neck stretching argument for holomorphic curves to
produce symplectic disks of small area and Maslov class with boundary
on Lagrangian submanifolds of nonpositive curvature. Applications
include the proof of Audin's conjecture on the Maslov class of
Lagrangian tori in linear symplectic space, the construction of a new
symplectic capacity, obstructions to Lagrangian embeddings into
uniruled symplectic manifolds, a quantitative version of Arnold's
chord conjecture, and estimates on the size of Weinstein
neighbourhoods. The main technical ingredient is transversality for
the relevant moduli spaces of punctured holomorphic curves with
tangency conditions. 
\end{abstract}

\section{Introduction}

In this paper we use punctured holomorphic curves to establish some
new restrictions on Lagrangian embeddings. We will denote by {\em
manifold of nonnegative curvature} a manifold which admits a
Riemannian metric of nonpositive sectional curvature. 
In fact, the only property that enters the proofs is the existence of
a metric for which all closed geodesics are noncontractible and have
no conjugate points.
\medskip

{\bf Complex projective space. }
Consider the complex projective space
$\C\P^n$ with its standard symplectic structure $\om$, normalized such
that a complex line has symplectic area $\pi$. Let $D$ denote the
closed unit disk. The following are the two main results of this paper. 

\begin{thm}\label{thm:disk}
Let $L\subset \C\P^n$ be a closed Lagrangian submanifold which admits
a metric of nonpositive curvature. Then there exists a smooth map
$f:(D,\p D)\to(\C\P^n,L)$ with $f^*\om\geq 0$ whose symplectic area
satisfies
$$
   0<\int_D f^*\om\leq \frac{\pi}{n+1}.
$$
\end{thm}

\begin{thm}[Audin's conjecture]\label{thm:Audin}
The Maslov index of the map $f$ in Theorem~\ref{thm:disk} satisfies: 

(a) $\mu(f)\in\{1,2\}$ if $L$ is monotone;

(b) $\mu(f)=2$ if $L$ is a (not necessarily monotone) torus. 
\end{thm}

In particular, Theorem~\ref{thm:Audin} (b) answers a question posed by 
M.~Audin~\cite{Au} in 1988: {\em Every Lagrangian torus in $\C^n$ admits  
a disk of Maslov index $2$.}
This question was answered earlier for $n=2$ by Viterbo~\cite{Vi90}
and Polterovich~\cite{Po}, in the monotone case for $n\leq 24$ by
Oh~\cite{Oh}, and in the monotone case for general $n$ by
Buhovsky~\cite{Bu} and by Fukaya, Oh, Ohta and Ono~\cite[Theorem
  6.4.35]{FOOO}, see also Damian~\cite{Da}. 
A different approach has been outlined by Fukaya~\cite{Fu}. 
The scheme to prove Audin's conjecture using punctured holomorphic
curves was suggested by Y.~Eliashberg around 2001. The reason it took
over 10 years to complete this paper are transversality problems in
the non-monotone case. We solve these problems using techniques
from~\cite{CM-trans}. 

The proof of Theorem~\ref{thm:disk} actually yields $n+1$ maps
$f_0,\dots,f_n:(D,\p D)\to(\C\P^n,L)$, each of positive area (and
Maslov index $1$ or $2$ in the situation of Theorem~\ref{thm:Audin}),
such that $\sum_{i=0}^n\int_D f_i^*\om\leq
\pi$. We illustrate this with two examples. 

(1) The {\em Clifford torus} is the monotone torus in $\C\P^n$ defined
    by
$$
   T^n_{\rm Clifford} := \Bigl\{[z_0:\dots:z_n]\;\Bigl|\; |z_0|=\dots=|z_n|\Bigr\}. 
$$
It admits $n+1$ holomorphic maps $f_0,\dots,f_n:(D,\p D)\to(\C
P^n,T^n_{\rm Clifford})$ of area $\pi/(n+1)$ and Maslov index $2$ given by
$f_i(z):=[1:\dots1:z:1\dots 1]$ with $z$ in the $i$-th component. Note
that the boundary loops $f_i(\p D)$ for $0\leq i\leq n$ generate the
first homology $H_1(T^n)$. 

(2) The {\em Chekanov torus} is an exotic monotone 2-torus $T^2_{\rm
  Chekanov}$ in $\C\P^2$ described in~\cite{CS}. 
We show in Appendix~\ref{app:Chekanov} that for the disks
$f_0,f_1,f_2$ obtained in Theorem~\ref{thm:disk}, all boundary loops
$f_i(\p D)$ represent multiples of the same homology class. 
This answers in the negative Viterbo's question whether for $L\cong
T^n$ the boundary loops $f_i(\p D)$ for $0\leq i\leq n$ always 
generate the first homology $H_1(T^n)$.

{\bf A new symplectic capacity. }
To explore the implications of Theorem~\ref{thm:disk}, we define a new
symplectic capacity, following a suggestion by J.~Etnyre. Define the
{\em minimal symplectic area} of a Lagrangian submanifold $L$ of a
symplectic manifold $(X,\om)$ by 
$$
        A_\min(L) := \inf\Bigl\{\int_\sigma\om\;\Bigl|\;\sigma\in\pi_2(X,L),
        \int_\sigma\om>0\Bigr\}\in[0,\infty].
$$
Define the {\em Lagrangian capacity} $c_L(X,\om)$
\footnote{The ``L'' in $c_L$ just stands for ``Lagrangian'' and does
  not refer to a particular Lagrangian submanifold $L$.}
of $(X,\om)$ to be
$$
        c_L(X,\om) := \sup\{A_\min(L)\;|\;L\subset X \text{ embedded
        Lagrangian torus}\}\in[0,\infty].
$$
The Lagrangian capacity satisfies
\begin{description}
\item[(Monotonicity)] $c_L(X,\om)\leq c_L(X',\om')$ if there exists a
symplectic embedding $\iota:(X,\om)\into(X',\om')$ with
$\pi_2\bigl(X',\iota(X)\bigr)=0$; 
\item[(Conformality)] $c_L(X,\alpha\om)=|\alpha| c_L(X,\om)$ for $0\neq
\alpha\in\R$;
\item[(Nontriviality)] $0<c_L\bigl(B^{2n}(1)\bigr)$ and
$c_L\bigl(Z^{2n}(1)\bigr)<\infty$.
\end{description}
In particular, $c_L$ is a (generalized) capacity in the sense
of~\cite{HZ,CHLS} on the class of symplectic manifolds $(X,\om)$
with $\pi_1(X)=\pi_2(X)=0$. 
Here $B^{2n}(r)$ is the open unit ball in $\C^n$ and
$Z^{2n}(r)=B^2(r)\times\C^{n-1}$; open subsets of $\C^n$ are always
equipped with the canonical symplectic structure. The first two
properties are obvious. The property $0<c_L\bigl(B^{2n}(1)\bigr)$ holds because
the unit ball contains a small {\em standard torus}
$
        T^n(r):=S^1(r)\times\dots\times S^1(r),
$
where $S^1(r)\subset\C$ is the sphere of radius $r$. For the last
property, recall Chekanov's result~\cite{Ch} that
$$
        A_\min(L) \leq d(L)
$$
for every closed Lagrangian submanifold $L$ in $\C^n$. Here $d(A)$ is
the {\em displacement energy} of a subset $A\subset\C^n$, i.e., the
minimal Hofer energy of a compactly supported Hamiltonian
diffeomorphism that displaces $A$ from itself (see~\cite{HZ}). This
implies
$$
        c_L(U) \leq d(U)
$$
for every open subset $U\subset\C^n$. Since the symplectic cylinder
$Z^{2n}(1)$ has displacement energy $\pi$ and contains the standard
torus $T^n(1)$ of minimal symplectic area
$\pi$, it follows that
$$
        c_L\bigl(Z^{2n}(1)\bigr) = \pi.
$$
Most known symplectic capacities (exceptions being the higher
Ekeland--Hofer capacities~\cite{EH} for subsets of $\C^n$ and the ECH
capacities~\cite{Hut} in dimension $4$) take the same value on the unit ball
and on the cylinder. Surprisingly, this is not the case for the
Lagrangian capacity:

\begin{cor}\label{cor:cap}
$$
        c_L\bigl(B^{2n}(1)\bigr) = \frac{\pi}{n}.
$$
\end{cor}

\begin{proof}
The standard torus $T^n(r)$ is contained in the unit ball if and only
if $r<1/\sqrt{n}$. This proves one inequality. For the converse
inequality, suppose that the unit ball contains a Lagrangian torus
$L$ with $A_\min(L)\geq \pi/n$. After replacing $L$ by $\alpha L$ for
a suitable $0<\alpha\leq 1$ we may assume that
$A_\min(L)=\pi/n$. Compactifying the closed unit ball to $\C\P^n$, $L$
gives rise to a Lagrangian torus $L'$ in $\C\P^n$ with the property that
all disks with boundary on $L'$ have symplectic area a multiple of
$\pi/n$. (This property is clear for disks contained in the open ball
$\C\P^n\setminus\C\P^{n-1}$; for a disk passing through $\C\P^{n-1}$
it follows by gluing it along its boundary with a disk in the ball to
a sphere whose area is a multiple of $\pi$.) But this property
contradicts Theorem~\ref{thm:disk}. 
\end{proof}

For an application of this result, consider a {\em polydisk}
$
        P^{2n}(r) := B^2(r)\times\dots\times B^2(r).
$
It contains the standard torus $T^n(r)$ and is contained in
the cylinder $Z^{2n}(r)$, hence
$$
        c_L\bigl(P^{2n}(r)\bigr) = \pi r^2.
$$
Monotonicity of the Lagrangian capacity now implies the following
non-squeezing result due to Ekeland and Hofer. 

\begin{figure}[ht]\label{fig:polydisk}
\begin{picture}(110,70)
\put(50,10){\qbezier(0,20)(0,0)(20,0)}
\put(70,10){\qbezier(0,0)(20,0)(20,20)}
\put(70,30){\qbezier(0,20)(20,20)(20,0)}
\put(50,30){\qbezier(0,0)(0,20)(20,20)}
\put(55,15){\line(0,1){30}}
\put(55,15){\line(1,0){30}}
\put(85,45){\line(0,-1){30}}
\put(85,45){\line(-1,0){30}}
\put(55,15){\circle*{2}}
\put(85,45){\circle*{2}}
\put(55,45){\circle*{2}}
\put(85,15){\circle*{2}}
\put(63,29){$P^{2n}(\frac{1}{\sqrt{n}})$}
\put(65,52){$B^{2n}(1)$}
\put(88,44){$T^n(\frac{1}{\sqrt{n}})$}
\end{picture}
\caption{A polydisc inscribed in a ball.}
\end{figure}
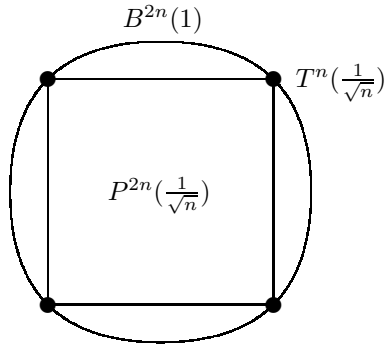

\begin{cor}[Ekeland-Hofer~\cite{EH}]
The polydisk $P^{2n}(r)$ can be symplectically embedded into the ball
$B^{2n}(1)$ if and only if
$$
        r \leq \frac{1}{\sqrt{n}}.
$$
\end{cor}

We end this paragraph with a conjecture for the Lagrangian capacity of
an {\em ellipsoid}
$$
   E(a_1,\dots,a_n) := \left\{z\in\C^n\Bigl|\;
   \frac{|z_1|^2}{a_1}+\dots+\frac{|z_n|^2}{a_n}\leq 1\right\}
$$
with $0<a_1\leq\dots\leq a_n<\infty$. 

\begin{conjecture}
The Lagrangian capacity of an ellipsoid is given by
$$
   c_L\Bigl(E(a_1,\dots,a_n)\Bigr) = \frac{\pi}{1/a_1+\dots+1/a_n}. 
$$
\end{conjecture}

It is shown in~\cite{CHLS} that this conjecture would imply
$c_L=\lim_{k\to\infty}\frac{1}{k}c_k^{\rm EH}$ on all ellipsoids, where
$c_k^{\rm EH}$ denotes the $k$-th Ekeland--Hofer capacity~\cite{EH}. 
It would be interesting to see whether this relation continues to hold
on more general subsets of $\R^{2n}$, e.g.~on convex sets. 
The relation may also be compared to a similar formula recovering the
volume of a $4$-dimensional Liouville domain from its ECH
capacities~\cite{CHG}. 

\begin{remark}
If we define the Lagrangian capacity $c_L$ using arbitrary closed
Lagrangian submanifolds instead of just tori we still obtain a
capacity with the property $c_L(U)\leq d(U)$ for $U\subset\C^n$. The
Lagrangian capacity of the unit ball is still $\geq\pi/n$, but we do
not know whether equality holds for $n>2$. For $n=2$, equality holds in
view of Theorem~\ref{thm:disk} (and the proof of
Corollary~\ref{cor:cap}) because all closed Lagrangian submanifolds of
$\C^2$ admit a metric of nonpositive curvature.
\end{remark}
\medskip

{\bf Extremal Lagrangian tori. }
Let us call a Lagrangian torus $L$ in a symplectic manifold $(X,\om)$
{\em extremal} if it maximizes $A_\min$, i.e., $A_\min(L)=c_L(X,\om)$. 
Recall that $L$ is {\em monotone} if its Maslov class
is positively proportional to the symplectic area class on
$\pi_2(X,L)$, i.e., $\mu=2a[\om]$ for some $a>0$. This
implies that $(X,\om)$ is monotone with $c_1(TX)=a[\om]$ on
$\pi_2(X)$ (see e.g.~\cite{CG}). In particular, for $X=\C\P^n$ with
its standard symplectic form we have $a=(n+1)/\pi$, so for a
monotone Lagrangian torus $L\subset \C\P^n$ the values of $[\om]$ on
$\pi_2(\C\P^n,L)$ are integer multiples $\pi/(n+1)$ (since all Maslov
indices are even). Therefore, Theorem~\ref{thm:disk} implies

\begin{cor}\label{cor:mon}
Every monotone Lagrangian torus in $\C\P^n$ is extremal. 
\end{cor}

\begin{conjecture}
Every extremal Lagrangian torus in $\C\P^n$ is monotone. 
\end{conjecture}

Turning to a different manifold, the following conjecture is motivated
by a question of L.~Lazzarini.

\begin{conjecture}
Every extremal Lagrangian torus in the unit ball $B^{2n}(1)$ is
entirely contained in the boundary $\p B^{2n}(1)$. 
\end{conjecture}

\begin{remark}\label{rem:rigid-intro}
Note that the standard torus $T^n(1/\sqrt{n})\subset\p B^{2n}(1)$ is
extremal. The following fact, which is proved in
Appendix~\ref{app:rigid}, lends some plausibility to the conjecture. 
It establishes a phenomenon which is called ``non-removable
intersection'' in~\cite{PPS}: 

{\em For a closed Lagrangian submanifold $L\subset\p B^{2n}(1)$ no
point can be pushed into the interior by a Hamiltonian isotopy
without making $L$ exit the closed ball at some other point.}
\end{remark}

\begin{remark}
One motivation for our interest in monotone Lagrangian submanifolds is
the observation (see e.g.~\cite{CG}) that minimal (i.e., of zero mean
curvature) Lagrangian submanifolds in $\C\P^n$ are monotone. 
\end{remark}

{\bf The chord conjecture. }
Another application concerns {\em Arnold's chord
conjecture}~\cite{Ar}. Let $U\subset\C^n$ be a bounded star-shaped
domain (with respect to the origin) with smooth boundary $S=\p U$. The
1-form
$$
        \lambda := \frac{1}{2}\sum_{j=1}^n(x_jdy_j-y_jdx_j)
$$
on $\C^n$ induces a contact form on $S$, and every contact form on the
sphere $S^{2n-1}$ defining the standard contact structure arises in
this way (see Section~\ref{sec:punctured} for the basic definitions
concerning contact 
manifolds). A {\em Reeb chord of length $T$} to a Legendrian submanifold
$\Lambda\subset S$ is an orbit $\gamma:[0,T]\to S$ of the Reeb vector
field with $\gamma(0),\gamma(T)\in\Lambda$. The chord conjecture
states that every closed Lagrangian submanifold of $S$ possesses a
Reeb chord. It was proved by the second author in~\cite{Mo1}.

To see the relation to the Lagrangian capacity, let us recall the
proof in~\cite{Mo1}. Suppose that there exists no Reeb chord of length
$\leq T$. Then we can construct a Lagrangian submanifold
$L\subset\C^n$ out of $\Lambda$ as follows. Move $\Lambda$ in $S$ by
the Reeb flow until time $T$, then push it down radially in $\C^n$ to
the sphere $\eps S$, $\eps>0$, move it in $\eps S$ by the backward
Reeb flow until time $-T$ and push it up radially to $S$. Smoothing
corners, this yields a Lagrangian submanifold $L\subset\C^n$
diffeomorphic to $S^1\times\Lambda$ with minimal symplectic area
$$
        A_\min(L) = (1-\eps)T.
$$
Since $A_\min(L)\leq d(L)\leq d(U)$ by Chekanov's theorem, this proves
the result in~\cite{Mo1}: {\sl Every closed Legendrian submanifold $\Lambda$
in $S=\p U$ possesses a Reeb chord of length $T\leq d(U)$.} If
$\Lambda$ admits a metric of nonpositive curvature the above argument
yields the following sharper estimate.

\begin{cor}\label{cor:chord}
Every closed Legendrian submanifold $\Lambda$ of
nonpositive curvature in the boundary $S$ of a star-shaped
domain $U\subset\C^n$ possesses a Reeb chord of length
$$
        T \leq c_L(U).
$$
\end{cor}

Let us compare the two estimates for the unit sphere $S$ in
$\C^n$. Here all Reeb orbits are closed of length $\pi$ and form the
fibres of the Hopf fibration $S^{2n-1}\to\C\P^{n-1}$. The estimate
in~\cite{Mo1} yields a Reeb chord of length $T\leq\pi$, which may just
be a closed Reeb orbit meeting $\Lambda$ once. The estimate in
Corollary~\ref{cor:chord} combined with Corollary~\ref{cor:cap} yields

\begin{cor}\label{cor:chord-sphere}
Every closed Legendrian submanifold $\Lambda$ of
nonpositive curvature in the unit sphere $S\subset\C^n$,
$n\geq 2$, possesses a Reeb chord of length
$$
        T \leq \frac{\pi}{n}.
$$
In particular, $\Lambda$ meets some fibre of the Hopf fibration at
least twice.
\end{cor}

In a recent preprint~\cite{Zi}, F.~Ziltener shows that the last
statement actually holds without the assumption of nonpositive curvature. This
follows from the proof in~\cite{Mo1} described above and the
observation that the displacement energy of the Lagrangian submanifold
$L\cong S^1\times\Lambda$ is strictly smaller than $\pi$. 

The last statement can be rephrased as follows. Call a Lagrangian
submanifold $L\subset\C\P^n$ {\em exact} if every disk with boundary
on $L$ has symplectic area an integer multiple of $\pi$. Every exact
Lagrangian submanifold $L\subset\C\P^n$ lifts to a Legendrian
submanifold $\Lambda\subset S^{2n+1}$ that intersects each fibre of
the Hopf fibration at most once. (A horizontal lift of $L$ exist
locally since $L$ is Lagrangian; the obstructions to a global
horizontal lift are the holonomies of the connection form $2\lambda$
along loops in $L$, which vanish because the integral of the curvature
$2\om$ over each disk with boundary on $L$ is an integer multiple of
$2\pi$.) Hence 
the preceding discussion implies

\begin{cor}[\cite{Zi}]\label{cor:exact}
There exists no exact closed Lagrangian submanifold
in $\C\P^n$.
\end{cor}

For $n=1$ this is obvious (and was the motivation for Arnold's
conjecture in~\cite{Ar}). For $n>1$ it generalizes (for manifolds of
nonpositive curvature) Gromov's result~\cite{Gr} that there is no
closed exact Lagrangian submanifold in $\C^n$. However, Gromov's
original technique does not seem applicable here.

The definition of a Reeb chord given above includes periodic orbits
intersecting $\Lambda$. Let us call a Reeb chord {\em honest} if it is
not a closed Reeb orbit meeting $\Lambda$ just once. Arnold's original
conjecture asks for the existence of honest Reeb chords.
Corollary~\ref{cor:chord} can be used to prove this stronger
conjecture. For example, if $S\subset\C^n$ is $C^1$-close to the unit
sphere, then all closed Reeb chords have length at least $\pi-\eps$,
whereas Corollary~\ref{cor:chord} yields a Reeb chord of length at
most $\pi/n+\eps$. Hence we have

\begin{cor}
Let $S\subset\C^n$, $n\geq 2$, be sufficiently $C^1$-close to the unit
sphere. Then every closed Legendrian submanifold $\Lambda$ 
of nonpositive curvature in $S$ possesses an
honest Reeb chord.
\end{cor}

{\bf Lagrangian submanifolds and symplectic balls. }
P.~Biran~\cite{Bi} has shown that a symplectically embedded ball
$B\subset\C\P^n$ of radius $r>1/\sqrt{2}$ must intersect $\R
P^n\subset\C\P^n$. This turns out to be related to a quite general
intersection phenomenon.

\begin{thm}\label{thm:ball}
Let $L\subset\C\P^n$ be a closed Lagrangian submanifold which admits a
metric without contractible geodesics (e.g.~one of nonpositive curvature). Let
$B\subset\C\P^n$ be a symplectically embedded ball of radius $r$. If
$L\cap B=\emptyset$, then
$$
        A_\min(L)+\pi r^2\leq\pi.
$$
\end{thm}

Note that $A_\min(\R\P^n)=\pi/2$, so the estimate in
Theorem~\ref{thm:ball} would imply Biran's result. However,
the theorem is not applicable to $L=\R\P^n$ because every
metric on $\R P^n$ has contractible geodesics. Combined with
Corollary~\ref{cor:mon}, Theorem~\ref{thm:ball} implies that every
monotone Lagrangian torus in $\C\P^n$ must intersect every embedded
ball of radius $r>\sqrt{\frac{n}{n+1}}$; see Biran and
Cornea~\cite{BC} for generalizations of this result. 
\medskip

{\bf Uniruled symplectic manifolds. }
Let us call a closed symplectic manifold $(X,\om)$ {\em uniruled} if
it has some
nonvanishing Gromov-Witten invariant of holomorphic spheres passing
through a point (plus additional constraints). Ruan~\cite{Ru} has
proved that on a K\"ahler manifold this is equivalent to the
algebro-geometric definition that $X$ is covered by rational
curves. For example, every Fano manifold is uniruled (see~\cite{Ko96}).
The following result is due to Viterbo~\cite{Vi00} for "strongly Fano
manifolds", and to Eliashberg, Givental and Hofer~\cite{EGH} in the
general case. We include its proof in Section~\ref{sec:proofs}.

\begin{thm}[\cite{Vi00,EGH}]\label{thm:uniruled}
Let $L$ be a closed manifold of dimension $\geq 3$ which carries a
metric of negative curvature. Then $L$ admits no Lagrangian embedding
into a uniruled symplectic manifold.
\end{thm}

In dimension 2 the result is not true because of
Givental's nonorientable Lagrangian surfaces in $\C^2$
(see~\cite{ALP}). But we still obtain the following restriction.

\begin{thm}\label{thm:uniruled-dim2}
A closed orientable surface of genus $\geq 2$ admits no Lagrangian
embedding into a uniruled symplectic 4-manifold.
\end{thm}

\begin{remark}
The proof of Theorem~\ref{thm:disk} yields for
every closed Lagrangian surface $L\subset\C\P^2$ which carries a
metric of negative curvature the estimate $A_\min(L)\leq\pi/6$.
\end{remark}

In ~\cite{Na}, J.~Nash proved that every compact smooth manifold is the
real part of a real algebraic manifold (i.e., a smooth projective
variety defined by real equations in $\C\P^N$). In the same paper he
conjectured that the real algebraic manifold can be chosen to be
birational to $\C\P^n$. Now the real part of a real algebraic
manifold is Lagrangian (see~\cite{Vi00}), and an algebraic manifold
birational to $\C\P^n$ is uniruled (it is even rationally connected,
see~\cite{Ko01}) and simply connected (see~\cite{GH}). Hence
Theorems~\ref{thm:uniruled} and~\ref{thm:uniruled-dim2} imply

\begin{cor}[\cite{Co,Ko98,Vi00,EGH}]
The Nash conjecture is false in all dimensions $\geq 2$. More
precisely, we have:

(a) For $n\geq 3$, the real part of a real algebraic manifold
birational to $\C\P^n$ carries no metric of negative curvature.

(b) The real part of a real algebraic surface birational to $\C\P^2$
is either the sphere, the torus, or a nonorientable surface.
\end{cor}

For $n=2$ the result is due to Comessatti~\cite{Co} and sharp: the
sphere and the torus occur in quadrics $x^2+y^2\pm z^2=t^2$, and the
nonorientable surfaces occur in blow-ups of $\C\P^2$ at real points. Already
for $n=3$ the result is far from optimal. In fact, Koll\'ar has
derived a short list of possible topological types for real parts of
real algebraic 3-folds~\cite{Ko98,Ko01}.

Suppose now that the values of the symplectic form $\om$ on $\pi_2(X)$
are given by $k\,a$, $k\in\Z$, for some $a>0$. Let us call $(X,\om)$
{\em minimally uniruled} if the holomorphic spheres in the definition
of "uniruled" have symplectic area $a$. Call a Lagrangian embedding
$L\into X$ {\em exact} if every disk with boundary on $L$ has
symplectic area an integer multiple of $a$. With these notations,
Corollary~\ref{cor:exact} can be generalized as follows.

\begin{thm}\label{thm:exact}
Let $L$ be a closed manifold which carries a
metric without contractible geodesics. Then $L$ admits no exact
Lagrangian embedding into a minimally uniruled symplectic manifold.
\end{thm}

Examples of minimally uniruled symplectic manifolds are $\C\P^n$ and
$\C\P^1\times \C\P^1$ with their standard symplectic structures. The
real locus $\R P^n\subset\C\P^n$ is not exact, while the antidiagonal
in $\C\P^1\times\C\P^1$ is exact.

{\bf Size of Weinstein neighbourhoods. }
Weinstein's Lagrangian neighbourhood theorem states that any
Lagrangian submanifold $L\subset(X,\om)$ has a tubular neighbourhood
which is symplectomorphic to a neighbourhood of the zero section in
the cotangent bundle $T^*L$ with the canonical symplectic form. We are
interested in the maximal ``size'' of such a neighbourhood, which can
be quantified as follows. . 
For a Riemannian manifold $(Q^n,g)$ let $D^*(Q,g):=\{(q,p)\in T^*Q\mid
\|p\|_g\leq 1\}$ be the unit codisk bundle. Fix a symplectic manifold
$(X^{2n},\om)$ and define the {\em embedding capacity}
$$
   c^{(X,\om)}(Q,g) := \inf\{\alpha>0\mid
   D^*(Q,g)\into(X,\alpha\om)\},
$$
where $\into$ means symplectic embedding with respect to the canonical
symplectic structure on $T^*Q$. This is the embedding capacity of
$D^*(Q,g)$ into $(X,\om)$ in the sense of~\cite{CHLS}. It has the
following elementary properties with respect to the metric:

\begin{description}
\item[(Invariance)] $c^{(X,\om)}(Q,g)=c^{(X,\om)}(Q,g)$ if $(Q,g)$ and
  $(Q',g')$ are isometric;
\item[(Monotonicity)] $c^{(X,\om)}(Q,g)\leq c^{(X,\om)}(Q,h)$ if
  $g\leq h$;
\item[(Conformality)] $c^{(X,\om)}(Q,\lambda
g)=\lambda c^{(X,\om)}(Q,g)$ for $\lambda>0$;
\item[(Nontriviality)] $c^{(X,\om)}(Q,g)<\infty$ if $Q$ admits
  Lagrangian embeddings into $(X,\om)$, and $c^{(X,\om)}(Q,g)>0$ if
  $(X,\om)$ has finite volume.  
\end{description}
For last property, note that a symplectic embedding
$D^*(Q,g)\into(X,\alpha\om)$ yields
$$
   \vol(Q,g)\vol B^n(1)=\vol D^*(Q,g)\leq\vol(X,\alpha\om) =
   \alpha^n\vol(X,\om) 
$$
and hence the lower estimate
\begin{equation}\label{eq:volume}
   c^{(X,\om)}(Q,g) \geq \sqrt[n]{\frac{\vol(Q,g)\vol
   B^n(1)}{\vol(X,\om)}}. 
\end{equation}
Let $\ell_\min(Q,g)$ denote the minimal length of a closed geodesic on
$(Q,g)$. Then the methods of this paper yield the following lower estimate. 

\begin{thm}\label{thm:emb-cap}
The embedding capacity of a Riemannian manifold $(Q,g)$ of nonpositive
curvature into $\C\P^n$ with its standard symplectic form satisfies 
$$
   c^{\C\P^n}(Q,g)\geq\frac{(n+1)\ell_\min(Q,g)}{\pi}.
$$
\end{thm}

Some explicit computations in Appendix~\ref{app:emb-cap} show:

(1) For the flat torus $T^n=(\R/2\pi\Z)^n$, the estimate in
Theorem~\ref{thm:emb-cap} is better than the volume
estimate~\eqref{eq:volume} for each $n>1$. Moreover, the volume bound
grows as $\sqrt{n}$, while the bound in Theorem~\ref{thm:emb-cap} grows
linearly in $n$. 

(2) Explicit embeddings of the flat torus $T^n=(\R/2\pi\Z)^n$ into $\C 
P^n$ yield an upper bound $c^{\C\P^n}(T^n)\leq 2(n+\sqrt{n})$, which
is stricly larger than the lower bound $2(n+1)$ from
Theorem~\ref{thm:emb-cap} for each $n>1$. This suggests that the
estimate in Theorem~\ref{thm:emb-cap} may not be sharp. Note,
however, that the quotient of the upper and lower bound tends to $1$
as $n\to\infty$. 

\begin{remark}
(a) The techniques of this paper work well for Lagrangian submanifolds
which admit metrics of nonpositive curvature, or at least without
contractible geodesics. On the other hand, traditional techniques
such as Gromov's holomorphic disks and Floer homology work well for
simply connected Lagrangians. Can one combine these techniques? A test
case could be the product of a manifold of positive with a manifold
of negative curvature.\\
(b) Neck stretching techniques were also used in~\cite{We,MW} to study
Lagrangian embeddings in dimensions $4$ and $6$. 
\end{remark}

{\bf Acknowledgements. }
We thank 
Y.~Chekanov,
Y.~Eliashberg,
J.~Etnyre,
L.~Lazzarini,
D.~McDuff,
L.~Polterovich,
and C.~Viterbo
for fruitful discussions.

\section{Punctured holomorphic curves}\label{sec:punctured}

In this section we recall the definitions and basic properties of
punctured holomorphic curves, see~\cite{EGH,BEHWZ,CM-comp} for details. 

{\bf Contact manifolds. }
A {\em contact form} on a $(2n-1)$--dimensional manifold
$M^{2n-1}$ is a 1-form $\lambda$  for which $\lambda\wedge (d\lambda)^{n-1}$ is a
volume form. Its {\em Reeb vector field} is the unique vector field
$R_\lambda$ satisfying $i_{R_\lambda}d\lambda=0$ and
$\lambda(R_\lambda)=1$. The distribution $\xi:=\ker\lambda$ is called
a {\em (cooriented) contact structure}, and the pair $(M,\xi)$ a {\em contact manifold}.

Consider a closed orbit $\gamma$ of the Reeb vector field. The
linearized Reeb flow along $\gamma$ 
preserves the contact distribution $\xi$, and the maps
$\Psi_t:\xi_{\gamma(0)}\to\xi_{\gamma(t)}$ preserve $d\lambda$.
Let $\Phi:\xi|_\gamma\overset{\cong}{\to}S^1\times\C^{n-1}$ be a
trivialization of $\xi$ along $\gamma$. Then $\Phi_t\Psi_t\Phi_0^{-1}$
is a path of symplectic matrices which has a {\em Conley-Zehnder
index} (defined in~\cite{CZ} if $\gamma$ is {\em nondegenerate},
i.e., the linearized return map on a transverse section has no
eigenvalue $1$, and in~\cite{Vi90} in general) and a {\em
  Robbin-Salamon index} (defined in~\cite{RS}). We call
$$\CZ(\gamma,\Phi):=\CZ(\Phi_t\Psi_t\Phi_0^{-1})\in\Z,\qquad 
\RS(\gamma,\Phi):=\RS(\Phi_t\Psi_t\Phi_0^{-1})\in\frac{1}{2}\Z$$
the Conley-Zehnder (resp.~Robbin-Salamon) index of $\gamma$ with
respect to the trivialization $\Phi$. 

Following~\cite{Bo,BEHWZ}, we say that $(M,\lambda)$ is   
{\em Morse-Bott} if the following holds: For all $T>0$
the set $N_T\subset M$ formed by the $T$-periodic Reeb
orbits is a closed submanifold, the rank of
$d\lambda|_{N_T}$ is locally constant, and
$T_pN_T=\ker(T_p\phi_T-\Id)$ for all
$p\in N_T$, where $\phi_t$ is the Reeb flow. Dividing out
time reparametrization, we obtain the orbifold of unparametrized Reeb
orbits 
$\Gamma_T:=N_T/S^1.$
Note that
the case $\dim \Gamma_T=\dim N_T-1=0$ corresponds to nondegeneracy of
closed Reeb orbits. We will refer to this case as the {\em Morse
case}. In the Morse-Bott case, the Conley-Zehnder and Robbin-Salamon
indices are constant along each connected component $\Gamma$ of $\Gamma_T$ and
we denote them by $\CZ(\Gamma)$ resp.~$\RS(\Gamma)$. They are related by 
\begin{equation}\label{eq:CZ-RS}
   \CZ(\Gamma) = \RS(\Gamma) - \frac{1}{2}\dim \Gamma. 
\end{equation}

An important special case is the {\em geodesic flow} on the unit
cotangent bundle $M=S^*L$ of a Riemannian manifold $(L,g)$. This is
the Reeb flow of the contact form $\lambda=p\,dq$ in canonical
coordinates $(q,p)$ on $T^*L$. A closed geodesic $q$ on $L$
lifts to the closed Reeb orbit
$\gamma(t)=\bigl(q(t),\dot q(t)\bigr)$.
A trivialization $\Phi$ of the contact
distribution along $\gamma$ is equivalent to a trivialization of
$T(T^*L)$ as a symplectic vector bundle along $q$ which is still denoted by $\Phi$. 
Note that the zero section $L\subset T^*L$ is Lagrangian for $d\lambda$. Let
$$\mu(\gamma,\Phi) := \mu(q,\Phi) :=
\mu(\Phi_tT_{q(t)}L)$$
be the {\em Maslov index} of the induced loop of Lagrangian subspaces
of $\C^n$ (see e.g.~\cite{ALP}).
Now $\CZ(\gamma,\Phi)$ and $\mu(\gamma,\Phi)$ transform in the same
way under change of the trivialization. So their difference is
independent of $\Phi$, and in fact is a well-known quantity:

\begin{lemma}[Viterbo~\cite{Vi90}]\label{lem:CZ}
For the lift $\gamma$ of a (possibly degenerate) closed geodesic $q$,
$$\CZ(\gamma,\Phi) + \mu(\gamma,\Phi) = \ind(\gamma),$$
where $\ind(\gamma):=\ind(q)$ is the Morse index of $q$.
\end{lemma}

We will need the following  lemma about Morse indices:

\begin{lemma}\label{lem:Morse}
Let $L^n$ be a closed manifold which admits a metric of nonpositive
sectional curvature. Then for every $c>0$ there exists a metric on
$L$ such that every closed geodesic $q$ of length $\leq c$ is
noncontractible and nondegenerate and satisfies
$$0\leq\ind(q)\leq n-1.$$
In a metric of negative curvature every closed geodesic is
noncontractible and nondegenerate of index zero.
\end{lemma}

\begin{proof}
By the Morse Index Theorem for closed geodesics (see e.g.~\cite{Kl2},
Theorem 2.5.14),
$$
   \ind(q) = \ind_\Om(q) + {\rm concav}(q),
$$
where $\ind_\Om(q)$ equals the number of conjugate points along $q$,
and the nullity and concavity of $q$ satisfy ${\rm null}(q)+{\rm
concav}(q)\leq n-1$. Now a metric $g_0$ of nonpositive curvature has
no conjugate points along any geodesic (see e.g.\cite{Kl2}, Theorem
2.6.2), hence $\ind(q)+{\rm null}(q)\leq n-1$ for every closed
$g_0$-geodesic.

Consider the space of all closed $g_0$-geodesics of length $\leq
c$, parametrized with constant speed over the time
interval $[0,1]$. Since this space is compact (say, with the $C^2$
topology), there exists a $C^2$-neighbourhood $\UU$ of $g$ in the
space of metrics such that, for every $g\in\UU$ and every closed
$g$-geodesic $q$ of length $\leq c$, we still have $\ind(q)+{\rm
null}(q)\leq n-1$. Now there exist metrics $g\in\UU$ for which all
closed geodesics $q$ of length $\leq c$ are nondegenerate (see
e.g.~\cite{Kl1}, Theorem 3.3.9), and therefore $\ind(q)\leq n-1$.

By Hadamard's Theorem (see e.g.~\cite{Kl2}, Theorem 2.6.6), all closed
$g_0$-geodesics are noncontractible; this persists for $g$-geodesics
of length $\leq c$ for $g\in\UU$. Finally, if $g_0$ has negative
curvature, then the nullity and concavity are zero, so every closed
$g_0$-geodesic is nondegenerate of index $0$.
\end{proof}

{\bf Symplectic cobordisms. }
The {\em symplectization} of a contact manifold $(M,\xi)$
is the symplectic manifold $\bigl(\R\times M,d(e^r\lambda)\bigr)$, where $r$
is the coordinate on $\R$. A {\em symplectic cobordism} is a
symplectic manifold $(X,\om)$ whose ends are the positive or negative
halves of symplectizations. This means that for a some compact subset
$K\subset X$ and contact manifolds $(\ol{M},\ol{\lambda})$
$(\ul{M},\ul{\lambda})$, we have a symplectomorphism
$$
        (X\setminus K,\om)\cong
        \bigl(\R_+\times\ol{M},d(e^r\ol{\lambda})\bigr) \amalg
        \bigl(\R_-\times\ul{M},d(e^r\ul{\lambda})\bigr).
$$
Here we allow one or both of $\ol{M},\ul{M}$ to be empty. The ends
modelled over $\ol{M}$ and $\ul{M}$ are called the {\em positive and
negative ends}, respectively. Obvious examples of symplectic cobordisms
are closed symplectic manifolds and symplectizations.

\begin{example}
Cotangent bundles are symplectic cobordisms as follows. Let $M=S^*L$
be the unit cotangent bundle of a Riemannian manifold $(L,g)$, with
the contact form $\lambda=p\,dq|_M$. We identify $L$ with the zero
section in $T^*L$. Then the map $\R\times M\to
T^*L\setminus L$, $(r,q,p)\mapsto(q,e^rp)$ yields a symplectomorphism
$$
        \bigl(\R\times M,d(e^r\lambda)\bigr)\cong (T^*L\setminus
        L,dp\,dq).
$$
This shows that the cotangent bundle $(T^*L,dp\,dq)$ is a symplectic
cobordism with one positive end modelled over $(M,\lambda)$.
\end{example}

{\bf Almost complex structures. }
An almost complex structure $J$ on a symplectization $\bigl(\R\times
M,d(e^r\lambda)\bigr)$ is called {\em compatible with $\lambda$} if it
is translation
invariant in the $\R$-direction, leaves the contact structure $\xi$
invariant, maps $\p/\p r$ to the Reeb vector field $R_\lambda$, and
$d(e^r\lambda)(\cdot,J\cdot)$ defines a Riemannian metric. An almost
complex structure on a symplectic cobordism $(X,\om)$ is called {\em
compatible} if it is compatible with the contact forms
$\ol{\lambda},\ul{\lambda}$ outside some compact set and
$\om(\cdot,J\cdot)$ defines a Riemannian metric. Compatible almost
complex structures exist and form a contractible space $\JJ$ (with the
$C^\infty$ topology).
\medskip

{\bf Punctured holomorphic curves. }
Consider a symplectization $\bigl(\R\times M,d(e^r\lambda)\bigr)$ with
a compatible almost complex structure $J$. Let $\gamma:[0,T]\to M$ be
a (not necessarily simple) closed orbit of the Reeb vector field
$R_\lambda$ of period $T$. A $J$-holomorphic map
$f=(a,u):D\setminus 0\to\R\times M$ of the punctured unit disk is
called {\em positively (resp.~negatively)} asymptotic to $\gamma$ if
$\lim_{\rho\to 0}a(\rho e^{i\theta})=\infty$ (resp.~$-\infty$) and
$\lim_{\rho\to 0}u(\rho e^{i\theta})=\gamma(T\theta/2\pi)$ (resp.~$\gamma(-T\theta/2\pi)$)
uniformly in $\theta$.

A {\em punctured holomorphic curve} in $(X,J)$ consists of the
following data:
\begin{itemize}
    \item A Riemann surface $(\Sigma,j)$ with distinct positive and
          negative points $\overline{z}:=
          (\overline{z}_1,...,\overline{z}_{\overline{p}})$,
          $\underline{z}:=
          (\underline{z}_1,...,\underline{z}_{\underline{p}})$.
          We denote by
          $\dot{\Sigma}:=\Sigma\setminus\{\overline{z}_i,\underline{z}_j\}$ the
          corresponding punctured Riemann surface.
    \item Corresponding vectors $\overline{\Gamma}:= (\overline{\gamma}_1,...,
          \overline{\gamma}_{\overline{p}}), \underline{\Gamma}:=
          (\underline{\gamma}_1,...,
          \underline{\gamma}_{\underline{p}})$ of
          closed Reeb orbits in $\overline{M},
          \underline{M}$.
    \item A $(j,J)$--holomorphic map $f:\dot{\Sigma}\to X$ which
          is positively (resp.~negatively) asymptotic to
          $\overline{\gamma}_j$ (resp.~$\underline{\gamma}_j$) at the
          punctures $\overline{z}_j$ (resp.~$\underline{z}_j$).
\end{itemize}
In a symplectization $\R\times M$, a {\em cylinder} over a $T$-periodic Reeb
orbit $\gamma$,
$$
        f:\R\times \R/\Z\to \R\times M,\qquad (s,t)\mapsto
\bigl(Ts,\gamma(Tt)\bigr),
$$
is a punctured holomorphic curve with one positive and one negative
puncture.

Denote by $\bar\Sigma$ the compactification of the punctured surface
$\dot\Sigma$ by adding a circle at each puncture, and by $\bar X$ the
compactification of the manifold $X$ by adding a copy of
$\ol{M}$ resp.~$\ul{M}$ at the positive resp.~negative end. In view of
the behaviour near the punctures, the holomorphic map $f:\dot\Sigma\to
X$ above extends to a continuous map $\bar f:\bar\Sigma\to\bar
X$. This extension represents a relative homology class
$$
   [\bar f]\in H_2(\bar X,\ol{\Gamma}\cup\ul{\Gamma}).
$$

\begin{example}
The cotangent bundle $T^*L$ of a Riemannian manifold $(L,g)$ carries a
canonical complex structure $J_g$ induced by the Riemannian metric. This
structure leaves the cylinders $\{\bigl(q(t),s\,\dot
q(t)\bigr)\;|\;s,t\in\R\}$ over closed geodesics $q$ invariant. Now
$J_g$ is not compatible according to the definition above. But it can
be deformed to a compatible almost complex structure $J$ which still
leaves all the cylinders over closed geodesics invariant. Then these
cylinders, appropriately parametrized, are punctured holomorphic
curves in $T^*L$ with two positive punctures.
\end{example}

{\bf Neck stretching. }
Consider a closed connected symplectic manifold $(X,\om)$ and a closed
hypersurface $M\subset X$ of {\em contact type}. This means that $M$
carries a contact form $\lambda$ such that $d\lambda=\om|_M$. Then a
neighbourhood of $M$ is symplectomorphic to $\bigl([-\eps,\eps]\times
M,d(e^r\lambda)\bigr)$. Assume that $X\setminus M = X^+_0\coprod
X^-_0$ consists of two components with $\{\pm\epsilon\}\times M\subset
X^\pm_0$. Pick a compatible almost complex structure $J$ on $(X,\om)$
whose restriction $J_M$ to $[-\eps,\eps]\times M$ is compatible with
$\lambda$.

Define a $1$-parameter family of symplectic manifolds as follows. For
a real number $k\geq 0$ let
$$
X_k := X_0^-\cup_M [-k,0]\times M \cup_M X_0^+.
$$
This manifold is, of course, canonically diffeomorphic to $X$, but
$X_k$ is a more convenient domain for describing the deformed
structures. Define an almost complex structure on $X_k$ by
$$
J_k:= \begin{cases}
          J & \mbox{ on } X^\pm_0, \\
          J_M & \mbox{ on } [-k,0]\times M.
      \end{cases}
$$
This almost complex structure is compatible with the symplectic form
on $X_k$ given by
$$
\omega_k:= \begin{cases}
               \omega & \mbox{ on } X^+_0, \\
               d(e^r\lambda) & \mbox{ on } [-k,0]\times M, \\
               e^{-k}\omega & \mbox{ on } X^-_0.
           \end{cases}
$$
In the limit $k\to\infty$ we obtain three symplectic cobordisms:
$X^+:=X_0^+\cup_M\R_-\times M$ (with one negative end),
$X^-:=X_0^-\cup_M\R_+\times M$ (with one positive end), and the
symplectization $\R\times M$. Here $\R\times M$ is equipped with the
symplectic form $d(e^r\lambda)$ and almost complex structure $J_M$,
and $X^\pm$ are equipped with the symplectic forms $\om^\pm$ and
almost complex structures $J^\pm$ satisfying $\om^\pm=\om$ on
$X_0^\pm$, $\om^\pm=d(e^r\lambda)$ on $\R_\pm\times M$, $J^\pm=J$ on
$X_0^\pm$, $J^\pm=J_M$ on $\R_\pm\times M$.

\begin{example}
We will apply the neck stretching construction in the following
situation. Let $L$ be a closed Lagrangian submanifold of
$(X,\om)$. Pick a Riemannian metric on $L$. After rescaling the metric
by a small constant, a tubular neighbourhood of $L$ is
symplectomorphic to  $W:=\{(q,p)\in T^*L\;\bigl|\;|p|\leq 1\}$ with
the canonical symplectic form $dp\,dq$. The boundary $M:=\p W$ is a
contact type hypersurface in $X$ with contact form
$\lambda:=p\,dq|_M$. It decomposes $X$ into two components $X_0^-={\rm
int}\,W$ and $X_0^+=X\setminus W$. So we can stretch the neck along
$M$ as described above. In this case, the limiting manifolds can be
identified as $(X^+,\om^+)\cong(X\setminus L,\om)$,
$(X^-,\om^-)\cong(T^*L,dp\,dq)$, and $\bigl(\R\times
M,d(e^r\lambda)\bigr)\cong(T^*L\setminus L,dp\,dq)$.
\end{example}

\begin{lemma}\label{lem:action}
Consider a Riemannian manifold $(L,g)$ and a symplectic embedding
$D^*L\into X$ of its unit cotangent bundle into a closed symplectic
manifold $(X,\om)$. Let $J$ be a compatible almost complex structure
on $X\setminus L$ which is cylindrical on $D^*L\setminus L$, and
$f:\C\to X\setminus L$ be a holomorphic plane asymptotic to the closed
Reeb orbit $(\gamma,\dot{\gamma})$ corresponding to a closed
geodesic $\gamma$. Then  the length of $\gamma$ satisfies
$$
   \ell(\gamma)\leq\int_f\om. 
$$
\end{lemma}

\begin{proof}
The map $(r,q,p)\mapsto(q,e^rp)$ defines a symplectomorphism
$(\R\times S^*L,d(e^r\lambda)\to(D^*L\setminus L,\om)$ with the
contact form $\lambda=p\,dq|_{S^*L}$. Denote by $\hat\gamma$ the lift
of the closed geodesic $\gamma$ to $S^*L$. Then positivity of $\om$ and
$d\lambda$ on $J$-holomorphic curves implies
$$
   \int_f\om \geq  \int_{f\cap D^*L}d(e^r\lambda) = \int_{f\cap
     S^*L}\lambda \geq \int_{\hat\gamma}\lambda = \ell(\gamma). 
$$
\end{proof}

{\bf Broken holomorphic curves. } We will now describe the
limiting objects of sequences of $J_k$-holomorphic curves in the
neck stretching procedure. We retain the setup of the preceding
section. For an integer $N\geq 2$ set
$$
   (X^{(\nu)},\om^{(\nu)},J^{(\nu)}) := \begin{cases}
   (X^-,\om^-,J^-) &\text{ for }\nu=1, \cr
   (\R\times M,d(e^r\lambda),J_M) &\text{ for }\nu=2,\dots,N-2, \cr 
   (X^+,\om^+,J^+) &\text{ for }\nu=N. \end{cases}
$$
Define the {\em split symplectic cobordism}
$$
   X^* := \coprod_{\nu=1}^N X^{(\nu)},
$$
equipped with the symplectic and almost complex structures $\om^*,J^*$
induced by the $\om^{(\nu)},J^{(\nu)}$. 
Glue the positive boundary component of the compactification
$\ol{X^{(\nu)}}$ (by copies of $M$) to the negative boundary component
of $\ol{X^{(\nu+1)}}$ to obtain a compact topological space
$$
   \bar X := \ol{X^{(1)}}\cup_M\dots\cup_M\ol{X^{(N)}}.
$$
Note that $\bar X$ is naturally homeomorphic to $X$ (see the
proof of Lemma~\ref{lem:homology} below for a particular
homeomorphism), so we can identify homology classes in $X$ and $\bar
X$.  

Let $\Sigma$ be a closed oriented surface and
$\Delta\subset\Sigma$ a collection of finitely many disjoint
simple loops. Write
$$
   \Sigma^* := \Sigma\setminus\Delta =:
   \coprod_{\nu=1}^N\Sigma^{(\nu)},
$$
as a disjoint union of (not necessarily connected) components
$\Sigma^{(\nu)}$. Let $j$ be a conformal structure on
$\Sigma^*$ such that $(\Sigma^*,j)$ is a punctured Riemann 
surface. A {\em broken holomorphic curve (with $N$ levels)}  
$$
   F=(F^{(1)},\dots,F^{(N)}):(\Sigma^*,j)\to(X^*,J^*)
$$ 
is a collection of punctured holomorphic curves
$F^{(\nu)}:(\Sigma^{(\nu)},j)\to(X^{(\nu)},J^{(\nu)})$ such that
$F:\Sigma^*\to X^*$ extends to a continuous map
$\bar F:\Sigma\to\bar X$. Moreover, we require that each
level is {\em stable}, i.e., $F^{(\nu)}$ contains a component
which is not a cylinder over a closed Reeb orbit. 

Note that, by continuity of $\bar F$, the number of positive punctures 
of $F^{(\nu)}$ agrees with the number of negative punctures of
$F^{(\nu+1)}$, and the asymptotic Reeb orbits at the punctures
agree correspondingly: $\overline{\Gamma}^{(\nu)} =
\underline{\Gamma}^{(\nu+1)}$.

The definition of a {\em stable broken holomorphic curve} only
differs from this as follows. All Riemann surfaces are allowed to
have nodes disjoint from the punctures, and the maps are
holomorphic outside the nodes and continuous at the nodes.

\begin{remark}
Every stable broken holomorphic curve has an underlying graph: Its
vertices are the connected components of $\Sigma^*$, and each
asymptotic Reeb orbit defines an edge between the corresponding
components. Note that if $\Sigma$ has genus zero the underlying graph
is a tree. 
\end{remark}

We have the following easy 

\begin{lemma}[\cite{CM-comp}]\label{lem:homology}
The homology class $A:=[\bar F]\in H_2(X;\Z)$ of a nonconstant stable
broken holomorphic curve $F:(\Sigma^*,j)\to(X^*,J^*)$ satisfies
$\om(A)>0$. 
\end{lemma}


{\bf Compactness. }
For $R\in\R$ denote by 
$$
   X_k\to X_0^-\cup_M [-k+R,R]\times M \cup_M X_0^+,\qquad x\mapsto
   x+R 
$$
the map which equals the identity on $X_0^-\cup X_0^+$ and is given by
$(r,x)\mapsto(r+R,x)$ on $[-k,0]\times M$. 
The following theorem collects the compactness properties that we need
in this paper. A more precise statement is proved in~\cite{CM-comp}.
%
\begin{thm}[Compactness]\label{thm:compact}
Let $(X_k,J_k)$ be as above and assume that 
$(M,\lambda)$ is Morse-Bott. 
Let $f_k:(\Sigma_k,j_k)\to (X_k,J_k)$
be a sequence of holomorphic curves of the same genus and in the
same homology class $[f_k]=A\in H_2(X;\Z)$. After passing to a
subsequence, there exist a stable broken holomorphic curve
$F:(\Sigma^*,j)\to(X^*,J^*)$, orientation preserving diffeomorphisms
$\phi_k:\Sigma_k\to\Sigma$, and numbers
$-k=r_k^{(1)}<r_k^{(2)}<\dots<r_k^{(N)}=0$
with $r_k^{(\nu+1)}-r_k^{(\nu)}\to\infty$
such that the following holds:
\begin{enumerate}
\item $(\phi_k)_*j_k\to j$ in $C^\infty_\loc$ on $\Sigma^*$. 

\item $f_k^{(\nu)}\circ\phi_k^{-1}\to F^{(\nu)}$ in
$C^\infty_\loc$ on $\Sigma^{(\nu)}$, where
$f_k^{(\nu)}$ is the shifted map $z\mapsto f_k(z)-r_k^{(\nu)}$. 

\item $\int_{\Sigma_k}f_k^*\om_k\to
\int_{\Sigma^{(N)}}(F^{(N)})^*\om^+$.

\item $[\bar F]=A\in H_2(X;\Z)$. 
\end{enumerate}
\end{thm}
Here on each component $\Sigma^{(\nu)}$ the convergence statements
(i-ii) have to be understood in the sense of nodal holomorphic curves,
see~\cite{CM-comp}.

\begin{cor}\label{cor:compact}
In the situation of Theorem~\ref{thm:compact}, assume in addition
that the genus is zero and the homology class $A$ cannot be written as
$A=B+C$ with $B,C\in H_2(X;\Z)$ satisfying $\om(B),\om(C)>0$. Then
$F$ is a broken holomorphic curve without nodes and the convergence
statements (i) and (ii) can be understood literally. 
\end{cor}

\begin{proof}
Suppose $F$ has a node. Since the genus is zero, the node decomposes
the domain $\Sigma$ into two connected components
$\Sigma_0,\Sigma_1$. The restrictions of $F$ to these components
define nonconstant stable broken holomorphic curves $F_0,F_1$
representing homology classes $A_0,A_1\in H_2(X;\Z)$ with
$A_0+A_1=A$. Lemma~\ref{lem:homology} yields $\om(A_0),\om(A_1)>0$,
contradicting the assumption on $A$.  
\end{proof}

{\bf Moduli spaces. }
Now we turn to moduli spaces of punctured holomorphic curves. From now
on we restrict to genus zero.
Consider a symplectic cobordism $(X,\om)$ of dimension $2n$ with
ends modeled over the contact manifolds $(\ol{M},\ol{\lambda})$ and
$(\ul{M},\ul{\lambda})$. Suppose first that 
all closed Reeb orbits on $\ol{M}$ and $\ul{M}$ are nondegenerate. 
Fix collections of closed
Reeb orbits
$$
        \ol{\Gamma}=(\ol{\gamma}_1,\dots,\ol{\gamma}_{\ol{p}}), \qquad
        \ul{\Gamma}=(\ul{\gamma}_1,\dots,\ul{\gamma}_{\ul{p}})
$$
on $\ol{M}$ and $\ul{M}$, respectively. Fix also a relative homology
class $A\in H_2(X,\ol{\Gamma}\cup\ul{\Gamma})$. Denote by
$\tilde\MM^{A,J}(\ol{\Gamma},\ul{\Gamma})$ the space of 
punctured holomorphic spheres asymptotic to $\ol{\Gamma},\ul{\Gamma}$
in the homology class $A$. More precisely, elements of
$\tilde\MM^{A,J}(\ol{\Gamma},\ul{\Gamma})$ are triples
$(f,\ol{z},\ul{z})$, where $\ol{z}=(\ol{z}_1,\dots,\ol{z}_{\ol{p}})$
and $\ul{z}=(\ul{z}_1,\dots,\ul{z}_{\ul{p}})$ are collections of
distinct points on the 2-sphere $S=S^2$, $\dot
S:=S^2\setminus\{\ol{z}_i,\ul{z}_j\}$, and $f:\dot S\to X$ is a
$J$-holomorphic map that is positively
asymptotic to the closed Reeb orbit $\ol{\gamma}_i$ at the puncture
$\ol{z}_i$, negatively asymptotic to $\ul{\gamma}_j$ at $\ul{z}_j$,
and represents the relative homology class $A$.

The space $\tilde\MM^{A,J}(\ol{\Gamma},\ul{\Gamma})$ is equipped with
the topology induced by weighted Sobolev norms as
in~\cite{BM},\cite{Bo}. Denote its quotient by the action of the
M\"obius group by
$$
        \MM^{A,J}(\ol{\Gamma},\ul{\Gamma}) :=
        \tilde\MM^{A,J}(\ol{\Gamma},\ul{\Gamma})\Bigl/\Aut(S).
$$
Its expected dimension is (see~\cite{BM})
\begin{align*}
        \dim \MM^{A,J}(\ol{\Gamma},\ul{\Gamma}) =
        (n-3)(2-\ol{p}-\ul{p}) + 2c_1(A) +
        \sum_{i=1}^{\ol{p}}\CZ(\ol{\gamma}_i)
        - \sum_{j=1}^{\ul{p}}\CZ(\ul{\gamma}_j).
\end{align*}
Here the Conley-Zehnder indices are computed with respect to some
fixed trivializations of the contact distribution $\xi$ over
$\ol{\gamma_i}$ and $\ul{\gamma_j}$. These trivializations induce
trivializations of the tangent bundle $TX$ over the asymptotic orbits,
and $c_1(A)$ is the relative first Chern class of $f^*TX$ with respect
to these trivializations for any representative $f$ of $A$.

If $\ol{M}$ and $\ul{M}$ are only Morse-Bott, then instead of
collections of closed Reeb orbits we fix collections
$$
        \ol{\Gamma}=(\ol{\Gamma}_1,\dots,\ol{\Gamma}_{\ol{p}}), \qquad
        \ul{\Gamma}=(\ul{\Gamma}_1,\dots,\ul{\Gamma}_{\ul{p}})
$$
of components of the orbifolds $\ol{\Gamma}_T$ and $\ul{\Gamma}_T$ of
$T$-periodic orbits (for varying $T$). The preceding discussion
carries over to this case, with the only difference that the
expected dimension becomes (see~\cite[Corollary 5.4]{Bo})
\begin{align*}
   \dim \MM^{A,J}(\ol{\Gamma},\ul{\Gamma}) 
   &= (n-3)(2-\ol{p}-\ul{p}) + 2c_1(A) \cr
   &\ \ \ +\sum_{i=1}^{\ol{p}}\bigl(\RS(\ol{\Gamma}_i)+\frac{1}{2}\dim\ol{\Gamma}_i\bigr)
   - \sum_{j=1}^{\ul{p}}\bigl(\RS(\ul{\Gamma}_j)-\frac{1}{2}\dim\ul{\Gamma}_j\bigr),
\end{align*}
or in view of equation~\eqref{eq:CZ-RS},
\begin{align}\label{eq:dim-Morse-Bott}
   \dim \MM^{A,J}(\ol{\Gamma},\ul{\Gamma}) 
   &= (n-3)(2-\ol{p}-\ul{p}) + 2c_1(A) \cr
   &\ \ \ + \sum_{i=1}^{\ol{p}}\bigl(\CZ(\ol{\Gamma}_i)+\dim\ol{\Gamma}_i\bigr)
        - \sum_{j=1}^{\ul{p}}\CZ(\ul{\Gamma}_j).
\end{align}

\section{Tangency conditions}\label{sec:tangent} 

In this section we introduce holomorphic curves satisfying tangency
conditions to a complex hypersurface, and we compute a specific such
invariant on $\C\P^n$. The discussion follows Section 6
in~\cite{CM-trans}. 

Consider a complex hypersurface in $Z=h^{-1}(0)\subset\C^n$ defined by
a holomorphic function $h:\C^n\to\C$ with $h(0)=0$ and $dh(0)\neq
0$. We say that a holomorphic map $f:\C\supset D\to\C^n$ with $f(0)=0$
is {\em tangent of order $\ell$ to $Z$ at $0$} if 
$$
   (h\circ f)'(0)=\dots=(h\circ f){(\ell)}(0)=0.
$$ 
This condition clearly only depends on the germs of $f$ and $Z$ at
$0$, and it is preserved under the action $(f,Z)\mapsto
\bigl(\Phi\circ f\circ\phi^{-1},\Phi(Z)\bigr)$ by local biholomorphisms of
$(\C,0)$ and $(\C^n,0)$. Note that, after applying a local
biholomorphism of $(\C^n,0)$, we may assume that $h(z_1,\dots,z_n)=z_n$
and the tangency condition becomes the condition
$f'(0),\dots,f^{(\ell)}(0)\in\C^{n-1}$ used in~\cite{CM-trans}. 

More generally, fix an integrable almost complex structure $J_0$ in a neighnorhood of
a point $x\in X$ and the germ of a complex hypersurface $Z$ through $x$.
Consider an almost complex structure  $J$ which near $x$  coincides with $J_0$
and a holomorphic map
$f:\Sigma\to X$ from a Riemann surface with $f(z)=x$. We say that $f$ 
is {\em tangent of order $\ell$ to $Z$ at $z$}, and write $d^\ell f(z)\in
T_xZ$, if $\Phi\circ f\circ\phi^{-1}$ is tangent of order $\ell$ to
$\Phi(Z)$ at $0$ for local holomorphic coordinates
$\phi:(\Sigma,z)\to(\C,0)$ and $\Phi:(X,x)\to(\C^n,0)$. 

Suppose now that $(X,J)$ has cylindrical ends and fix collections
$\ol{\Gamma},\ul{\Gamma}$ of closed Reeb orbits.  
Denote by $\tilde\MM^{A,J}(\ol{\Gamma},\ul{\Gamma};x,Z,\ell)$ the space of
punctured $J$-holomorphic spheres which are tangent of order $\ell$ to
$Z$ at $x$, i.e., 
\begin{align*}
        \tilde\MM^{A,J}(\ol{\Gamma},\ul{\Gamma};x,Z,\ell) := \{ &(f,\ol{z},\ul{z},z)
        \;|\;(f,\ol{z},\ul{z})\in
\tilde\MM^{A,J}(\ol{\Gamma},\ul{\Gamma}), z\in\dot S, \cr
        & f(z)=x, d^\ell f(z)\in T_xZ\}.
\end{align*}
Denote the quotient by the action of the M\"obius group by
$$
        \MM^{A,J}(\ol{\Gamma},\ul{\Gamma};x,Z,\ell) :=
\tilde\MM^{A,J}(\ol{\Gamma},\ul{\Gamma};x,Z,\ell)\Bigl/\Aut(S).
$$
The following result is a straighforward extension
of~\cite[Proposition 6.9]{CM-trans} to punctured holomorphic curves,
with a slight modification concerning transversality 
of the evaluation map at $z$ to the point $x$ rather than $Z$.

\begin{prop}\label{prop:reg}
For a generic almost complex structure $J$ as above the moduli space
$\MM_s^{A,J}(\ol{\Gamma},\ul{\Gamma};x,Z,\ell)\subset
\MM^{A,J}(\ol{\Gamma},\ul{\Gamma};x,Z,\ell)$ of {\em simple}
$J$-holomorphic spheres tangent of order $\ell$ to
$Z$ at $x$ is a manifold of dimension
\begin{align*}
        \dim \MM_s^{A,J}(\ol{\Gamma},\ul{\Gamma};x,Z,\ell) = &
        (n-3)(2-\ol{p}-\ul{p}) + 2c_1(A) +
        \sum_{i=1}^{\ol{p}}\CZ(\ol{\gamma}_i) \cr
        &- \sum_{j=1}^{\ul{p}}\CZ(\ul{\gamma}_j) - (2n-2) -
        2\ell.
\end{align*}
\end{prop}

{\bf Tangency conditions in cotangent bundles. }
We apply this proposition to 
to the cotangent bundle $T^*L$ of a manifold
$L^n$ which admits a metric of nonpositive curvature. For given $c>0$,
equip $L$ with the Riemannian metric provided by Lemma~\ref{lem:Morse}
such that all closed geodesics of length $\leq c$ are noncontractible
and nondegenerate and have Morse index $\leq n-1$. Let
$\Gamma=(\gamma_1,\dots,\gamma_k)$ be a
collection of (lifts of) closed geodesics of length
$\leq c$. Fix a point $x\in T^*L$ and the germ of a complex
hypersurface $Z$ through $x$. Consider the moduli
space $\MM^J(\Gamma;x,Z,\ell)$ of holomorphic spheres with $k$
positive punctures which are tangent of order $\ell$ to $Z$ at $x$. 
For a punctured holomorphic sphere
$f:\dot S\to T^*L$ trivialize the pullback bundle $f^*T^*L\to\dot S$
and denote by $\CZ(\gamma_i)$ the Conley-Zehnder indices in this
trivialization. Since the sum of the $\gamma_i$ is
null-homologous, the Maslov indices in this trivialization sum up to
zero. Hence by Lemma~\ref{lem:CZ}, we can replace the Conley-Zehnder
index by the Morse index in the dimension formula and obtain
\begin{align*}
   \dim \MM_s^J(\Gamma;x,Z,\ell) &= (n-3)(2-k) +
   \sum_{i=1}^k\CZ(\gamma_i) - (2n-2) - 2\ell \cr
   & \leq (n-3)(2-k) + k(n-1) - (2n-2) - 2\ell \cr 
   &= 2k-4-2\ell.
\end{align*}
Since by Proposition~\ref{prop:reg} this dimension is nonnegative for
generic $J$, we obtain

\begin{lemma}\label{lem:reg-torus}
For generic $J$
every {\em simple} $J$-holomorphic sphere in $T^*L$ which
is asymptotic at the punctures to geodesics of length $\leq c$ and
tangent of order $\ell$ to $Z$ at $x$ must have at least $\ell+2$
punctures.
\end{lemma}

Next we want to derive a version of this lemma for non-simple
spheres. So consider $\tilde f=f\circ\phi$ for a simple sphere $f:\dot
S\to T^*L$ as above and a $d$-fold branched covering $\phi:S^2\to
S^2$. Thus $f$ has $k$ positive punctures asymptotic to geodesics
$\gamma_1,\dots,\gamma_k$ and $\tilde f$ has $\tilde k\geq k$ positive
punctures asymptotic to multiples of the $\gamma_i$. We assume that
$\tilde f$ is tangent of order $n-1$ to $Z$ at a point $\tilde p\in
S^2$. In suitable holomorphic coordinates near $\tilde p$ and
$p=\phi(\tilde p)$ we have $\phi(z)=z^b$, where $b=\ord(\tilde p)$ is
the branching order of $\tilde p$. In holomorphic coordinates near
$x\in T^*L$ in which $Z=\C^{n-1}\subset\C^n$ we have
$f_n(z)=a_1z+a_2z^2+\dots$ and  
$$
   \tilde f_n(z)=a_1z^b+a_2z^{2b}+\dots = O(z^n),
$$ 
hence $a_1=\dots=a_\ell=0$ where $\ell$ is the smallest integer $\geq
n/b-1$. By the Riemann-Hurwitz formula we have
$$
   2d-2 = \sum_{z\in S^2}\bigl(\ord(z)-1\bigr) = \sum_{i=1}^{\tilde k}
     \bigl(\ord(z_i)-1\bigr) + \sum_{z\in \dot
       S}\bigl(\ord(z)-1\bigr).  
$$
Since $\phi$ maps all punctures $\tilde z_i$ of $\tilde f$ to
punctures of $f$, we have $\sum_{i=1}^{\tilde k}\ord(z_i) = kd$. The
sum over $z\in\dot S$ is estimated below by the contribution $b-1$
coming from $\tilde p$ and we obtain $2d-2\geq kd-\tilde
k+b-1$. Combining this with the estimates $k\geq \ell+2$ from
Lemma~\ref{lem:reg-torus}, $d\geq b$ and $\ell\geq n/b-1$ we find
$$
   \tilde k\geq (k-2)d+b+1 \geq (k-1)b+1 \geq (\ell+1)b+1 \geq n+1.
$$  
So we have shown

\begin{cor}\label{cor:reg-torus}
Let $L^n$ be a manifold which admits a metric of nonpositive
curvature. For $c>0$, equip $L$ with a metric such that all closed
geodesics of length $\leq c$ are noncontractible and nondegenerate and
have Morse index $\leq n-1$. Pick a point $x\in T^*L$ and the germ of
a complex hypersurface $Z$ through $x$. Then for generic $J$
every (not necessarily simple) $J$-holomorphic sphere in $T^*L$ which
is asymptotic at the punctures to geodesics of length $\leq c$ and
tangent of order $n-1$ to $Z$ at $x$ must have at least $n+1$
punctures.
\end{cor}

{\bf Tangency conditions in $\C\P^n$. }
The second situation we consider are non-punctured holomorphic
spheres in $\C\P^n$. Let $A=[\C\P^1]$ and fix a point $x\in\C\P^n$ and
the germ of a complex hypersurface $\Sigma$ through $x$. 
Since $c_1(A)=n+1$, the dimension of the moduli space in
Proposition~\ref{prop:reg} with $\ell=n-1$ is 
$$
        \dim \MM^{A,J}(x,Z,\ell) = 2n-6
        + 2n+2 - (2n -2) - (2n-2) = 0.
$$
(Note that all holomorphic spheres in class $A$ are simple). 
By Gromov compactness, the zero-dimensional manifold
$\MM^{A,J}(x,Z,n-1)$ is compact. By the usual cobordism argument, the
signed count of its points is independent of the regular almost complex
structure $J$, the point $x$, and the complex hypersurface $Z$. 

\begin{prop}\label{prop:reg-proj}
The signed count of points in $\MM^{A,J}(x,Z,n-1)$ equals
$(n-1)!$. 
\end{prop}

By Gromov compactness, this implies

\begin{cor}\label{cor:reg-proj}
For every point $x\in\C\P^n$, germ of complex hypersurface $\Sigma$
through $x$, and compatibe almost complex structure $J$, there
exists a $J$-holomorphic sphere in class $[\C\P^1]$ which is tangent
to order $n-1$ to $Z$ at $x$. 
\end{cor}

The proof of Proposition~\ref{prop:reg-proj} is based on two lemmas. 

\begin{lemma}\label{lem:lin}
For all positive real numbers $a_1,\dots,a_n$ the system of $n$
equations for $n$ complex variables $z_1,\dots,z_n$ 
\begin{align}\label{eq:z}
   & a_1z_1+\dots+a_nz_n=0 \cr
   & a_1z_1^2+\dots+a_nz_n^2=0 \cr
   & \dots \cr
   & a_1z_1^n+\dots+a_nz_n^n=0
\end{align} 
has only the trivial solution $z_1=\dots=z_n=0$. 
\end{lemma}

\begin{proof}
We prove the lemma by induction on $n$. The case $n=1$ is
clear. Suppose the lemma holds for $n-1$ but not for $n$, so
equation~\eqref{eq:z} has a nontrivial solution
$z=(z_1,\dots,z_n)$. Writing~\eqref{eq:z} as
\begin{align*}
\left(\begin{matrix}   
   a_1 & \dots & a_n \\
   a_1z_1 & \dots & a_nz_n \\
   \dots & \dots & \dots \\
   a_1z_1^{n-1} & \dots & a_nz_n^{n-1} 
\end{matrix}\right)
\left(\begin{matrix}   
   z_1 \\
   z_2 \\
   \dots \\
   z_n 
\end{matrix}\right)
= 0,
\end{align*} 
this implies
$$
\det
\left(\begin{matrix}   
   1 & \dots & 1 \\
   z_1 & \dots & z_n \\
   \dots & \dots & \dots \\
   z_1^{n-1} & \dots & z_n^{n-1} 
\end{matrix}\right)
= \prod_{i<j}(z_j-z_i) = 0. 
$$
So after reordering the $z_i$ we may assume $z_{n-1}=z_n$. Inserting
this in~\eqref{eq:z} and deleting the last equation, we obtain 
\begin{align*}
   & a_1z_1+\dots+(a_{n-1}+a_n)z_{n-1}=0 \cr
   & a_1z_1^2+\dots+(a_{n-1}+a_n)z_{n-1}^2=0 \cr
   & \dots \cr
   & a_1z_1^{n-1}+\dots+(a_{n-1}+a_n)z_{n-1}^{n-1}=0. 
\end{align*}
By hypothesis this has only the trivial solution
$0=z_1=\dots=z_{n-1}=z_n$, contraticting the assumption $z\neq 0$. 
\end{proof}

\begin{lemma}\label{lem:proj}
For all positive real numbers $a_0,\dots,a_n$ the $n$ complex hypersurfaces
in $\C\P^n$ defined by the homogeneous equations  
\begin{align}\label{eq:pn}
   & a_0z_0+\dots+a_nz_n=0 \cr
   & a_0z_0^2+\dots+a_nz_n^2=0 \cr
   & \dots \cr
   & a_0z_0^n+\dots+a_nz_n^n=0
\end{align} 
intersect transversally in $n!$ points. 
\end{lemma}

\begin{remark}
It follows that the hypersurfaces defined by~\eqref{eq:pn} intersect
transversally for all complex numbers $a_0,\dots,a_n$ outside an
algebraic subvariety in $\C^{n+1}$.  
\end{remark}

\begin{proof}
Suppose $[z_0:\dots:z_n]$ is a non-transverse intersection point. Then
the matrix of linearized equations
\begin{align*}
\left(\begin{matrix}   
   a_0 & \dots & a_n \\
   a_0z_0 & \dots & a_nz_n \\
   \dots & \dots & \dots \\
   a_0z_0^{n-1} & \dots & a_nz_n^{n-1} 
\end{matrix}\right)
\end{align*}
has rank $<n$, which implies
$$
\det
\left(\begin{matrix}   
   1 & \dots & 1 \\
   z_0 & \dots & z_{n-1} \\
   \dots & \dots & \dots \\
   z_0^{n-1} & \dots & z_{n-1}^{n-1} 
\end{matrix}\right)
= \prod_{i<j}(z_j-z_i) = 0. 
$$
So after reordering the $z_i$ we may assume $z_0=z_1$. Inserting
this in~\eqref{eq:pn}, we obtain 
\begin{align*}
   & (a_0+a_1)z_1+\dots+a_nz_n=0 \cr
   & (a_0+a_1)z_1^2+\dots+a_nz_n^2=0 \cr
   & \dots \cr
   & (a_0+a_1)z_0^n+\dots+a_nz_n^n=0. 
\end{align*} 
By Lemma~\ref{lem:lin} this system has only the trivial solution
$z_0=z_1=\dots=z_n=0$, contradicting the assumption $z\neq 0$. 

By Bezout's theorem, the number of intersection points equals the
product $1\cdot 2\cdots n$ of the degrees of the equations
in~\eqref{eq:pn}. 
\end{proof}

\begin{proof}[Proof of Proposition~\ref{prop:reg-proj}]
Consider the standard complex structure on $\C\P^n$. Denote by
$\tilde\MM$ the space of holomorphic maps
$f:\C\P^1\to\C\P^n$ of degree $1$ mapping $[1:0]$ to $[1:0:\dots:0]$
and $[0:1]$ to the hyperplane $\{z_0=0\}$ at infinity. Each such map
is of the form
$$
   f_p([z_0:z_1]) = [z_0:p_1z_1:\dots:p_nz_1]\in\C\P^n,
$$
or in affine coordinates
$$
   f_p(z) = (p_1z,\dots,p_nz)\in\C^n,
$$
with $p=(p_1,\dots,p_n)\in\C^n\setminus 0$. The correspondence
$p\mapsto f_p$ thus gives a diffeomorphism from $\C^n\setminus 0$ to
$\tilde\MM$. Reparametrizing $f$ as
$f(\lambda z)$ for $\lambda\in\C^*$ corresponds to replacing $p$ by
$\lambda p$, so the correspondence $p\mapsto f_p$ induces a
diffeomorphism from $\C\P^{n-1}$ to the moduli space
$\MM=\tilde\MM/\C^*$ of degree $1$ holomorphic spheres passing through
$0$ in $\C\P^n$. It is well known (see~\cite{MS}) that this moduli
space is transversely cut out and its orientation agrees with the
complex orientation of $\C\P^{n-1}$. 

Let $Z\subset\C\P^n$ be a complex hypersurface of degree $n-1$ defined
in affine coordinates $z=(z_1,\dots,z_n)$ by a homogeneous equation
$$
   h(z) = h_1(z)+\dots+h_{n-1}(z) = 0,
$$
where $h_k(z)=a_1z_1^k+\dots+a_nz_n^k$ with positive real numbers
$a_0,\dots,a_n$. Note that $h_k\circ f_p(z)=z^ph_k(p)$ and thus
$(h\circ f_p)^{(k)}(0)=k!h_k(p)$. So the jet evaluation map at $0$
$$
   j:\tilde\MM\to\C^{n-1},\qquad f_p\mapsto\Bigl((h\circ
   f_p)'(0),\dots,(h\circ f_p)^{(n-1)}(0)\Bigr)
$$ 
corresponds to the map
$$
   \C^n\setminus 0\to\C^{n-1},\qquad p\mapsto\Bigl(h_1(p),\dots,
   (n-1)!h_{n-1}(p)\Bigr).
$$ 
By Lemma~\ref{lem:proj} this map is transverse to $0\in\C^{n-1}$ and
its zero set in $\MM\cong\C\P^{n-1}$ consists of $(n-1)!$
points. Since all spaces are equipped with their complex orientations,
all these points count positively and yield the signed count $(n-1)!$
of points in the space $\MM^{A,J}(x,Z,n-1)$ in
Proposition~\ref{prop:reg-proj}.  
\end{proof}

\section{Proofs requiring only standard transversality}\label{sec:proofs}

In this section we prove all theorems from the Introduction except 
Theorem~\ref{thm:Audin} (b) and Theorem~\ref{thm:uniruled-dim2}. 

\begin{proof}[Proof of Theorem~\ref{thm:disk}]
Equip $L$ with a Riemannian metric $g_0$ of nonpositive curvature such
that the set $\{(q,p)\in T^*L\;\bigl|\;|p|\leq
2\}$ (with the symplectic form $dp\,dq$) embeds symplectically into
$\C\P^n$. Such a metric exists by the Lagrangian neighbourhood theorem
and rescaling the metric. Perturb $g_0$ to a metric $g$ as in
Lemma~\ref{lem:Morse} such that all closed $g$-geodesics of length
$\leq \pi$ are noncontractible and nondegenerate and have Morse index
$\leq n-1$. Clearly we can achieve that the unit ball cotangent bundle
$W:=\{(q,p)\in T^*L\;\bigl|\;|p|\leq 1\}$ with respect to $g$ still embeds
symplectically into $\C\P^n$. Denote by $M:=\p W$ the unit cotangent
bundle of $L$. We identify $W$ and $M$ with their images in $\C
P^n$. Pick a compatible almost complex structure $J$ on $\C\P^n$ 
with $J=J_M$ near $M$. Let $(J_k)_{k\in\N}$ be the sequence of almost
complex structures on $\C\P^n$ obtained by the neck stretching
procedure described before Theorem~\ref{thm:compact}.

\begin{figure}[ht]\label{fig:stretching1}
\begin{picture}(10,90)
\put(30,10){\qbezier(0,40)(0,0)(40,0)}
\put(70,10){\qbezier(0,0)(40,0)(40,40)}
\put(70,50){\qbezier(0,40)(40,40)(40,0)}
\put(30,50){\qbezier(0,0)(0,40)(40,40)}
\put(31,40){\qbezier(0,0)(12,10)(39,10)}
\put(70,40){\qbezier(0,10)(27,10)(39,0)}
\thicklines
\put(50,20){\qbezier(0,40)(10,0)(20,0)}
\put(50,60){\qbezier(0,0)(0,16)(20,16)}
\put(90,20){\qbezier(0,40)(-10,0)(-20,0)}
\put(90,60){\qbezier(0,0)(0,16)(-20,16)}
\thinlines
\put(50,60){\qbezier(0,0)(0,-4)(20,-4)}
\put(50,60){\qbezier(0,0)(0,4)(20,4)}
\put(90,60){\qbezier(0,0)(0,-4)(-20,-4)}
\put(90,60){\qbezier(0,0)(0,4)(-20,4)}
\put(71,18){\line(1,1){12}}
\put(75.5,22.5){\circle*{1.5}}
\put(67,40.5){\qbezier(0,0)(5,-22)(18,-22)}
\thicklines
\put(100,82){$\C\P^n$}
\put(91,60){$f_k$}
\put(84,28){$Z$}
\put(77,22){$x$}
\put(110,37){$M$}
\put(69,37){$L$}
\put(40,30){$W$}
\end{picture}
\caption{A $J_k$--holomorphic sphere through $x\in L$ and tangent of order $(n\!-\!1)$ to  $Z$.}
\end{figure}
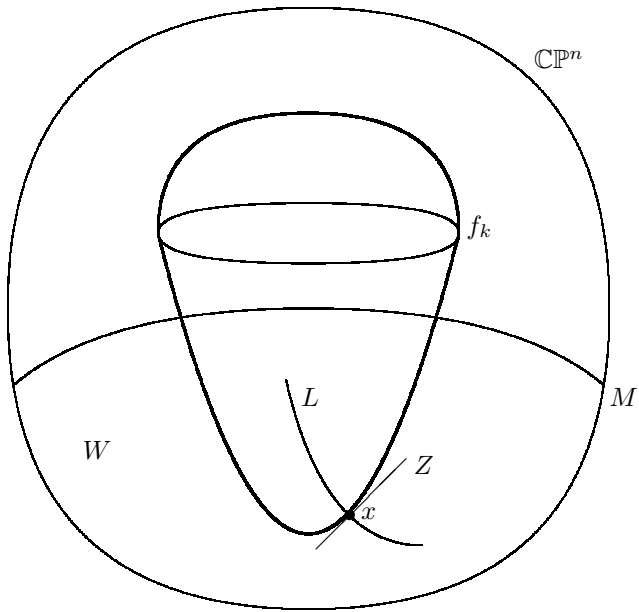

Fix a point $x$ on $L$ and the germ of a complex hypersurface
$Z\subset T^*L$ through $x$. Choose $J$ such that the conclusion of
Corollary~\ref{cor:reg-torus} holds (with $c=\pi$). 
By Corollary~\ref{cor:reg-proj} there exists for each $k$
a $J_k$-holomorphic sphere $f_k:S^2\to\C\P^n$ in the class $[\C\P^1]$
which is tangent of order $n-1$ to $Z$ at $x$. 
See Figure~2. 


Let $z_k\in S^2$ be such that
$$
        f_k(z_k)=x,\qquad d^{n-1}f_k(z_k)\in T_xZ.
$$
By Corollary~\ref{cor:compact}, after passing to a
subsequence, $f_k$ converges to an $N$-level broken holomorphic curve
$F=(F^{(1)},\dots,F^{(N)}):(\Sigma^*,j)\to(X^*,J^*)$. In
particular, there exist diffeomorphisms $\phi_k:S^2\to\Sigma$
such that $f_k^{(\nu)}\circ\phi_k^{-1}$ converges in $C^\infty_{\rm
  loc}$ on $\Sigma^{(\nu)}$ to $F^{(\nu)}:\Sigma^{(\nu)}\to 
X^{(\nu)}$. Recall that $X^{(1)}=T^*L$, $X^{(N)}=\C
P^n\setminus L$, and $X^{(\nu)}=\R\times M$ for $\nu=2,\dots,N-1$.

Note that $x\in X_0^-\subset X^{(1)}$, and therefore
$y_k:=\phi_k(z_k)\in\Sigma^{(1)}$ with
$f_k^{(1)}\circ\phi_k^{-1}(y_k)=x$ for all $k$. By the asymptotics of
$F^{(1)}$ and the $C^\infty_\loc$-convergence $f_k^{(1)}\to F^{(1)}$,
after passing to a subsequence we obtain $y_k\to y\in\Sigma^{(1)}$
with $F^{(1)}(y)=x$ and $d^{n-1}F^{(1)}(y)\in T_xZ$.
Denote by $C$ the component of $F^{(1)}$ containing $y$
(see Figure~3).

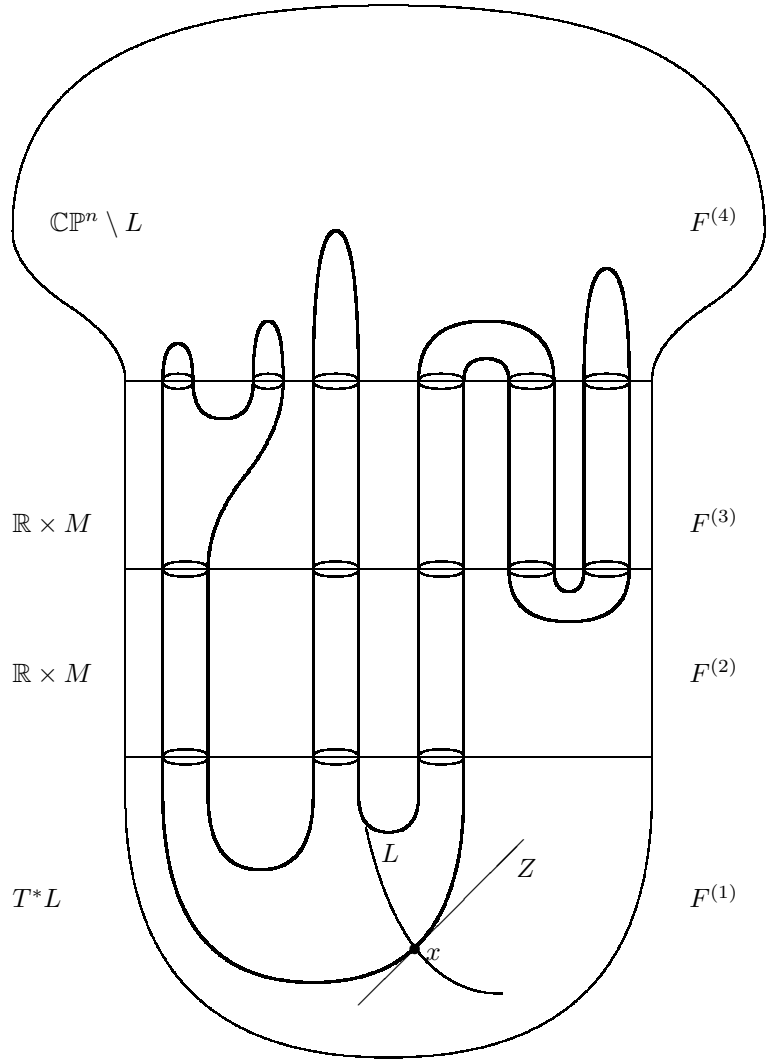
\begin{figure}[ht]
\begin{picture}(90,120)\label{fig:stretching2}
\put(70,10){\qbezier(0,0)(-35,0)(-35,35)\qbezier(0,0)(35,0)(35,35)}
\put(35,45){\line(0,1){55}}
\put(105,45){\line(0,1){55}}
\put(35,50){\line(1,0){70}}
\put(35,75){\line(1,0){70}}
\put(35,100){\line(1,0){70}}
\put(27.5,110){\qbezier(0,0)(7.5,-5)(7.5,-10)\qbezier(0,0)(-7.5,5)(-7.5,10)}
\put(112.5,110){\qbezier(0,0)(-7.5,-5)(-7.5,-10)\qbezier(0,0)(7.5,5)(7.5,10)}
\put(70,150){\qbezier(0,0)(-50,0)(-50,-30)\qbezier(0,0)(50,0)(50,-30)}
\thicklines
\put(60,20){\qbezier(0,0)(-20,0)(-20,25)\qbezier(0,0)(20,0)(20,25)}
\put(40,45){\line(0,1){55}}
\put(46,45){\line(0,1){30}}
\put(60,45){\line(0,1){55}}
\put(66,45){\line(0,1){55}}
\put(74,45){\line(0,1){55}}
\put(80,45){\line(0,1){55}}
\put(86,75){\line(0,1){25}}
\put(92,75){\line(0,1){25}}
\put(96,75){\line(0,1){25}}
\put(102,75){\line(0,1){25}}
\put(51,87.5){\qbezier(0,0)(-5,-6.25)(-5,-12.5)\qbezier(0,0)(5,6.25)(5,12.5)}
\put(53,35){\qbezier(0,0)(-7,0)(-7,10)\qbezier(0,0)(7,0)(7,10)}
\put(70,40){\qbezier(0,0)(-4,0)(-4,5)\qbezier(0,0)(4,0)(4,5)}
\put(48,95){\qbezier(0,0)(-4,0)(-4,5)\qbezier(0,0)(4,0)(4,5)}
\put(94,72){\qbezier(0,0)(-2,0)(-2,3)\qbezier(0,0)(2,0)(2,3)}
\put(94,68){\qbezier(0,0)(-8,0)(-8,7)\qbezier(0,0)(8,0)(8,7)}
\put(83,103){\qbezier(0,0)(-3,0)(-3,-3)\qbezier(0,0)(3,0)(3,-3)}
\put(83,108){\qbezier(0,0)(-9,0)(-9,-8)\qbezier(0,0)(9,0)(9,-8)}
\put(99,115){\qbezier(0,0)(-3,0)(-3,-15)\qbezier(0,0)(3,0)(3,-15)}
\put(42,105){\qbezier(0,0)(-2,0)(-2,-5)\qbezier(0,0)(2,0)(2,-5)}
\put(54,108){\qbezier(0,0)(-2,0)(-2,-8)\qbezier(0,0)(2,0)(2,-8)}
\put(63,120){\qbezier(0,0)(-3,0)(-3,-20)\qbezier(0,0)(3,0)(3,-20)}

\thinlines
\put(66,17){\line(1,1){22}}
\put(73.5,24.5){\circle*{1.5}}
\put(67,40.5){\qbezier(0,0)(5,-22)(18,-22)}
\put(87,34){$Z$}
\put(75,23){$x$}
\put(69,36){$L$}
\put(43,50){\qbezier(-3,0)(-3,1)(0,1)\qbezier(3,0)(3,1)(0,1)
            \qbezier(-3,0)(-3,-1)(0,-1)\qbezier(3,0)(3,-1)(0,-1)}
\put(43,75){\qbezier(-3,0)(-3,1)(0,1)\qbezier(3,0)(3,1)(0,1)
            \qbezier(-3,0)(-3,-1)(0,-1)\qbezier(3,0)(3,-1)(0,-1)}
\put(63,50){\qbezier(-3,0)(-3,1)(0,1)\qbezier(3,0)(3,1)(0,1)
            \qbezier(-3,0)(-3,-1)(0,-1)\qbezier(3,0)(3,-1)(0,-1)}
\put(63,75){\qbezier(-3,0)(-3,1)(0,1)\qbezier(3,0)(3,1)(0,1)
            \qbezier(-3,0)(-3,-1)(0,-1)\qbezier(3,0)(3,-1)(0,-1)}
\put(63,100){\qbezier(-3,0)(-3,1)(0,1)\qbezier(3,0)(3,1)(0,1)
            \qbezier(-3,0)(-3,-1)(0,-1)\qbezier(3,0)(3,-1)(0,-1)}
\put(77,50){\qbezier(-3,0)(-3,1)(0,1)\qbezier(3,0)(3,1)(0,1)
            \qbezier(-3,0)(-3,-1)(0,-1)\qbezier(3,0)(3,-1)(0,-1)}
\put(77,75){\qbezier(-3,0)(-3,1)(0,1)\qbezier(3,0)(3,1)(0,1)
            \qbezier(-3,0)(-3,-1)(0,-1)\qbezier(3,0)(3,-1)(0,-1)}
\put(77,100){\qbezier(-3,0)(-3,1)(0,1)\qbezier(3,0)(3,1)(0,1)
            \qbezier(-3,0)(-3,-1)(0,-1)\qbezier(3,0)(3,-1)(0,-1)}
\put(89,75){\qbezier(-3,0)(-3,1)(0,1)\qbezier(3,0)(3,1)(0,1)
            \qbezier(-3,0)(-3,-1)(0,-1)\qbezier(3,0)(3,-1)(0,-1)}
\put(89,100){\qbezier(-3,0)(-3,1)(0,1)\qbezier(3,0)(3,1)(0,1)
            \qbezier(-3,0)(-3,-1)(0,-1)\qbezier(3,0)(3,-1)(0,-1)}
\put(99,75){\qbezier(-3,0)(-3,1)(0,1)\qbezier(3,0)(3,1)(0,1)
            \qbezier(-3,0)(-3,-1)(0,-1)\qbezier(3,0)(3,-1)(0,-1)}
\put(99,100){\qbezier(-3,0)(-3,1)(0,1)\qbezier(3,0)(3,1)(0,1)
            \qbezier(-3,0)(-3,-1)(0,-1)\qbezier(3,0)(3,-1)(0,-1)}
\put(42,100){\qbezier(-2,0)(-2,1)(0,1)\qbezier(2,0)(2,1)(0,1)
            \qbezier(-2,0)(-2,-1)(0,-1)\qbezier(2,0)(2,-1)(0,-1)}
\put(54,100){\qbezier(-2,0)(-2,1)(0,1)\qbezier(2,0)(2,1)(0,1)
            \qbezier(-2,0)(-2,-1)(0,-1)\qbezier(2,0)(2,-1)(0,-1)}
\put(25,120){$\C\P^n\setminus L$}
\put(110,120){$F^{(4)}$}
\put(20,80){$\R\times M$}
\put(110,80){$F^{(3)}$}	    	    	    	    	    	    	    	    	    	    	    	    
\put(20,60){$\R\times M$}
\put(110,60){$F^{(2)}$}
\put(20,30){$T^\ast L$}
\put(110,30){$F^{(1)}$}
\end{picture}
\caption{A broken $J_\infty$-holomorphic sphere through
$x\in L$ and tangent of order $n\!-\!1$ to  $Z$.}
\end{figure}

Since $\Sigma$ has genus zero, the graph underlying $F$ is
a tree (see the remark preceding Lemma~\ref{lem:homology}). 
Let us replace each subtree emanating from the node $C$ by {\em one}
node which corresponds to a broken holomorphic curve with components
in $\C\P^n\setminus L$, $T^{\ast}L$ and $M\times\R$ and formally one
negative puncture. (All other asymptotics appear in pairs of a negative
and a positive puncture asymptotic to the same closed Reeb orbit). 
So we have a component $C$ in $T^*L$ with $m$ positive punctures
asymptotic to geodesics $\gamma_1,\dots,\gamma_m$ and $m$ (broken)
holomorphic planes $C_i$, $i=1,\dots,m$, asymptotic to
$\gamma_i$ at their negative punctures. Since the $\gamma_i$ are not
contractible in $L$, each $C_i$ must have a component in $\C
P^n\setminus L$ and hence satisfies $\int_{C_i}\om>0$. 

By Lemma~\ref{lem:action} (applied to the broken holomorphic planes
$C_i$), all the geodesics $\gamma_i$ have length $\leq\pi$.
So by Corollary~\ref{cor:reg-torus}, the number of punctures of $C_0$
satisfies $m\geq n+1$. So the $C_i$ give rise to $m\geq n+1$
nonconstant disks $g_i:D\to\C\P^n$ with boundary on $L$. Since
$C_i$ is holomorphic, $g_i^*\om\geq 0$ and $\int_D g_i^*\om>0$. Since
$\sum_{i=1}^m\int_D g_i^*\om=\pi$, the $\om$-area of one of the
disks must be $\leq\pi/(n+1)$.
\end{proof}

\begin{proof}[Proof of Theorem~\ref{thm:Audin} (a)]
We retain the notation from the preceding proof. 
Assume that $L$ is monotone. Then, since $C_i$ has
positive symplectic area, it must have positive Maslov index
$\mu(C_i)\geq 1$. On the other hand, the sum of the Maslov indices
equals the value of the first Chern class of $\C\P^n$ on $[\C\P^1]$, 
$$
   \sum_{i=1}^m\mu(C_i) = 2n+2. 
$$
Since $m\geq n+1$, this implies that some $\mu(C_i)\leq 2$. If in
addition $L$ is orientable, then all Maslov indices are even, hence
$m=n+1$ and $\mu(C_i)=2$ for all $i$. 
\end{proof}

\begin{proof}[Proof of Theorem~\ref{thm:ball}]
We keep the notation of the proof of Theorem~\ref{thm:disk}.
Let $W$ be a tubular neighbourhood of $L$ as above with $W\cap
B=\emptyset$ and let $M=\p W$. Let $J$ and the sequence $J_k$ be as
above. Moreover, choose $J$ to be standard
on the ball $B$. Let $f_k$ be $J_k$-holomorphic spheres in the class
of a complex line passing through the center $p$ of $B$ and a point
$x$ on $L$. In the limit $k\to\infty$ we find an $N$-level broken
holomorphic curve $F=(F^{(1)},\dots,F^{(N)})$ passing through $x$ and
$p$. Since $L$ has no contractible geodesics, the component of
$F^{(1)}$ passing through $x$ must have at least 2 positive punctures.
Hence $F^{(N)}$ has at least two components $C_1,C_2$. One of them,
say $C_1$, passes through the point $p$. By the classical
isoperimetric inequality (see~\cite{Gr}), $C_1$ has symplectic area at
least $\pi r^2$. Therefore we have $0<\int_{C_2}\om\leq \pi-\pi r^2$.
\end{proof}

\begin{proof}[Proof of Theorem~\ref{thm:uniruled}]
Suppose $L$ admits a Lagrangian embedding into a uniruled symplectic
manifold $(X,\om)$. Pick a tubular neighbourhood of $L$ and almost
complex structures $J$, $J_k$ as in the
proof of Theorem~\ref{thm:disk}. Let $f_k$ be $J_k$-holomorphic spheres passing
through a point $x$ on $L$ (they exist because $X$ is uniruled). In
the limit $k\to\infty$ we find an $N$-level broken holomorphic curve
$F=(F^{(1)},\dots,F^{(N)})$. Consider the simple curve underlying the
component of $F^{(1)}$ which passes through $x$. It belongs to some
moduli space of $J^-$-holomorphic
spheres in $T^*L$ with $m$ punctures passing through $x$. Since $L$
has no contractible geodesics, we must have $m\geq 2$. Since all
geodesics have Morse index zero, the moduli space is (if $J^-$ is
regular for simple curves passing through $x$) a manifold of dimension
$(n-3)(2-m)-(2n-2)$. If $n\geq 3$ 
this dimension is negative and we have a contradiction.
\end{proof}

\begin{proof}[Proof of Theorem~\ref{thm:exact}]
Consider a Lagrangian embedding of $L$ into a minimally uniruled
symplectic manifold $(X,\om)$. Construct an $N$-level broken
holomorphic curve $F=(F^{(1)},\dots,F^{(N)})$ of total area $a$
passing through a point $x\in L$ as in the proof of
Theorem~\ref{thm:uniruled}. Since $L$ has no contractible geodesics,
the component of $F^{(1)}$ passing through $x$ has at least 2 positive
punctures. Hence $F^{(N)}$ has at least two components
$C_1,C_2$. Since the total symplectic area of $F$ equals $a$, both
components must have area $0<\int_{C_i}\om< a$. So
the Lagrangian embedding is not exact.
\end{proof}

\begin{proof}[Proof of Theorem~\ref{thm:emb-cap}]
Consider $\alpha>0$ such that the unit codisk bundle $D^*(Q,g)$ embeds
symplectically into $(\C\P^n,\alpha\om)$. Pick an almost complex
structure $J$ as in Lemma~\ref{lem:action}. The proof of
Theorem~\ref{thm:disk} yields a $J$-holomorphic plane
$f:\C\to\C\P^n\setminus L$ asymptotic to a closed geodesic $\gamma$ 
whose symplectic area satisfies $\int_f\om\leq \frac{\alpha\pi}{n+1}$.
Since $\ell(\gamma)\leq\int_f\om$ by Lemma~\ref{lem:action}, we
conclude 
$$
   \alpha \geq \frac{(n+1)\ell(\gamma)}{\pi} \geq
   \frac{(n+1)\ell_\min(Q,g)}{\pi}. 
$$
\end{proof}

\section{Proofs requiring full transversality}\label{sec:proofs-trans} 

The remainder of the paper is devoted to the proof of
Theorem~\ref{thm:Audin} (b) and Theorem~\ref{thm:uniruled-dim2} 
which require stronger transversality. Let us first explain their
proofs assuming the necessary transversality. 

\begin{proof}[Proof of Theorem~\ref{thm:Audin} (b)]
Suppose that $L\subset\C\P^n$ is a Lagrangian torus. We equip $L$ with
the standard flat metric, rescaled such that the unit ball cotangent
bundle $W:=\{(q,p)\in T^*L\;\bigl|\;|p|\leq 1\}$ embeds symplectically
into $\C\P^n$. Note that for this metric all closed geodesics occur in
$(n-1)$-dimensional Morse-Bott families of Morse index zero.

Construct $C,C_1,\dots,C_m$ as in the proof of
Theorem~\ref{thm:disk}. Thus $C$ is a holomorphic sphere in $T^*L$
with $m$ positive punctures asymptotic to families of geodesics
$\Gamma_1,\dots,\Gamma_m$ and tangent of order $n-1$ to $Z$ at $x\in L$, and
$C_i$ is a broken holomorphic plane negatively asymptotic to $\Gamma_i$. 
Assume that all components of $C_i$ are regular and $C$ is regular
subject to the tangency condition to $Z$, and the evaluation 
maps at adjacent punctures are transverse to each other. It is shown
in Section~\ref{sec:trans} that $C_i$ and $C$ belong to
moduli spaces $\MM_i$ resp.~$\MM$ of dimensions  
\begin{align*}
   \dim\MM_{i} &= (n-3) - \CZ(\Gamma_i),\qquad i=1,\dots,m, \cr
   \dim\MM &= (n-3)(2-m) + (2n+2) + \sum_{\i=1}^{m}\bigl(\CZ(\Gamma_{i}) + \dim
   \Gamma_i\bigr) - (4n-4). 
\end{align*}
Here the Conley-Zehnder indices are defined using trivializations
over $C_i$ resp.~$C$. The $(2n+2)$ in the last line corresponds to the
value of the first Chern class of $\C\P^n$ on a complex line. The
$-(4n-4)$ arises from $2n-2$ conditions for passing through $x$ and 
$2n-2$ conditions for order $n-1$ tangency to $Z$. Note that the
dimension formulae agree with the ones resulting from
equation~\eqref{eq:dim-Morse-Bott} for non-broken curves. 

Denote by $\mu_i$ the Maslov index of $C_i$. Using $\dim \Gamma_i=n-1$ as
well as $\sum_{i=1}^m\mu_i=2n+2$ and the relation 
$$
   \CZ(\Gamma_i) = \ind(\Gamma_i)-\mu_i = -\mu_i
$$
from Lemma~\ref{lem:CZ} we obtain 
\begin{align*}
   \dim\MM_{i} &= (n-3) + \mu_i,\qquad i=1,\dots,m, \cr
   \dim\MM &= (n-3)(2-m) + \sum_{\i=1}^{m} \dim \Gamma_i - (4n-4) \cr
   &= (n-3)(2-m) + m(n-1) - (4n-4) \cr
   &= 2m-2n-2. 
\end{align*}
Note that nonnegativity of $\dim\MM$ yields $m \geq n+1$. 
According to our regularity assumptions, the evaluation maps at the
punctures 
$$
   \ev_1\times\dots\times\ev_m:\MM_1\times\dots\times\MM_m\to
   \Gamma:=\Gamma_1\times\cdots\times \Gamma_m, \qquad 
   \ev:\MM\to \Gamma
$$
are transverse to each other. Hence the images of the linearized
evaluation maps at any point must satisfy
$$
   \sum_{i=1}^m\dim\im(T\ev_i) + \dim\im(T\ev) \geq \dim \Gamma = m(n-1). 
$$
Inserting
\begin{align*}
   \dim\im(T\ev) &\leq \dim\MM = 2m-2n-2, \cr
   \dim\im(T\ev_i) &\leq \min\{\dim \Gamma_i,\dim\MM_i\} =
   \min\{n-1,n-3+\mu_i\} \cr
   &= \frac{1}{2}(2n-4+\mu_i-|2-\mu_i|)
\end{align*}
as well as $\sum\mu_i=2n+2$ we infer
\begin{align*}
   0 &\leq \sum_{i=1}^m\dim\im(T\ev_i) + \dim\im(T\ev) - \dim \Gamma \cr
   &\leq \frac{1}{2}\sum_{i=1}^m(2n-4+\mu_i-|2-\mu_i|) + 2m-2n-2 -
   m(n-1) \cr
   &= \frac{1}{2}\sum_{i=1}^m(\mu_i-|2-\mu_i|) + m - (2n+2) \cr 
   &= \frac{1}{2}\sum_{i=1}^m(-\mu_i-|2-\mu_i|) + m, 
\end{align*}
hence 
$$
   \sum_{i=1}^m(\mu_i+|2-\mu_i|) \leq 2m. 
$$
Now note that $\mu_i+|2-\mu_i|\geq 2$ for any number $\mu_i$, with
equality iff $\mu_i\leq 2$. It follows that $\mu_i\leq 2$ for all
$i$. Since $\sum_{i=1}^m\mu_i=2n+2$ and all $\mu_i$ are even, at least
$n+1$ of the $\mu_i$ must be equal to $2$.  
\end{proof}

The proof of Theorem~\ref{thm:uniruled-dim2} uses the following
immediate consequence of results by McDuff
in~\cite{McD90,McD91,McD92}. Here a symplectic $2$-manifold is called
{\em rational} if it is symplectomorphic to $\C\P^2$ with a 
multiple of the Fubini-Study form, and {\em ruled} if it is the total
space of a symplectic fibration over a closed oriented surface.  

\begin{prop}\label{prop:rat-ruled}
Every uniruled symplectic $4$-manifold $(X,\om)$ is the blow-up of a
rational or ruled symplectic $4$-manifold $(\bar X,\bar\om)$. In the
rational case, $(X,\om)$ possesses a nonzero Gromov-Witten invariant
$\GW^{X,A}(\pt,\pt)$ of holomorphic spheres passing through two
prescribed points (with no further constraints), and in the ruled case it
possesses a nonzero Gromov-Witten invariant $\GW^{X,A}(\pt)$ of
holomorphic spheres passing through one prescribed point. 
\end{prop}

\begin{proof}
By definition of uniruledness, there exists a nonzero Gromov-Witten
invariant $\GW^{X,A}(\alpha_1,\dots,\alpha_r)$ of $J$-holomorphic spheres
in $X$ in the homology class $A$ passing through a prescribed point $\alpha_1$ and cycles (in
singular homology) $\alpha_2,\dots,\alpha_r$. The corresponding moduli
space $\MM$ is a manifold of expected dimension 
$$
   \dim\MM = -2+2c_1(A) - \sum_{i=1}^r(\codim\,\alpha_i-2) = 0. 
$$
If $\codim\,\alpha_i<2$ for some $i$, then the dimension of the moduli
space with the $i$-th marked point removed would be negative and thus
$\MM$ would be empty. Hence $\codim\,\alpha_i\geq 2$ for all
$i$. Since $\codim\,\alpha_1=4$, it follows that 
$$
   2c_1(A) = 2+\sum_{i=1}^r(\codim\,\alpha_i-2) \geq 4,
$$
and thus $c_1(A)\geq 2$. Consider a $J$-holomorphic sphere
$C\in\MM$. According to~\cite[Proposition 1.2]{McD91}, after a
$C^1$-small perturbation of $C$ and a $C^0$-small perturbation of $J$
we may assume that $C$ is immersed. Then $C$ is an immersed symplectic
sphere with positive transverse self-intersections and Chern number
$c_1([S])\geq 2$, and~\cite[Theorem 1.4]{McD92} yields that $(X,\om)$
is the blow-up of a rational or ruled symplectic $4$-manifold $(\bar
X,\bar\om)$.

In the rational case, we have the Gromov-Witten invariant $\GW^{\C
  P^2,[\C\P^1]}(\pt,\pt)=1$ of complex lines passing through two prescribed
points $p\neq q$, and blowing up away from the unique complex line $C$
through $p,q$ we obtain a nonvanishing two-point invariant of
$(X,\om)$. (To see this, note that any other holomorphic sphere $C'$
in $X$ in the homology class $A$ passing through $p,q$ would have at
least two intersection points with $C$ and homological intersection
number $C\cdot C'=A\cdot A=1$, so $C'=C$ by positivity of
intersections.)   
In the ruled case, we have the Gromov-Witten invariant $\GW^{\bar
  X,[{\rm fibre}]}(\pt)=1$ of holomorphic spheres in the class of a
fibre through a prescribed point $p$, and blowing up away from the
fibre through $p$ we obtain a nonvanishing one-point invariant of
$(X,\om)$.  
\end{proof}

\begin{proof}[Proof of Theorem~\ref{thm:uniruled-dim2}]
Let $L$ be a closed orientable surface with a metric of negative
curvature. Suppose $L$ admits a Lagrangian embedding into a uniruled
symplectic 4-manifold $(X,\om)$.
By Proposition~\ref{prop:rat-ruled}, there exists a nonzero Gromov-Witten
invariant $\GW^{X,A}(\alpha_1,\dots,\alpha_r)$ of holomorphic spheres
passing through prescribed points $\alpha_1,\dots,\alpha_r$ (where $r$
equals $1$ or $2$). Pick a tubular
neighbourhood of $L$ and almost complex structures $J$, $J_k$ as in the
proof of Theorem~\ref{thm:disk}. Let $f_k$ be $J_k$-holomorphic spheres
representing the Gromov-Witten invariantm, where we choose $\alpha_1$
on $L$ and $\alpha_2$ (if it appears) outside the tubular
neighbourhood of $L$. In the limit $k\to\infty$ we
find an $N$-level broken holomorphic curve $F=(F^{(1)},\dots,F^{(N)})$.

Assume that every component of $F$ in $X\setminus L$ is regular. Hence
the index (expected dimension, including possibly the constraint
$\alpha_2$) of each such component is nonnegative. As all closed
geodesics have Morse index zero, the index of a component with $k$
punctures and no constraint $\alpha_1$ in $T^*L$ or the
symplectization $\R\times M$ of the unit cotangent bundle equals
$k-2\geq 0$. The component $C_1$ in $T^*L$ carrying the point constraint
$\alpha_1$ has index $k-4$, where $k$ is the number of (positive)
punctures. If $k=2$, then $C_1$ is a holomorphic cylinder in $T^*L$
projecting onto a closed geodesic $\gamma$ with two positive punctures
asymptotic to $\pm\gamma$ (this follows from a maximum principle
argument similar to the proof of Lemma~\ref{lem:torus} below). If
$k=3$, then $C_1$ is either simple (and thus doesn't exist for generic
$J$) or a $2$-fold branched cover of a cylinder over a closed geodesic. 
So if we choose the point $\alpha_1$ not to lie on a closed geodesic,
then the cases $k=2$ and $k=3$ cannot occur and $C_1$ has index $\geq
0$. Since the total index of $F$, i.e., the sum of the indices
of all components, equals zero and all indices are nonnegative, each
component must have index zero. 

Consider now a component $C$ of $F$ in $X\setminus L$ with just one
negative puncture asymptotic to a closed geodesic $\gamma$. Its index
equals either $-1+\mu(\gamma,\Phi)$ or
$-1+\mu(\gamma,\Phi)-2$ (if it carries the point constraint
$\alpha_2$), where $\mu(\gamma,\Phi)$ is the Maslov index of $\gamma$ in the
trivialization of $T^*X$ that extends over $C$.
Now any loop of oriented Lagrangian subspaces has even Maslov index.
Hence the index of $C$ is odd, contradicting the fact that it must be
zero. 
\end{proof}

We see that both proofs work provided that we can achieve regularity
for all punctured holomorphic spheres (with suitable point and
tangency constraints) in $T^*L$, $\R\times M$ and $X\setminus L$
resulting from the neck stretching procedure. 
(In the case of Theorem~\ref{thm:uniruled-dim2} the proof actually
only uses regularity in $X\setminus L$). There are three known
techniques one could invoke to achieve such regularity: Kuranishi
structures~\cite{FOOO}, polyfolds~\cite{HWZ}, and domain dependent
perturbations~\cite{CM-trans}. In the remainder of this paper we
will carry out the third approach.


\section{Coherent perturbations}\label{sec:coh}

In this section we define a suitable class of perturbations of the
Cauchy-Riemann equation parametrized by the Deligne-Mumford space
$\bar\MM_{k+1}$. The discussion closely follows~\cite{CM-trans}, see
also~\cite{Ge}. 

{\bf Nodal curves. }
Let us first describe the Delige-Mumford space $\bar\MM_k$ of stable
curves of genus zero with $k$ marked points. We adopt the approach
and notation of~\cite{MS}, see also~\cite{CM-trans}.

A {\em $k$-labelled tree} is a triple $T=(T,E,\Lambda)$, where $(T,E)$
is a (connected) tree with set of vertices $T$ and edge relation
$E\subset T\times
T$, and $\Lambda=\{\Lambda_\alpha\}_{\alpha\in T}$ is a decomposition
of the index set $\{1,\dots,k\}=\amalg_{\alpha\in T}\Lambda_\alpha$. We
write $\alpha E\beta$ if $(\alpha,\beta)\in E$. Note that the
labelling $\Lambda$ defines a unique map $\{1,\dots,k\}\to T$,
$i\mapsto\alpha_i$ by the requirement $i\in\Lambda_{\alpha_i}$.
Let
$$
   e(T) = |T|-1
$$
be the number of edges. A {\em tree homomorphism} $\tau:T\to\tilde T$
is a map which collapses some subtrees of $T$ to vertices of $\tilde
T$. 
A tree homomorphism $\tau$ is called {\em tree
  isomorphism} if it is bijective and $\tau^{-1}$ is a tree
homomorphism.

A tree $T$ is called {\em stable} if for each $\alpha\in T$,
$$
   n_\alpha := \#\Lambda_\alpha + \#\{\beta\mid \alpha E\beta\} \geq
   3.
$$
Note that for $k<3$ every $k$-labelled tree is unstable. For $k\geq
3$, a $k$-labelled tree $T$ can be canonically stabilized  to a
stable $k$-labelled tree $\st(T)$ by deleting the vertices with
$n_\alpha<3$ and modifying the edges in the obvious way.

A {\em nodal curve of genus zero with $k$ marked points
  modelled over the tree $T=(T,E,\Lambda)$} is a tuple
$$
   \z = (\{z_{\alpha\beta}\}_{\alpha E\beta},\{z_i\}_{1\leq i\leq k})
$$
of points $z_{\alpha\beta},z_i\in S^2$ such that for each $\alpha\in
T$ the {\em special points}
$$
   SP_\alpha:=\{z_{\alpha\beta}\mid \alpha E\beta\}\cup \{z_i\mid
   \alpha_i=\alpha\}
$$
are pairwise distinct. Note
that $n_\alpha=\# SP_\alpha$. For $\alpha\in T$ and $i\in\{1,...,k\}$
denote by $z_{\alpha i}$ either the point $z_i$ if
$i\in\Lambda_\alpha$, or the point $z_{\alpha\beta_1}$ if
$z_i\in\Lambda_{\beta_r}$ and
$(\alpha,\beta_1),(\beta_1,\beta_2),...,(\beta_{r-1},\beta_r)\in
E$. We associate to $\z$ the {\em nodal
Riemann surface}
$$
   \Sigma_\z := \coprod_{\alpha\in T}S_\alpha\Bigr/z_{\alpha\beta}\sim
   z_{\beta\alpha},
$$
obtained by gluing a collection of standard spheres
$\{S_\alpha\}_{\alpha\in
  T}$ at the points $z_{\alpha\beta}$ for $\alpha E\beta$, with marked
points $z_i\in S_{\alpha_i}$, $i=1,\dots,k$. Note that $\z$ can be
uniquely recovered from $\Sigma_\z$, so we will sometimes not
distinguish between the two. A nodal curve $\z$ is called {\em
stable} if the underlying tree is stable, i.e., every sphere
$S_\alpha$ carries at least $3$ special points. Stabilization of
trees induces a canonical stabilization of nodal curves
$\z\mapsto\st(\z)$. 

We will usually omit the
genus zero from the notation. Denote the space of all nodal curves
(of genus zero) with $k$ marked points modelled over $T$ by
$$
   \tilde\MM_T\subset (S^2)^E\times (S^2)^k.
$$
Note that this is an open subset of the product of spheres.
A {\em
morphism} between nodal curves $\z,\tilde\z$ modelled over trees
$T,\tilde T$ is a tuple
$$
   \phi=(\tau,\{\phi_\alpha\}_{\alpha\in T}),
$$
where $\tau:T\to\tilde T$ is a tree homomorphism and
$\phi_\alpha:S^2\cong S_\alpha\to S_{\tau(\alpha)}\cong S^2$ are
(possibly constant) holomorphic maps such that
\begin{align*}
   \tilde
   z_{\tau(\alpha)\tau(\beta)} &= \phi_\alpha(z_{\alpha\beta}) \mbox{ if }
   \tau(\alpha)\neq\tau(\beta),\\
   \phi_\alpha(z_{\alpha\beta}) &= \phi_\beta(z_{\beta\alpha}) \mbox{ if }
   \tau(\alpha)= \tau(\beta),\\
   \tilde\alpha_i & =\tau(\alpha_i),\qquad \tilde
   z_i=\phi_{\alpha_i}(z_i)
\end{align*}
for $i=1,\dots,k$ and $\alpha E\beta$. A morphism
$\phi:\z\to\tilde\z$ induces a natural holomorphic map
$\Sigma_\z\to\Sigma_{\tilde\z}$ (i.e., a continuous map that is
holomorphic on each component $S_\alpha$). A morphism
$(\tau,\{\phi_\alpha\})$ is called {\em isomorphism} if $\tau$ is a
tree isomorphism and each $\phi_\alpha$ is biholomorphic.

Isomorphisms from $\z$ to itself are called {\em automorphisms}.
If $\z$ is stable its only automorphism is the identity
(see~\cite{MS}, discussion after Definition D.3.4). Thus for a stable
tree $T$ we have a free and proper holomorphic action
$$
   G_T\times\tilde\MM_T\to\tilde\MM_T
$$
of the group $G_T$ of isomorphisms fixing $T$.
Hence the quotient
$$
   \MM_T :=\tilde\MM_T/G_T
$$
is a complex manifold of dimension
$$
   \dim_\C\MM_T = k+2e(T)-3|T| = k-3-e(T).
$$
For $k\geq 3$, denote by $\MM_k=\tilde\MM_k/G$ the moduli space of
stable curves modelled over the $k$-labelled tree with one
vertex. As a set, the {\em Deligne-Mumford space (of
genus zero) with $k$ marked points} is given by
$$
   \bar\MM_k := \coprod_T\MM_T,
$$
where the union is taken over the (finitely many) isomorphism
classes of stable $k$-labelled trees. However, $\bar\MM_k$ is
equipped with the topology of Gromov convergence which makes it a
compact connected metrizable space (see~\cite{MS}). As the
notation suggests, $\bar\MM_k$ is the compactification of $\MM_k$
in the Gromov topology. For a stable $k$-labelled tree $T$, the
closure of $\MM_T$ in $\bar\MM_k$ is given by
$$
   \bar\MM_T = \coprod_{\tilde T}\MM_{\tilde T},
$$
where the union is taken over all isomorphism classes of stable
$k$-labelled trees $\tilde T$ for which there exists a surjective tree
homomorphism $\tau:\tilde T\to T$ with $\tau(\tilde\alpha_i)=\alpha_i$
for $i=1,\dots,k$.

We have the following result of Knudsen (cf.~\cite{MS}).

\begin{thm}\label{thm:DM}
For $k\geq 3$, the Deligne-Mumford space $\bar\MM_k$ is a compact
complex manifold of dimension $\dim_\C\bar\MM_k=k-3$. Moreover, for
each stable $k$-labelled tree $T$, the space
$\bar\MM_T\subset\bar\MM_k$ is a compact complex submanifold of
codimension $\codim_\C\bar\MM_T=e(T)$.
\end{thm}


{\bf Stable decompositions. }
Consider the Deligne-Mumford space $\bar\MM_{k+1}$ with $k+1$
marked points $z_0,\dots,z_k$. In the following discussion, the point
$z_0$ plays a special role (it will be the variable for holomorphic
maps in later sections).

We have a canonical projection $\pi:\bar\MM_{k+1}\to\bar\MM_k$ by
forgetting the marked point $z_0$ and stabilizing. The map $\pi$
is holomorphic and the fibre $\pi^{-1}([\z])$ is naturally
biholomorphic to $\Sigma_\z$. The projection
$$
   \pi:\bar\MM_{k+1}\to\bar\MM_k
$$
is called the {\em universal curve}.

Given a stable $(k+1)$-labelled tree $T$, we 
define an equivalence relation on
$\{0,\dots,k\}$ by $i\sim j$ iff $z_{\alpha_0
i}=z_{\alpha_0j}$. The equivalence classes yield a decomposition
$$
   \{0,\dots,k\} = I_0\cup\dots\cup I_\ell.
$$
Note that the marked points on $S_{\alpha_0}$ correspond to
equivalence classes consisting of one element; in particular, we may
put $I_0:=\{0\}$. Stability implies $\ell+1=n_{\alpha_0}\geq
3$. Conversely, we call a decomposition $\I=(I_0,\dots,I_\ell)$ of
$\{0,\dots,k\}$ {\em stable} if $I_0=\{0\}$ and $|\I|:=\ell+1\geq 3$.  
We will always order the $I_j$ such that the integers
$i_j:=\min\{i\mid i\in I_j\}$ satisfy
$$
   0=i_0<i_1<\dots<i_\ell. 
$$ 
Denote by $\MM_\I=\MM_{(I_0,\dots,I_\ell)}\subset\bar\MM_{k+1}$ the
union over those stable trees that give rise to the stable
decomposition $\I=(I_0,\dots,I_\ell)$. The
$\MM_\I$ are submanifolds of $\bar\MM_{k+1}$ with
$$
   \bar\MM_{k+1} = \bigcup_\I\MM_\I,
$$
and the closure of $\MM_\J$ is a union of certain
strata $\MM_\I$ with $|\I|\leq|\J|$. The above ordering of the $I_j$ 
determines a projection
$$
   p_\I:\MM_\I\to\MM_{|\I|},
$$
sending a stable curve $\z$ to the special points on the component
$S_{\alpha_0}$.  

We call the union of the strata $\MM_\I$ with $|\I|=3$ the {\em set of 
  special points} of $\bar\MM_{k+1}$. On its complement
$\bar\MM_{k+1}^*$, the projection 
$$
   \pi:\bar\MM_{k+1}^*\to \bar\MM_k
$$
is a fibration. We denote by
$$
   T^v\bar\MM_{k+1}^* := \ker(d\pi) \to \bar\MM_{k+1}^*
$$
the {\em vertical tangent bundle}. 

{\bf Coherent perturbations. }
Let now $(E,J)\to X$ be a complex vector bundle. 
The pullbacks under the obvious projections give complex vector
bundles $(p_1^*T^v\bar\MM_{k+1}^*,j)$ and $(p_2^*E,J)$ over the
product $\bar\MM_{k+1}^*\times X$. We denote by 
\begin{equation}\label{eq:pert}
   \Hom^{0,1}(p_1^*T^v\bar\MM_{k+1}^*,p_2^*E) \to
   \bar\MM_{k+1}^*\times X 
\end{equation}
the bundle of complex antilinear bundle homomorphisms. Note that for
every stable decomposition $\I$ with $|\I|\geq 4$, the restriction of
this bundle to the stratum $\MM_\I$ is the pullback bundle 
\begin{equation}\label{eq:pert2}
\begin{CD}
   \Hom^{0,1}(p_1^*T^v\MM_\I,p_2^*E) @>>>
   \Hom^{0,1}(p_1^*T^v\MM_{|\I|},p_2^*E) \\ 
   @VVV @VVV \\
   \MM_\I\times X @>{p_\I\times\id}>> \MM_{|\I|}\times X
\end{CD}
\end{equation}

\begin{definition}\label{def:coherent}
A {\em coherent perturbation} on $(E,J)$ is a continuous section $K$
of the bundle~\eqref{eq:pert} satisfying the following two conditions: 

(a) $K$ has compact support; 

(b) for every stable decomposition $\I$ with $|\I|\geq 4$, the
restriction $K|_{\MM_\I}$ is the pullback under $p_\I\times\id$ of a  
smooth section $K_\I$ of the right-hand bundle in~\eqref{eq:pert2}. 
\end{definition}

In view of condition (a), we may view $K$ as being extended by zero to
all of $\bar\MM_{k+1}$ (although the bundle~\eqref{eq:pert} is not
defined at the special points).  
We denote the space of coherent perturbations on $(E,J)$ by
$$
   \KK(\bar\MM_{k+1},E). 
$$
It is equipped with the $C^0$-topology on $\bar\MM_{k+1}$, and the
$C^\infty$-topology on each $\MM_\I$ via the pullback
diagram~\eqref{eq:pert2}. 

To better understand this definition, fix a stable curve
$\z\in\bar\MM_k$, modelled over the $k$-labelled tree $T$. Recall
that $\pi^{-1}[\z]$ is naturally identified with the nodal Riemann
surface $\Sigma_\z$ which is a union of $|T|$ copies $S_\alpha$ of
$S^2$, glued together at the points $z_{\alpha\beta}$. Denote by
$\Sigma_\z^*\subset\Sigma_\z$ the complement of the special points,
which is a smooth (possibly disconnected) punctured Riemann
surface. Restriction of $K\in\KK(\bar\MM_{k+1},E)$ to $\pi^{-1}[\z]$
yields a smooth section $K_\z$ with compact support in the bundle
\begin{equation}\label{eq:pert3}
   \Hom^{0,1}(p_1^*T\Sigma_\z^*,p_2^*E) \to
   \Sigma_\z^*\times X.  
\end{equation}
In other words, $K_\z$ is a $(0,1)$-form on $\Sigma_\z^*$ with values
in the sections of $E$. If $E=TX$ is the tangent bundle of an almost
complex manifold $(X,J)$, then $K_\z$ is a $(0,1)$-form on
$\Sigma_\z^*$ with values in the vector fields on $X$. 

In general, for any (possibly noncompact and disconnected) Riemann
surface $\Sigma$, we denote by  
$$
   \KK(\Sigma,E)
$$ 
the space of smooth sections with compact support in the bundle
\begin{equation}\label{eq:pert4}
   \Hom^{0,1}(p_1^*T\Sigma,p_2^*E) \to
   \Sigma\times X.  
\end{equation}
The proof of the following lemma is analogous to the proof
of~\cite[Lemma 3.10]{CM-trans}. 

\begin{lemma}\label{lem:J-smooth}
For $\z\in\bar\MM_k$, let $K_\z$ be a smooth section with compact
support in the bundle~\eqref{eq:pert3}. Then there exists a
$K\in\KK(\bar\MM_{k+1},E)$ whose restriction to $\pi^{-1}[\z]$ equals
$K_\z$. \hfill$\square$ 
\end{lemma}

{\bf $K$-holomorphic maps. }
Let $(X,J)$ be an almost complex manifold, $\Sigma$ a Riemann surface,
and $K\in\KK(\Sigma,TX)$. Then to a map $f:\Sigma\to X$ we associate
the smooth section
$$
   \pb_{J,K}f := df + J(f)\circ df\circ j +
   K(f) 
$$
of the bundle $\Hom^{0,1}(T\Sigma,f^*TX)\to \Sigma$ whose value at
$z\in\Sigma$ is the complex antilinear homomorphism 
$$
   d_zf + J\bigl(f(z)\bigr)\circ d_zf\circ j_z
   + K(z,f(z))\ :\ T_z\Sigma\to T_{f(z)}X. 
$$
A map $f$ satisfying $\pb_{J,K}f=0$ will be called {\em
  $K$-holomorphic}. 

More globally, a coherent perturbation $K\in\KK(\bar\MM_{k+1},TX)$
associates to each $\z\in\bar\MM_k$ a coherent perturbation
$K_\z\in\KK(\Sigma_\z^*,TX)$, and thus to a map $f:\Sigma_\z^*\to X$ 
the smooth section $\pb_{J,K_\z}f$ of the bundle
$\Hom^{0,1}(T\Sigma_\z^*,f^*TX)\to \Sigma_\z^*$.  
\medskip

{\bf Taming conditions. }
Let $(X,J)$ be an almost complex manifold, $\Sigma$ a Riemann surface,
and $K\in\KK(\Sigma,TX)$. Then a $K$-holomorphic map $\Sigma\to X$ can
be viewed as a holomorphic section for the almost complex structure
$\wh J$ on $\Sigma\times X$ defined by the matrix
$$
   \wh J := \left(\begin{matrix}
      j & 0 \\ -K\circ j & J\end{matrix}\right)
$$ 
with respect to the splitting $T(\Sigma\times X) = T\Sigma\oplus TX$. 
Here $\wh J^2=-\Id$ follows from $K\circ j=-J\circ K:T\Sigma\to TX$,
and a short computation shows that $f:\Sigma\to X$ is $K$-holomorphic
iff $\wh f:=\id\times f:\Sigma\to\Sigma\times X$ is $\wh
J$-holomorphic. Another short computation shows that the projection
$\Sigma\times X\to\Sigma$ is holomorphic. 
 
For an area form $\sigma$ on $\Sigma$, we obtain a symplectic form
$\wh\om := \sigma\oplus\om$ on $\Sigma\times X$. The following
computation shows that $\wh\om$ tames $\wh J$ if the $C^0$-norm of
$K$, measured with respect to the metrics $\sigma(\cdot,j\cdot)$ and
$\om(\cdot,J\cdot)$, satisfies $\|K\|\leq 1$:
\begin{align*}
   \wh\om\left(\left(\begin{matrix}v \\ w\end{matrix}\right),
   \wh J\left(\begin{matrix}v \\ w\end{matrix}\right)\right)
   &= \wh\om\left(\left(\begin{matrix}v \\ w\end{matrix}\right),
   \left(\begin{matrix}jv \\ Jw-Kjv\end{matrix}\right)\right) \cr
   &= \sigma(v,jv) + \om(w,Jw) - \om(w,Kjv) \cr
   &\geq |v|^2 + |w|^2 - \|K\|\,|v|\,|w| \cr
   &\geq\frac{1}{2}(|v|^2+|w|^2). 
\end{align*}

Next, consider an oriented manifold $M^{2n-1}$ with a {\em stable
Hamiltonian structure} $(\om,\lambda)$, i.e., a closed $2$-form $\om$
and a $1$-form $\lambda$ satisfying
$$
   \ker(\om)\subset\ker(d\lambda),\qquad \lambda\wedge\om^{n-1}>0.   
$$
It determines a {\em Reeb vector field} $R$ by the conditions
$i_R\lambda=1$ and $i_R\om=0$. Note that a contact form $\lambda$
induces a stable Hamitonian structure $(d\lambda,\lambda)$. As in the
contact case, an almost complex structure $J$ on $\R\times M$ is
called {\em compatible with
$(\om,\lambda)$} if $J$ is $\R$-invariant, preserves
$\xi=\ker\lambda$, maps $\p_r$ to the Reeb vector field $R$, and
$J|_\xi$ is compatible with $\om$. We say that $J$ is {\em tamed by
$(\om,\lambda)$} if the compatibility condition on $J|_\xi$ is relaxed
to the taming condition $\om(v,Jv)>0$ for all $0\neq v\in\xi$.  

For a Riemann surface $\Sigma$ with area form $\sigma$, the pair
$(\wh\om=\sigma\oplus\om,\wh\lambda=\lambda)$ defines a stable
Hamiltonian structure on $\Sigma\times M$. Let $J$ be an
almost complex structure on $\R\times M$ tamed by
$(\om,\lambda)$. For
$K\in\KK\bigl((\Sigma,T(\R\times M)\bigr)$, we define an almost complex structure $\wh J$ on 
$\Sigma\times(\R\times M)$ by $\wh J(\p_r)=R$ and the matrix
$$
   \wh J := \left(\begin{matrix}
      j & 0 \\ -K\circ j & J\end{matrix}\right)
$$ 
on $T\Sigma\oplus \xi$. This $\wh J$ is tamed by $(\wh\om,\wh\lambda)$
provided that $K$ is $\R$-invariant and $\|K\|\leq 1$.  

\begin{remark}
A map $f=(a,u):\Sigma\to\R\times M$
defines a $\wh J$-holomorphic section $\wh f=\id\times f$ in
$\Sigma\times(\R\times M)$ iff it satisfies the equations
\begin{align*}
   da - u^*\lambda\circ j = 0,\qquad \pi_\xi du+J(u)\circ\pi_\xi
   du\circ j + K(z,u)=0,
\end{align*}
where $\pi_\xi:TM\to\xi$ is the projection along $R$. Note that the
second equation is {\em not} satisfied by branched covers of orbit
cylinders provided that $K(z,u)$ does not vanish identically. This
suggests that the class of coherent perturbation is large enough to
get rid of such non-regular branched covers, which will be made
precise in the next section. 
\end{remark}

{\bf Globalization and compactness. }
To globalize this construction, we fix a K\"ahler form $\sigma$ on
the Deligne-Mumford space $\bar\MM_{k+1}$. Such a K\"ahler form can,
for example, be obtained from the embedding of $\bar\MM_{k+1}$ into a
product of $\C\P^1$'s constructed in~\cite{MS}. Now every
component $S_\alpha$ of a fibre $\pi^{-1}(\z)$ of the projection
$\pi:\bar\MM_{k+1}\to\bar\MM_k$ is an embedded holomorphic sphere, so
$\sigma|_{S_\alpha}$ defines a positive area form on $S_\alpha$. For
an almost complex manifold $(X,J)$ tamed by $\om$ and a
coherent perturbation $K\in\KK(\bar\MM_{k+1},TX)$, the construction
above yields a complex structure $\wh J$ on the vertical tangent bundle
(outside the special points) of the projection $\bar\MM_{k+1}\times
X\to\bar\MM_k$. This complex structure is tamed by the restriction of
$\wh\om=\sigma\oplus\om$ to the fibres provided that the $C^0$-norm of
$K$, measured with respect to the metrics $\sigma(\cdot,j\cdot)$ and
$\om(\cdot,J\cdot)$, satisfies $\|K\|\leq 1$. 

If $(X,J)$ is a symplectization $\R\times M$ of a manifold with a
stable Hamiltonian structure, then we require in
addition that $K$ is $\R$-invariant. 
If $(X,J)$ has cylindrical ends of this type, then we require that $K$
is $\R$-invariant outside a compact subset $X_0\subset X$. 
If $X^*=X^+\amalg(\R\times M)\amalg X^-$ is a split manifold, then $K$ is a triple $(K^+,K_M,K^-)$ where $K^\pm\in\KK(\bar\MM_{k+1},TX^\pm)$ and $K_M\in\KK(\bar\MM_{k+1},T(\R\times M))$, and the $K^\pm$ agree with $K_M$ outside compact subsets $X_0^\pm\subset X^\pm$. 

The Compactness Theorem~\ref{thm:compact} continues to hold for
$K$-holomorphic maps by formally applying it (in the setting of stable
Hamiltonian structures) to the graphs of the
holomorphic maps $f_k$ in $\Sigma_k\times X_k$, using the convergence
of a subsequence of the domains $(\Sigma_k,j_k)$ in the
Deligne-Mumford space. 

\begin{remark}
In the preceding discussion we have used the setting of stable
Hamiltonian structures because this is most natural for the graph
construction. If the cylindrical ends are contact manifolds (which is
the case we are interested in), one could avoid the use of stable
Hamiltonian structures as follows. Consider a contact manifold
$(M,\lambda)$. For a fixed open Riemann surface
$\Sigma$, one considers a $1$-form $\beta$ on $\Sigma$ such that
$d\beta=\sigma$ is an area form and a family of contact forms
$\lambda_z$ on $M$ depending with compact support on
$z\in\Sigma$. Then $\wh\lambda=\lambda_\bullet+\beta$ defines a contact
form on $\Sigma\times M$ provided the derivatives of $\lambda_z$ with
respect to $z$ are sufficiently small. Now one considers almost
complex structures $\wh J$ on $\Sigma\times(\R\times M)$ compatible
with (or tamed by) $\wh\lambda$. To globalize the construction, one
picks a primitive $\beta$ of the restriction of the K\"ahler form
$\sigma$ on $\bar\MM_{k+1}$ to the open stratum $\MM_{k+1}$ (which
exists because the codimension two homology of $\bar\MM_{k+1}$ is
generated by boundary divisors). The corresponding holomorphic curves
will again satisfy Gromov-Hofer compactness. However, the computations
in the following section would be more involved for this class of
perturbations. 
\end{remark}

\section{Linearization at punctured $K$-holomorphic
  maps}\label{sec:lin}  

In this section we study the linearization of the universal
Cauchy-Riemann operator at punctured $K$-holomorphic maps into almost
complex manifolds with cylindrical ends. Our goal is to derive
conditions under which this linearization is surjective. 
Let $(X,J)$ be an almost complex manifold with cylindrical ends
$\R\times\ol{M},\R\times\ul{M}$ modelled over contact manifolds (or
manifolds with stable Hamiltonian structures) $\ol{M},\ul{M}$. We
assume that $\ol{M}$ and $\ul{M}$ are of Morse-Bott type. 

{\bf Fredholm setup. }
We adapt the Fredholm setup for moduli spaces of punctured
holomorphic curves in~\cite{BM} to our situation. 
We consider first the case of a fixed domain and fix the following data: 
\begin{itemize}
\item distinct points $\overline{z}:=(\overline{z}_1,\dots,
\overline{z}_{\overline{p}})$, $\underline{z}:=
(\underline{z}_1,\dots,\underline{z}_{\underline{p}})$ be on the
Riemann sphere $S=S^2$, denoting by 
$\dot{S}:=S\setminus\{\overline{z}_i,\underline{z}_j\}$ the
corresponding punctured sphere;
\item holomorphic cylindrical coordinates $\R_\pm\times S^1$ near the
(positive resp.~negative) punctures of $\dot S$; 
\item vectors $\overline{\Gamma}:= (\overline{\Gamma}_1,\dots, 
\overline{\Gamma}_{\overline{p}}), \underline{\Gamma}:=
(\underline{\Gamma}_1,\dots, \underline{\Gamma}_{\underline{p}})$ of
components $\ol{\Gamma}_i=\ol{N}_i/S^1$, $\ul{\Gamma}_i=\ul{N}_i/S^1$
in the manifolds of closed Reeb orbits in $\overline{M},
\underline{M}$;
\item $m\in\N$ and $p>1$ with $mp>2$ and a small $\delta>0$.  
\end{itemize}
We denote by
$$
   \BB:=\BB^{m,p,\delta}(X,\ol{\Gamma},\ul{\Gamma}) 
$$ 
the space of all maps $f:\dot S\to X$ of Sobolev class $W^{m,p}_{\rm
  loc}$ which converge near each positive puncture $\ol{z}_i$ to a
closed Reeb orbit $\ol{\gamma}_i$ in the family $\ol{\Gamma}_i$ in the
following sense (and similary near the negative punctures):
Suppose first that $\ol{\gamma}_i$ is simple. Pick a tubular
neighbourhood
$U\cong\R^{\dim\ol{\Gamma}_i}\times\R/\ol{T}_i\Z\times\R^{2n-\dim\ol{\Gamma}_i-2}$ 
of $\ol{\gamma}_i\cong \{0\}\times\R/\ol{T}_i\Z\times\{0\}$ in $\ol{M}$
such that  $\ol{\Gamma}_i\cap U\cong
\R^{\dim\ol{\Gamma}_i}\times\R/\ol{T}_i\Z\times\{0\}$. 
Then $f$ maps a neighbourhood $[\rho,\infty)\times S^1$ of $\ol{z}_i$
to $\R\times U$ such that in the cylindrical coordinates
$(s,t)\in[\rho,\infty)\times 
S^1$ near $\ol{z}_i$ the map $f=(a,v,\vartheta,w):[\rho,\infty)\times
S^1\to\R\times
\R^{\dim\ol{\Gamma}_i}\times\R/\ol{T}_i\Z\times\R^{2n-\dim\ol{\Gamma}_i-2}$
satisfies 
\begin{align*}
   &\Bigl(a(s,t)-\ol{a}_i-\ol{T}_is,v(s,t),\vartheta(s,t)-\ol{\vartheta}_i-
   \ol{T}_it,w(s,t)\Bigr) \cr
   &\in W^{m,p,\delta}\bigl([\rho,\infty)\times
   S^1,\R\times \R^{\dim\ol{\Gamma}_i}\times\R\times\R^{2n-\dim\ol{\Gamma}_i-2}\bigr)
\end{align*}
for constants $\ol{a}_i\in\R$,
$\ol{\vartheta}_i\in S^1$. Here $W^{m,p,\delta}\bigl([\rho,\infty)\times 
S^1,\R^{2n}\bigr)$ denotes the weighted Sobolev space of functions $g$
such that $e^{\delta s}g(s,t)$ belongs to $W^{m,p}$.
An open neighborhood of $f$ is given by maps $f':\dot S\to X$ of Sobolev class $W^{m,p}_{\rm loc}$,
which map a neighborhood $[\rho,\infty)\times S^1$ of $\ol{z}_i$
to $\R\times U$ such that 
\begin{align*}
   &\Bigl(a'(s,t)-\ol{a'}_i-\ol{T}_is,v'(s,t)-\ol{v'}_i,\vartheta'(s,t)-\ol{\vartheta'}_i-
   \ol{T'}_it,w'(s,t)\Bigr) \cr
   &\in W^{m,p,\delta}\bigl([\rho,\infty)\times
   S^1,\R\times \R^{\dim\ol{\Gamma}_i}\times\R\times\R^{2n-\dim\ol{\Gamma}_i-2}\bigr)
\end{align*}
for constants $\ol{a}_i\in\R$,
$\ol{\vartheta}_i\in S^1$, and $\ol{v'}_i\in \R^{\dim\ol{\Gamma}_i}$.
The case of $d$-fold covered $\ol{\gamma}_i$ is reduced to the
previous one by passing to the $d$-fold cover of the neighbourhood
$U$, see~\cite{BEHWZ}. 
Note that $\BB$ is a Banach manifold with a bundle projection
$$
   \Ev:\BB\to \prod_{i=1}^{\ol{p}}(\R\times\ol{N}_i)\times
   \prod_{j=1}^{\ul{p}}(\R\times\ul{N}_j),\qquad f\mapsto 
   \{\ol{a}_i,\ol{v}_i,\ol{\vartheta}_i,\ul{a}_j,\ul{v}_j,\ul{\vartheta}_j\}.  
$$
\begin{remark}
The Banach manifold $\BB$ does not depend on the chosen cylindrical
coordinates near the punctures, but the bundle projection $\Ev$ does.
Its composition with the projections to the $\ol{N}_i$ and
$\ul{N}_j$ depends on the choice of asymptotic directions at the
punctures, and 
its composition with the projections to the $\ol{\Gamma}_i$ and
$\ul{\Gamma}_j$ is also independent of the asymptotic directions. 
\end{remark}

We will fix $J$ (and often suppress it from the notation), and vary
$K$ in the space 
$$\KK_{\dot S}\subset\KK(\dot S,TX)$$ 
of perturbations as defined above satisfying $\|K\|<1$. Recall that 
$K$ is required to be $\R$-invariant outside a compact subset
$X_0\subset X$ (and on all of $X=\R\times M$ in the cylindrical
case). 
Since $K$ has compact support in $\dot S$ by definition,
$K$-holomorphic maps $\dot S\to X$ are just $J$-holomorphic near the
punctures. 
In view of the asymptotics in~\cite{HWZ96}, each punctured
$K$-holomorphic map $f:\dot S\to X$ asymptotic to
$\ol{\Gamma},\ul{\Gamma}$ belongs to $\BB^{m,p,\delta}$ for $\delta$
sufficiently small. Let
$$
   \EE := \EE^{m-1,p,\delta}\to\BB
$$
be the Banach bundle whose fibre at $f\in\BB$ is given by
$$
   \EE_f=W^{m-1,p,\delta}\bigl(\dot S,\Om^{0,1}(f^*TX)\bigr).
$$ 
Then for $K\in\KK_{\dot S}$ the Cauchy-Riemann operator defines a
smooth section 
$$
   \pb_{J,K}:\BB\to\EE,\qquad f\mapsto \pb_{J,K}f.
$$
We denote its linearization by
$$
   D_f:T_f\BB\to\EE_f.
$$
More generally, the {\em universal Cauchy-Riemann operator} is the
section
$$
   \pb:\BB\times\KK_{\dot S}\to\EE,\qquad (f,K)\mapsto\pb_{J,K}f.
$$
Its linearization
$$
   T_f\BB\oplus T_K\KK_{\dot S}\to\EE_f
$$
can be described as follows. Recall that $f$ maps a neighbourhood
$[\rho,\infty)$ of the puncture $\ol{z}_i$ into $\R\times U$ with $U$
as defined above. 
Let $\ol{\zeta}_i\in\R^{\dim\ol{\Gamma}_i+2}$ be a tangent
vector to $\R\times \R^{\dim\ol{\Gamma}_i}\times\R/\ol{T}_i\Z$ at  
$(\ol{a}_i,\ol{v}_i,\ol{\vartheta}_i)$. 
Pick a cutoff function $\beta:\R\to[0,1]$ with $\beta(s)=1$ for $s\geq
\rho+1$ and $\beta(s)=0$ for $s\leq\rho$. Then $\ol{\zeta}_i$ extends to a
section $\beta(s)\ol{\zeta}_i$ in $f^*TX$ and we get an isomorphism
$$
   \bigoplus_{i=1}^{\ol{p}}\R^{\dim\ol{\Gamma}_i+2}\oplus
   \bigoplus_{j=1}^{\ul{p}}\R^{\dim\ul{\Gamma}_j+2}\oplus W^{m,p,\delta}(f^*TX) \to T_f\BB, \qquad
   (\zeta,g)\mapsto\hat\zeta + g, 
$$
where for $\zeta = (\ol{\zeta}_i,\ul{\zeta}_j)$ we have set
$$
   \hat\zeta := \sum_i\beta(s)\ol{\zeta}_i+\sum_j\beta(-s)\ul{\zeta}_j.  
$$
Under this isomorphism, the linearization of the universal
Cauchy-Riemann operator at $(f,K)$ is given by
\begin{gather}\label{eq:lin}
   \bigoplus_{i=1}^{\ol{p}}\R^{\dim\ol{\Gamma}_i+2}\oplus
   \bigoplus_{j=1}^{\ul{p}}\R^{\dim\ul{\Gamma}_j+2}\oplus W^{m,p,\delta}(f^*TX) \oplus
   T_K\KK_{\dot S}\to\EE_f, \cr 
   (\zeta,g,Y) \mapsto D_fg + D_f\hat\zeta + Y(z,f). 
\end{gather}

The following lemma is an easy adaptation of
the corresponding result in~\cite{CM-trans}.

\begin{lemma}\label{lem:surj}
Let $(X,J)$ be an almost complex manifold with cylindrical ends. 
Let $f\in\BB=\BB^{m,p,\delta}(X,\ol{\Gamma},\ul{\Gamma})$ and
$K\in\KK_{\dot S}$ satisfy $\pb_{J,K}f=0$. 
Then the linearization of
$$
   (\pb,\Ev):\BB\times\KK_{\dot S}\to \EE\times \prod_{i=1}^{\ol{p}}(\R\times\ol{N}_i)\times
   \prod_{j=1}^{\ul{p}}(\R\times\ul{N}_j)
$$
at $(f,K)$ is surjective. More precisely, for any nonempty open
subset $U\subset\dot S$, 
the restriction of the linearization to the subspace
$$
   T_f\BB\oplus T_K\KK_U\subset T_f\BB\oplus T_K\KK_{\dot S}
$$
of sections with support in $U$ is surjective.
\end{lemma}

\begin{proof}
Note that under the isomorphism above, the linearized evaluation map
is simply the projection $(\zeta,g,Y) \mapsto \zeta$. Hence, in view
of equation~\eqref{eq:lin}, the assertion
of the lemma is equivalent to surjectivity of the map
\begin{align*}
   L:W^{m,p,\delta}(f^*TX) \oplus T_K\KK_U\to\EE_f, \cr 
   (g,Y) \mapsto D_fg + Y(z,f). 
\end{align*}
The proof of surjectivity of $L$ is completely analogous to the proof
of~\cite[Lemma 4.1]{CM-trans}. For the convenience of the reader, let
us recall the argument. 
We first consider the case $m=1$. Suppose that 
$\eta$ is an $L^q$-section orthogonal to the image of $L$. Then
$\eta$ is smooth and satisfies the equations
\begin{align*}
   D_f^*\eta &= 0, \cr
   \la Y,\eta\ra_{L^2} &= 0 \text{ for all }Y\in
   T_K\KK_U. 
\end{align*}
The second equation implies $\eta|_U\equiv 0$. Note that this also
holds if $f(U)$ is entirely contained in the cylindrical part (in
which case we must choose $Y$ to be $\R$-invariant) because we can
still arbitrarily prescribe $Y$ at one point and then cut it off in
the domain $U$. Since $\rho$ satisfies unique continuation by the
first equation, it must vanish identically on $\dot S$. This proves
surjectivity of $L$ for $m=1$, and the general case follows by
elliptic regularity.  
\end{proof}

Finally, let us record a lemma for (unperturbed) holomorphic cylinders
that we will need in the proof of Theorem~\ref{thm:Audin} (b). 
Let $T^*T^n$ be the cotangent bundle of the $n$-torus, equipped with
an almost complex structure $J_\rho$ induced by the standard flat metric on
$T^n$,
$$
   J_\rho:\frac{\p}{\p q_i}\mapsto -\rho(|p|)\frac{\p}{\p p_i},\qquad
   \rho(|p|)\frac{\p}{\p p_i}\mapsto\frac{\p}{\p q_i}.  
$$
Here $(q_i,p_i)$, $i=1,\dots,n$, are standard coordinates on $T^*T^n\cong
T^n\times\R^n$ and $\rho:[0,\infty)\to(0,\infty)$ is a smooth function
satisfying $\rho(r)=1$ near $r=0$ and $\rho(r)=r$ for large $r$. 

\begin{lemma}\label{lem:torus}
Let $\sigma\in H_1(T^n;\Z)$ be the $k$-fold multiple of a primitive
class in $H_1(T^n;\Z)$.  
Then the moduli space $\MM_\sigma$ of $J_\rho$-holomorphic cylinders in
$T^*T^n$ with two positive punctures asymptotic to closed geodesics in
the classes $\pm\sigma$ is a smooth manifold of dimension
$2n-2$ whose image under the evaluation map is a $k$-fold covering of
$T^*T^n$. 
\end{lemma}

\begin{proof}
Write $\sigma=k\bar\sigma$ for a primitive class $\bar\sigma\in
H_1(T^n;\Z)\cong\Z^n$ and $k\geq 1$. Note that a linear coordinate change
$(q,p)\mapsto(Aq,Ap)$ does not change the form of $J_\rho$ (only the
norm $|p|$ looks different in the new coordinates). So after such a
coordinate change we may assume that $\bar\sigma=(1,0,\dots,0)$. 

Note that for each constant $(\bar q',\bar p')=(\bar q_2,\dots,\bar
q_n,\bar p_2,\dots,\bar p_n)$ the {\em orbit cylinder}
$$
   C_{(\bar q',\bar p')} := \{q_i=\bar q_i,p_i=\bar
     p_i,i=2,\dots,n\}\subset T^*T^n
$$
is $J_\rho$-invariant, and these orbit cylinders foliate
$T^*T^n$. Each orbit cylinder is the image of a $J_\rho$-holomorphic
$k$-fold covering map $f:\R\times S^1\to C_{(\bar q',\bar p')}$, which is
unique up to automorphism of the domain. To see this, let us write out
the equations for $J_\rho$-holomorphicity of a map $f=(q,p):\R\times
S^1\to T^*T^n$: 
\begin{align}\label{eq:hol}
   \p_tq_i = \frac{1}{\rho}\p_sp_i,\qquad 
   \p_sq_i = -\frac{1}{\rho}\p_tp_i,\qquad i=1,\dots,n. 
\end{align}
So we obtain a $J_\rho$-holomorphic $k$-fold covering map
$f=(q,p):\R\times S^1\to C_{(\bar q',\bar p')}$ by setting $q_i=\bar
q_i,p_i=\bar p_i$ for $i=2,\dots,n$, $q_1(s,t):=kt+\bar q_1$, and
defining $p_1(s,t)=p_1(s)$ to be the solution of $\p_sp_1 =
\rho\bigl(|(p_1(s),\bar p')|\bigr)$ with $p_1(0)=\bar p_1$. The
conditions on $\rho$ ensure that $p_1$ defines a diffeomorphism
$\R\to\R$. Hence $f$ is a $J_\rho$-holomorphic cylinder with two
positive punctures asymptotic to the closed geodesics
$t\mapsto(\pm kt,\bar q')$ in the classes $\pm\sigma$. Varying the
constants $(\bar q',\bar p')$, these cylinders form a smooth manifold
of dimension $2n-2$ whose image under the evaluation map is a $k$-fold
covering of $T^*T^n$. So we are done once we show that the moduli
space $\MM_\sigma$ consists only of the $k$-fold orbit cylinders. 

Consider an arbitrary $J_\rho$-holomorphic cylinder $f=(q,p):\R\times
S^1\to T^*T^n$ belonging to $\MM_\sigma$. It is asymptotic at its
punctures to closed geodesics $t\mapsto(\pm kt,q^\pm)$ for
constants $q^\pm=(q^\pm_2,\dots,q^\pm_n)$. Suppose first that $q^+\neq
q^-$, say $q^+_n\neq q^-_n$. Pick a point $s^*\in \R$ at which
$f(s^*,0)=(q^*,p^*)$ with $q_n^*\neq q_n^\pm$. The codimension $2$
hypersurface 
$$
   Y := \{q_n=q_n^*,p_n=p_n^*\} = \bigcup_{\bar
     q_n=q_n^*,\bar p_n=p_n^*}C_{(\bar q',\bar p')}
$$
is a union of orbit cylinders and thus $J_\rho$-invariant. Since
$Y$ lies over the set $\{q_n=q_n^*\}$ and $f$ lies
over the set $\{q_n=q_n^\pm\}$ at infinity, the image of $f$ is
disjoint from $Y$ at infinity and the signed count of intersection
points gives a well-defined homological intersection number
$[f]\cdot[Y]\in\Z$. Since $f$ and $Y$ intersect by construction and
are both $J_\rho$-holomorphic, positivity of intersection implies
$[f]\cdot[Y]>0$. On the other hand, we can deform $Y=Y_0$ through the
hypersurfaces $Y_r:=\{q_n=q_n^*,p_n=p_n^*+r\}$ to $Y_R$ for large
$R$. The preceding argument shows that intersections with $f$ remain
in a compact region during this deformation, hence
$[f]\cdot[Y_R]=[f]\cdot[Y]>0$. Now on each compact subset
$K\subset\R\times S^1$, the component $p_n$ of $f$ is uniformly bounded
and thus not equal to $p_n^*+R$ for large $R$. Hence for large $R$,
intersections of $f$ with $Y_R$ can only occur near infinity in
$\R\times S^1$, where they are excluded by the difference of the
components $q_n$. So we conclude $[f]\cdot[Y_R]=0$ for large $R$,
contradicting our previous assertion. This proves that the constants
$q^\pm$ must coincide, i.e., $f$ is asymptotic at its
punctures to closed geodesics $t\mapsto(\pm kt,\bar q')$ for a
constant $\bar q'=(\bar q_2,\dots,\bar q_n)$. 

The asymptotic behaviour shows that for $i=2,\dots,n$ the component
$q_i$ of $f$ defines a function $q_i:\R\times S^1\to\R$. Adding up
partial derivatives of equations~\eqref{eq:hol}, we see that $q_i$
satisfies the equation
$$
   \Delta q_i + \frac{1}{\rho}(\p_s\rho\,\p_sq_i+\p_t\rho\,\p_tq_i) = 0. 
$$
Since $q_i(s,t)\to\bar q_i$ as $s\to\pm\infty$, the maximum principle
implies that $q_i$ is constant equal to $\bar
q_i$. Equation~\eqref{eq:hol} shows that $p_i$ is constant as well,
hence $f$ defines a $k$-fold covered parametrization of an orbit
cylinder $C_{(\bar q',\bar p')}$. 
\end{proof}

{\bf Broken punctured $K$-holomorphic maps. }
Next we generalize the preceding discussion to broken punctured $K$-holomorphic maps. We fix the following data:
\begin{itemize}
\item a split manifold $X^*=X^+\amalg(\R\times M)\amalg X^-$ with almost complex structure $J^*=(J^+,J_ M, J^-)$, where $J_M$ is $\R$-invariant and $J^\pm$ agree with $J_M$ outside compact subsets $X_0^\pm\subset X^\pm$;
\item a stable nodal curve (of genus zero) $\z$ modelled over the $k$-labelled tree $T$; 
\item for each vertex $\alpha\in T$ a component $X_\alpha$ of $X^*$;
\item asymptotic data $\Gamma^*=(\{\Gamma_{\alpha\beta}\}_{\alpha E\beta},\{\Gamma_i\}_{1\leq i\leq k})$, where each $\Gamma_{\alpha\beta}$ (resp.~$\Gamma_i$) is either a Morse-Bott component of closed Reeb orbits in $M$ (case i), or the manifold $X_\alpha$ (resp.~$X_{\alpha_i}$) (case ii).
\end{itemize}
We assume that all the Morse-Bott families $\Gamma_T$ of closed Reeb orbits in $M$ are manifolds (not orbifolds).  
We denote by $\BB_T(X^*,\Gamma^*)=\prod_{\alpha\in T}\BB_\alpha$ the space of collections $\f=(f_\alpha)_{\alpha\in T}$ of maps
$f_\alpha:\dot S_\alpha\to X_\alpha$ of Sobolev class $W^{m,p,\delta}$ which near each puncture are asymptotic to the corresponding Morse-Bott component of closed Reeb orbits in case (i), or extend smoothly over the puncture in case (ii). 
Let $\KK(\bar\MM_{k+1},TX^*)$ be the space of coherent perturbations in the split manifold $X^*$ as defined in Section~\ref{sec:coh} that satisfy $\|K\|<1$.
Then we get a universal Cauchy-Riemann operator
$$ 
   \pb:\BB_T(X^*,\Gamma^*)\times \KK(\bar\MM_{k+1},TX^*)\to\EE, \qquad
   (\f,K)\mapsto \pb_{J,K_\z}\f = (\pb_{J,K_\z}f_\alpha)_{\alpha\in T}
$$
with values in a suitable bundle 
$\EE_T=\prod_{\alpha\in T}\EE_\alpha\to\BB_T=\prod_{\alpha\in T}\BB_\alpha$. 
Moreover, we have an obvious evaluation map
$$
    \Ev: \BB_T(X^*,\Gamma^*)\to \prod_{\alpha E\beta}\Gamma_{\alpha\beta}\times \prod_{1\leq i\leq k}\Gamma_i =:          \prod_{\alpha\in T}\Gamma_{\alpha},
$$
where $\Gamma_\alpha$ denotes the products of the factors associated to the special points on $\alpha$. 

\begin{prop}\label{prop:surj}
Let $X^*$ and $\z$ be as above. Then the linearization of 
$$ 
   \pb\times\Ev: \BB_T(X^*,\Gamma^*)\times \KK(\bar\MM_{k+1},TX^*)\to\EE_T
   \times \prod_{\alpha E\beta}\Gamma_{\alpha\beta}\times \prod_{1\leq
     i\leq k}\Gamma_i 
$$
at each $(\f,K)$ with $\pb_{J,K_\z}\f = 0$ is surjective. 
\end{prop}

\begin{proof}
Removing the special points from the nodal surface $\Sigma_\z$ yields the disconnected punctured surface $\Sigma_\z^*=\amalg_{\alpha\in T}\dot S_\alpha$. So the space of sections with compact support over $\Sigma_\z^*$ is the direct product
$$
   \KK(\Sigma_\z^*,TX^*) = \prod_{\alpha\in T}\KK(\dot S_\alpha,TX^*).
$$
Combined with the preceding discussion, this shows that the linearization at $(\f,K)$ of the map
$$ 
   \pb\times\Ev: \BB_T(X^*,\Gamma^*)\times \KK(\Sigma_\z^*,TX^*)\to\EE_T \times \prod_{\alpha E\beta}\Gamma_{\alpha\beta}\times \prod_{1\leq i\leq k}\Gamma_i
$$
splits as a product of factors corresponding to $\alpha\in T$ and is therefore surjective by Lemma~\ref{lem:surj}. In view of Lemma~\ref{lem:J-smooth}, this proves the proposition.
\end{proof}

\begin{remark}
Proposition~\ref{prop:surj} continues to hold if we replace the evaluation map at some removable puncture by an $\ell$-jet evaluation map (picking the Sobolev class so high that it embeds into $C^\ell$). This follows by combining the preceding proof with the proof of~\cite[Proposition 6.9]{CM-trans}.  
\end{remark}

Proposition~\ref{prop:surj} leads to regularity of moduli spaces by a
standard argument. 
We keep the fixed tree $T$ and data $X_\alpha$, $\Gamma^*$, but
let $\z$ vary in the stratum $\MM_T$ of Deligne-Mumford space. We
denote by $\Delta^E\subset\prod_{\alpha E\beta}\Gamma_{\alpha\beta}$
the product of the diagonals in $\Gamma_{\alpha\beta}\times
\Gamma_{\beta\alpha}$ for $\alpha E\beta$. Moreover, we fix a
submanifold $\ZZ\subset \prod_{1\leq i\leq k}\Gamma_i$. 
By Proposition~\ref{prop:surj}, the map
$$ 
   \pb\times\Ev: \MM_T\times\BB(X^*,\Gamma^*)\times \KK(\bar\MM_{k+1},TX^*)\to\EE
   \times \prod_{\alpha E\beta}\Gamma_{\alpha\beta}\times \prod_{1\leq
     i\leq k}\Gamma_i 
$$
is transverse to the submanifold $0_\EE\times\Delta^E\times\ZZ$,
where $0_\EE$ denotes the zero section of $\EE$. 
The implicit function theorem implies that
$(\pb\times\Ev)^{-1}(0_\EE\times\Delta^E\times\ZZ)$ is a Banach
manifold. Pick $K\in\KK(\bar\MM_{1+\ell+1},TX^*)$ to be a regular
value of the projection onto the third factor for this space (this
exists by the Sard-Smale theorem). Moreover, we can pick $K$ to be a
regular value also for the space with $\ZZ$ replaced by $\prod_{1\leq
i\leq k}\Gamma_i$, as well as for all the corresponding spaces
associated to individual components of $T$. Then all components of
punctured $K$-holomorphic spheres in the corresponding moduli space 
$\MM^{A,K}_T(X^*,\Gamma^*)$ appear 
in manifolds of the expected dimensions, and the evaluation maps at
the edges are transverse to the diagonal. Hence the total moduli space
(taking preimages of the diagonals under the edge evaluation maps) is
a smooth manifold of the expected dimension, and the evaluation map at
the marked points is transverse to $\ZZ$. To work out the dimensions,
we write
$$
   k = \ol{p}+\ul{p}+m,\qquad
   n_\alpha=\ol{p}_\alpha+\ul{p}_\alpha+m_\alpha, 
$$
where $\ol{p},\ul{p},m$ is the number of marked points correponding to
positive/negative punctures and marked points, respectively, and
similarly for the number $n_\alpha$ of special points on the component
$\alpha$. We also write 
$$
   |T|-1 = e(T) = p(T)+m(T),
$$
where $p(T),m(T)$ is the number of edges corresponding to punctures
and nodes, respectively. We have the obvious relations
\begin{align*}
   \sum_{\alpha\in T}(\ol{p}_\alpha+\ul{p}_\alpha) 
   = \ol{p}+\ul{p}+2p(T), \qquad 
   \sum_{\alpha\in T}m_\alpha &= m + 2m(T), \qquad
   \sum_{\alpha\in T}A_\alpha = A, 
\end{align*}
where $A_\alpha$ denotes the homology class of the map
$f_\alpha$. According to equation~\eqref{eq:dim-Morse-Bott}, the
dimension of the moduli space $\MM_\alpha$ of punctured holomorphic spheres
associated to the component $\alpha$ is now (in the obvious notation)
given by
\begin{align*}
   \dim \MM_\alpha 
   &= (n-3)(2-\ol{p}_\alpha-\ul{p}_\alpha) + 2c_1(A_\alpha) + 2m_\alpha \cr
   &\ \ \ + \sum_{i=1}^{\ol{p}_\alpha}\bigl(\CZ(\ol{\Gamma}_i^\alpha)+\dim\ol{\Gamma}_i^\alpha\bigr)
        - \sum_{j=1}^{\ul{p}_\alpha}\CZ(\ul{\Gamma}_j^\alpha).
\end{align*}
Taking preimages of the diagonals at the edges yields the dimension
formula 
\begin{align*}
   \dim \MM^{A,K}_T(X^*,\Gamma^*)
   &= \sum_{\alpha\in T}\dim\MM_\alpha - \frac{1}{2}\sum_{\alpha
     E\beta}\dim\Gamma_{\alpha\beta} \cr
   &= (n-3)(2|T|-2p(T)-\ol{p}-\ul{p}) + 2c_1(A) + 2m + 4m(T) \cr
   &\ \ \ + \sum_{i=1}^{\ol{p}}\bigl(\CZ(\ol{\Gamma}_i)+\dim\ol{\Gamma}_i\bigr)
        - \sum_{j=1}^{\ul{p}}\CZ(\ul{\Gamma}_j) - 2nm(T) \cr
   &= (n-3)(2-\ol{p}-\ul{p}) + 2c_1(A) + 2m \cr
   &\ \ \ + \sum_{i=1}^{\ol{p}}\bigl(\CZ(\ol{\Gamma}_i)+\dim\ol{\Gamma}_i\bigr)
        - \sum_{j=1}^{\ul{p}}\CZ(\ul{\Gamma}_j) - 2m(T).
\end{align*}
Here in the first line the factor $1/2$ appears because each edge
contributes two terms to the sum over $\alpha E\beta$. For the second
equality we use the dimension formula for $\MM_\alpha$ and the
relations above. Here the terms $\CZ(\ol{\Gamma}_i^\alpha)$ and
$\CZ(\ul{\Gamma}_j^\alpha)$, as well as the dimensions of the Bott
manifolds, corresponding to edges associated to punctures drop out and
we have
$$
   \sum_{\alpha\in T}\sum_{i=1}^{\ol{p}_\alpha}\dim\ol{\Gamma}_i^\alpha
   - \frac{1}{2}\sum_{\alpha E\beta}\dim\Gamma_{\alpha\beta} 
   = \sum_{i=1}^{\ol{p}}\dim\ol{\Gamma}_i - 2nm(T). 
$$
The third equality follows from the computation 
\begin{align*}
   &(n-3)(2|T|-2p(T)) + 4m(T) - 2nm(T) \cr
   &= (n-3)(2+2m(T)) + 4m(T) - 2nm(T) 
   = 2(n-3) - 2m(T). 
\end{align*}
So we have shown

\begin{cor}\label{cor:dim}
For generic $K\in\KK(\bar\MM_{1+\ell+1},TX^*)$, all components of
punctured $K$-holomorphic spheres in the moduli space 
$\MM^{A,K}_T(X^*,\Gamma^*)$ are manifolds of the expected dimensions,
and the evaluation maps at the edges are transverse to the
diagonal. Hence the total moduli space is a manifold of the expected
dimension
\begin{align*}
   \dim \MM^{A,K}_T(X^*,\Gamma^*)
   &= (n-3)(2-\ol{p}-\ul{p}) + 2c_1(A) + 2m \cr
   &\ \ \ + \sum_{i=1}^{\ol{p}}\bigl(\CZ(\ol{\Gamma}_i)+\dim\ol{\Gamma}_i\bigr)
        - \sum_{j=1}^{\ul{p}}\CZ(\ul{\Gamma}_j) - 2m(T).
\end{align*}
Moreover, we can arrange for the evaluation map on this moduli space
at the marked points to be transverse to any given submanifold
$\ZZ\subset \prod_{1\leq i\leq k}\Gamma_i$. \qed 
\end{cor}

\section{Symplectic hypersurfaces in the complement of Lagrangian
  submanifolds}\label{sec:hyp-lag}  


Consider a closed Lagrangian submanifold $L$ 
in a closed symplectic manifold $(X^{2n},\om)$. Suppose that
$\om$ represents an integral relative cohomology
class $[\om]\in H^2(X,L;\Z)$. Denote by $\Tor
H_1(L;\Z)$ the torsion subgroup of $H_1(L;\Z)$ and by $|\Tor
H_1(L;\Z)|\in\N$ its order. Fix a compatible almost complex structure
$J$ on $X$. We say that a submanifold $Y\subset X$ is {\em
approximately holomorphic} if its tangent space at each point differs
from a complex subspace by an angle at most $CD^{-1/2}$, where $C$ is
a constant independent of $D$. We say that two submanifolds
$Y_0,Y_1\subset X$ intersect {\em $\eps$-transversally} if they
intersect transversally and at each intersection point their tangent
spaces form an angle at least $\eps$ for an $\eps>0$ independent of
$D$. The following theorem is an easy adaptation of the main result
in~\cite{AGM}.

\begin{theorem}[Auroux--Gayet--Mohsen~\cite{AGM}]\label{thm:AGM}
For each sufficiently large integer multiple $D$ of
$|\Tor H_1(L;\Z)|$, there exists an approximately $J$-holomorphic closed
codimension two submanifold $Y\subset X\setminus L$ whose Poincar\'e
dual in $H^2(X,L;\Z)$ equals $D[\om]$. 
\end{theorem}

\begin{remark}
It follows directly from the definitions that the submanifold $Y$ 
constructed in Theorem~\ref{thm:AGM} is symplectic, and $\bar
J$-holomorphic for a compatible almost complex structure $\bar J$
arbitrarily $C^0$-close to $J$. 
\end{remark}

The proof of Theorem~\ref{thm:AGM} is based on the following
observation. Consider a compact 
oriented surface $\Sigma$ with (possibly empty) boundary
$\p\Sigma$. Let $E\to\Sigma$ be a Hermitian line bundle together with
a unitary trivialization $\tau:\p\Sigma\times\C\to E|_{\p\Sigma}$ over
the boundary. We identify $\tau$ with the induced section
$z\mapsto\tau(z,1)$ of the unit circle bundle $P\subset E$ over
$\p\Sigma$. Define the {\em relative degree} of $E$ with respect to
the trivialization $\tau$ as 
$$
   \deg(E,\tau) := \sum_{p\in s^{-1}(0)}\ind_p(s)
$$
for any section $s$ in $E$ which is transverse to the zero section and
coincides with $\tau$ along $\p\Sigma$. The following relative version
of the Chern-Weil theorem is proved exactly like the relative
Gauss-Bonnet theorem (see e.g.~\cite{dC-forms}), which is the special
case $E=T\Sigma$. 

\begin{lemma}\label{lem:Chern-Weil}
Let $A\in\Om^1(P,i\R)$ be a unitary connection 1-form on $E$ with
curvature $F_A\in\Om^2(\Sigma,i\R)$. Then
$$
   \deg(E,\tau) = \frac{i}{2\pi}\left(\int_\Sigma
   F_A+\int_{\p\Sigma}\tau^*A\right). 
$$
\end{lemma}

Note that the quantity $\int_{\p\Sigma}\tau^*A$ measures the total
deviation of the section $\tau$ from being parallel. In particular, if
$\tau$ is parallel we have
$$
   \deg(E,\tau) = \frac{i}{2\pi}\int_\Sigma F_A. 
$$

\begin{proof}[Proof of Theorem~\ref{thm:AGM}]
We follow the proof in~\cite{AGM} (for compatibility with the notation
in~\cite{AGM}, we will replace $D$ by $k$).    
Fix a Hermitian line bunde $E\to X$ with first Chern class $[\om]$ and
and a Hermitian connection $A$ on $E$ with curvature form
$F_A=-2\pi i\om$. Let $k$ be an integer multiple of $|\Tor
H_1(L;\Z)|$. The induced connection $A_k$ on $E^{\otimes k}$ has
curvature $F_{A_k}=-2\pi ik\om$. 

We first refine Lemma 2 in~\cite{AGM}. Let $\tau_k$ be unitary section
of $E^{\otimes k}|_L$. For any map $f:(\Sigma,\p\Sigma)\to(X,L)$ from
a compact surface $\Sigma$, Lemma~\ref{lem:Chern-Weil} yields 
$$
   \deg(f^*E^{\otimes k},f^*\tau_k) = k\int_\Sigma\om +
   \frac{i}{2\pi}\int_{\p\Sigma}f^*\tau_k^*A_k.  
$$
In particular, the holonomy
$\frac{i}{2\pi}\int_{\p\Sigma}f^*\tau_k^*A_k$ is 
integral, so we can make it zero by a suitable change of
$\tau_k$. Applying this to all such maps $f$, we find a
unitary section $\tau_k$ such that the closed 1-form $\tau_k^*A_k$
vanishes on the image of the map $\p:H_2(X,L;\Z)\to H_1(L;\Z)$.
Arguing as in~\cite{AGM}, we arrange for $\tau_k$ to satisfy
the estimates in Lemma 2. 

With this modification, the proofs
of Lemmas 3 and 4, and thus of Theorem 2, in~\cite{AGM} work as
before. Thus we find asymptotically holomorphic sections $s_k$ of
$E^{\otimes k}$ that are uniformly transverse
to zero and do not vanish on $L$. Moreover, as remarked in Section 3.3
of~\cite{AGM}, the sections satisfy $|\arg(s_k/\tau_k)|\leq\pi/3$ on
$L$. Hence we can deform $s_k$ near $L$ to arrange $s_k=\tau_k$ on
$L$, without affecting its zero set $Y_k=s^{-1}(0)$. But then for any map
$f:(\Sigma,\p\Sigma)\to(X,L)$ from a compact surface we have
$$
   \deg(f^*E^{\otimes k},f^*\tau_k) = \sum_{p\in
   s_k^{-1}(0)}\ind_p(s_k) = [Y_k]\cdot [f]
$$
by definition of the degree. 
On the other hand, by the choice of $\tau_k$ above and
Lemma~\ref{lem:Chern-Weil}, this degree equals $k\int_\Sigma\om$, and
Theorem~\ref{thm:AGM} is proved.
\end{proof}

\section{Proof of transversality}\label{sec:trans} 

The proof of transversality now follows the scheme in~\cite{CM-trans}.  
Consider a closed symplectic manifold $(X^{2n},\om)$ and a closed
Lagrangian submanifold $L\subset X$. After a perturbation of $\om$, we
may assume that it represents a rational relative cohomology class
$[\om]\in H^2(X,L;\Q)$, so $D_0[\om]\in H^2(X,L;\Z)$ for some positive
integer $D_0$. We fix an energy $E>0$, larger than the symplectic area
of the holomorphic curves we want to consider ($\pi$ in the case
of Theorem~\ref{thm:Audin} (b), and the area of the relevant
Gromov-Witten invariant in the case of
Theorem~\ref{thm:uniruled-dim2}). 

We fix some $\om$-compatible almost complex structure $J_0$ on
$X$. Since the taming condition is open with respect to the $C^0$-norm
$\|\ \|$ on $X$, there exists a constant
$\theta_1>0$ with the following property: for every symplectic form
$\bar\om$ and every almost complex structure $J$ satisfying
$\|\bar\om-\om\|<\theta_1$ and $\|J-J_0\|<\theta_1$,
the symplectic form $\bar\om$ tames $J$. 

Let us pick a tubular neighbourhood $U$ of $L$. Since
$H^2(X,L;\R)\cong H^2(X,U;\R)$, we can represent a basis of
$H^2(X,L;\R)$ by closed forms with compact support in $X\setminus
U$. It follows that there exists an open neighbourhood $\Om$ of $[\om]$
in $H^2(X,L;\R)$ represented by symplectic forms $\bar\om$ which agree
with $\om$ on $U$ and satisfy $\|\bar\om-\om\|<\theta_1$. Thus for
every almost complex structure $J$ with $\|J-J_0\|<\theta_1$
and every $J$-holomorphic map $f:(\Sigma,\p\Sigma)\to(X,L)$ from
a compact Riemann surface with boundary, 
\begin{equation}\label{eq:cone}
   \la[\bar\om,[f]\ra = \int_\Sigma f^*\bar\om \geq 0 \text{ for all
     }[\bar\om]\in \Om. 
\end{equation}
This condition together with an energy bound $\la[\om],[f]\ra<E$
defines the bounded intersection of a cone with a half-space. 
This shows that there is only a finite set $\AA_E\subset H_2(X,L;\Z)$
of homology classes that can be represented by such $J$-holomorphic
maps $f$ with $\|J-J_0\|<\theta_1$ and $\la[\om],[f]\ra<E$. 

Note that~\eqref{eq:cone} continues to hold for every
$J^+$-holomorphic map $f:\Sigma\to X\setminus L$ from a closed Riemann 
surface (of genus $g$) with $p$ punctures asymptotic to closed geodesics
$\gamma_1,\dots,\gamma_p$ on $L$, where $\|J-J_0\|<\theta_1$ and $J^+$
is obtained by deforming $J$ inside $U$ to a cylindrical almost
complex structure tamed by $\om$. Such holomorphic curves appear in
moduli spaces of expected dimension
\begin{align*}
   (n-3)(2-2g-p) - \sum_{i=1}^p \CZ(\gamma_i) 
   &= (n-3)(2-2g-p) + \mu([f]) - \sum_{i=1}^p \ind(\gamma_i) \cr
   &\leq (n-3)(2-2g-p) + \mu([f]).
\end{align*}
Since $[f]$ belongs to the finite set $\AA_E$ from above, restricting
to $g=0$ and $p=1$ we have shown

\begin{lemma}\label{lem:dim-bound}
There exists a constant $D_E>0$ with the following property: for every
$\om$-tamed cylindrical almost complex structure $J^+$ satisfying 
$\|J^+-J_0\|<\theta_1$ outside $U$, every nonempty moduli space of
$J^+$-holomorphic planes in $X\setminus L$ asymptotic to $L$ of energy
$<E$ has expected dimension $\leq D_E$. \qed
\end{lemma}

By Theorem~\ref{thm:AGM}, for each sufficiently large integer multiple
$D$ of $D_0|\Tor H_1(L;\Z)|$, there exists an approximately
$J_0$-holomorphic closed codimension two submanifold $Y\subset
X\setminus L$ whose Poincar\'e dual in $H^2(X,L;\Z)$ equals $D[\om]$. 
By choosing $D$ large, we can make the maximal angle between a tangent
space to $Y$ and a $J_0$-complex subspace (the ``K\"ahler angle'')
smaller than any given constant $\theta_2>0$. As shown
in~\cite{CM-trans}, for $\theta_2$ sufficiently small there exists an
almost complex structure $J$ with $\|J-J_0\|<\theta_1$ such that $Y$
is $J$-complex. 

We fix a Riemannian metric on $L$ such that the unit disk cotangent
bundle $D^*L$ embeds symplectically into $(X\setminus Y)\cap U$. We
deform the $\om$-compatible almost complex structure $J$ inside $U$ to
make it cylindrical near $M=\p D^*L$, so it induces asymptotically
cylindrical almost complex structures $J^+$ on $X^+=X\setminus L$,
$J^-$ on $X^-=T^*L$, and $J_M$ on $\R\times M$. 

We fix a point $x\in L$ and the germ of a $J$-complex hypersurface $Z$
at $x$. We perturb $J$ such that all moduli spaces of {\em simple}
punctured holomorphic curves for $J^\pm$ and $J_M$ (including tangency
conditions to $Y$ and $Z$) as well as $J|_Y$ are transversely cut
out. By \cite[Proposition 6.9]{CM-trans}, tangency of order $\ell$ to
$Y$ (at some point which is not fixed) is a condition of codimension
$2\ell$ for simple curves. In view of Lemma~\ref{lem:dim-bound}, the
maximal order of tangency to $Y$ of simple $J^+$-holomorphic planes in
$X\setminus L$ of energy $<E$ is therefore bounded by $D_E/2$. Since
the homology classes of $J^+$-holomorphic planes in $X\setminus L$
belong to the finite set $\AA_E\subset H_2(X,L;\Z)$ from above, their
multiplicity is bounded by some constant $M_E$ and we have shown  

\begin{lemma}\label{lem:tangency-bound}
The maximal order of tangency to $Y$ of (not necessarily simple) $J^+$-holomorphic
planes in $X\setminus L$ of energy $<E$ is bounded by $(M_E+D_E)/2$. \qed 
\end{lemma}

Since the first Chern class of the normal bundle $N(Y,X)$ is
represented by $D\om|_Y$, the adjunction formula
$$
   c_1(TY) = c_1(TX|_Y) - c_1\bigl(N(Y,X)\bigr) = c_1(TX|_Y) -
   D[\om|_Y] 
$$
shows that raising the degree $D$ of $Y$ lowers the expected dimension 
of moduli spaces of nonconstant $J$-holomorphic spheres in $Y$. So for
$D$ sufficiently large and energy bounded by $E$, these do not
occur (because $J|_Y$ is regular on simple curves).
Next, observe that the intersection number with $Y$ of a
nonconstant punctured $J$-holomorphic curve $f:\Sigma\to X\setminus L$
satisfies 
$$
   [f]\cdot[Y] = D\int_\Sigma f^*\om \geq D/D_0
$$
because $D_0[\om]\in H_2(X,L;\Z)$. So the intersection number is
positive and increases with $D$. Since a point of tangency of order
$\ell$ to $Y$ contributes $(\ell+1)$ to the intersection number
(see~\cite{CM-trans}), and the maximal order of tangency for
holomorphic planes is bounded by Lemma~\ref{lem:tangency-bound}, we
have shown

\begin{lemma}\label{lem:domain-stable}
For $D$ sufficiently large (depending only on $\om$, $J_0$ and $E$)
and $J^+$ as above the following holds:
\begin{enumerate}
\item all $J^+$-holomorphic spheres of energy $<E$ in $Y$ are constant; 
\item all nonconstant punctured $J^+$-holomorphic curves in
  $X\setminus L$ intersect $Y$; 
\item all nonconstant $J^+$-holomorphic planes of energy $<E$ in
  $X\setminus L$ intersect $Y$ in at least $2$ distinct points in the
  domain. \qed
\end{enumerate} 
\end{lemma}

Consider now the homology class $A\in H_2(X;\Z)$ of holomorphic
spheres we want to split along $M$, where $\om(A)<E$. Thus $A$ is the
class of a complex line in $\C\P^n$ in the case of
Theorem~\ref{thm:Audin} (b), and a class with a nonvanishing genus
zero Gromov-Witten invariant through a point (plus additional
constraints) in the case of Theorem~\ref{thm:uniruled-dim2}. Note that
the intersection number of $J$-holomorphic spheres in the class $A$
with the hypersurface $Y$ equals 
$$
   \ell :=  [Y]\cdot A = D\,\om(A). 
$$
Consider the space of coherent perturbations
$\KK(\bar\MM_{k+\ell+1},TX^*)$ in the bundle $(TX^*,J)$ over the split
manifold $X^*=X^+\amalg 
(\R\times M)\amalg X^-$, as defined in Section~\ref{sec:coh}. Here
$k=1$ in the first case, and $k$ is the number of constraints
$(Z_1=x,Z_2,\dots,Z_k)$ for the Gromov-Witten invariant in the second
case. Recall that $K\in \KK(\bar\MM_{k+\ell+1},TX^*)$ is required to 
satisfy $\|K\|<1$, where the norm of $K$ is defined with respect to the
metrics $\sigma(\cdot,j\cdot)$ and $\om(\cdot,J\cdot)$ as in
Section~\ref{sec:coh}. 
In the first case, let 
$$
   \KK(\bar\MM_{1+\ell+1},TX^*;Y,x)\subset \KK(\bar\MM_{1+\ell+1},TX^*)
$$ 
be the subset of $K$ which vanish near $Y$ as well as near the point
$x$. In the second case, let 
$$
   \KK(\bar\MM_{k+\ell+1},TX^*;Y)\subset \KK(\bar\MM_{k+\ell+1},TX^*)
$$ 
be the subset of $K$ which vanish near $Y$. 
Pick such a $K$. In the first case, let  
$$
   \bar\MM = \bar\MM_{1+\ell}([\C\P^1],J,K;x,Z,Y)
$$ 
be the moduli space of broken $(J,K)$-holomorphic spheres in $X^*$ in
the class $[\C\P^1]$ with $\ell$ marked points mapped to $Y$, and $1$ marked
point passing through $x$ and tangent of order $n-1$ to $Z$ at $x$. 
In the second case, let  
$$
   \bar\MM = \bar\MM_{k+\ell}(A,J,K;Z_1,\dots,Z_k,Y)
$$ 
be the moduli space of broken $(J,K)$-holomorphic spheres in $X^*$ in
the class $A$ with $\ell$ marked points mapped to $Y$, and $k$ marked
points mapped to $Z_1,\dots,Z_k$. Note that $\MM$ has the same
expected dimension as the corresponding moduli space without the
$\ell$ additional marked points mapped to $Y$. By an argument using
positivity of intersections as in~\cite{CM-trans},
Lemma~\ref{lem:domain-stable} implies

\begin{prop}\label{prop:domain-stable}
For $D$ sufficiently large (depending only on $\om$, $J_0$ and $E$),
$(J,K)$ as above, and any broken holomorphic curve $f\in\MM$ the
following holds: 
\begin{enumerate}
\item all components of $f$ contained in $Y$ are constant; 
\item all components of $f$ in $X\setminus L$ have a stable
  domain. \qed 
\end{enumerate} 
\end{prop}

Now we can conclude the transversality part of the proofs of
Theorems~\ref{thm:Audin} (b) and~\ref{thm:uniruled-dim2}. In view of
the preceding discussion, it suffices to show that all the
components of broken holomorphic spheres in $\ol{\MM}$ belong to moduli
spaces that are manifolds of the expected dimensions, and in the first
case the asymptotic evaluation maps are transverse to the diagonal. 

\begin{proof}[Proof of transversality for
    Theorem~\ref{thm:Audin}(b)]
Here $L$ is an $n$-dimensional torus with a multiple of the standard
flat metric. Hence there are no holomorphic planes in $T^*L$ and
$\R\times M$, and the only holomorphic cylinders in $T^*L$ and
$\R\times M$ are orbit cylinders in $\R\times M$, and holomorphic
cylinders in $T^*L$ with two positive punctures. 
According to Proposition~\ref{prop:domain-stable}, all other punctured
holomorphic spheres in $T^*L$, $\R\times M$ or $T^*L$ of energy $<E$
have a stable domain. 

Consider a stable $(J,K)$-holomorphic map $(\z,\f)\in\bar\MM$. Thus
$\z$ is a nodal curve modelled over a (not necessarily stable) 
$(1+\ell)$-labelled tree $T$ such that the map $f_\alpha:\dot
S_\alpha\to X^*$ is nonconstant over each unstable component
$\alpha\in T$. As in Section~\ref{sec:lin}, we consider 
the linearization $L$ of 
$$ 
   \pb\times\Ev: \BB(X^*,\Gamma^*)\times
   \KK(\bar\MM_{1+\ell+1},TX^*;Y,x)\to\EE \times \prod_{\alpha
     E\beta}\Gamma_{\alpha\beta}\times \prod_{1\leq i\leq
     1+\ell}\Gamma_i 
$$
at $(\f,K)$. Here the normal $(n-1)$-jet evaluation map at $z_1$ takes
values in $\Gamma_1=(T_xT^*L/T_xZ)^n$, while the evaluation maps at
the points $z_i$ take values in $\Gamma_i=X\setminus L$ for
$i=2,\dots,\ell+1$. The points $z_{\alpha\beta}$ can correspond to to
punctures, in which case $\Gamma_{\alpha\beta}$ is a component of 
the space of closed geodesics, or to nodes, in which case
$\Gamma_{\alpha\beta}$ is a component of $X^*$. 
In contrast to Proposition~\ref{prop:surj}, we cannot conclude that
$L$ is surjective for two reasons:

(1) The perturbations in $\KK(\bar\MM_{1+\ell+1},TX^*;Y,x)$ are
required to vanish near $Y$ and near $x$, so they are of no use to
achieve transversality of the evaluation map on components $f_\alpha$
that are mapping entirely into $Y\cup\{x\}$. This can be remedied 
as follows. 

Note first that no component can be mapped entirely into the point
$x$. Indeed, such a component would be constant and carry at most one
marked point (because the points $z_2,\dots,z_{\ell+1}$ are mapped to
$Y$ and $x\notin Y$). In view of stability, the component must carry
at least two nodal points and thus split the tree into at least two
subtrees. Each of these subtrees must contain a component in $\C
P^n\setminus L$ and thus represent a nontrivial homology class in $\C
P^n$, which is impossible because the total homology class $[\C\P^1]$
is indecomposable.  

Next consider a component $f_\alpha$ that is mapped entirely into $Y$. 
By Proposition~\ref{prop:domain-stable}(a), the map
$f_\alpha$ is constant. Following~\cite{CM-trans}, we denote by
$T_1\subset T$ the maximal subtree containing $\alpha$ consisting of
constant components in $Y$ (a ``ghost tree''). Indecomposability of
the homology class $[\C\P^1]$ implies that $T\setminus T_1$ is
connected, so by stability $T_1$ must contain at least two of the
marked points $z_2,\dots,z_{\ell+1}$. According to~\cite[Proposition
7.1 and Lemma 7.2]{CM-trans}, the nonconstant map $f_\beta$ adjacent
to $T_1$ is tangent to $Y$ at the point $z_{\beta\alpha}$. We replace
$T$ by $T\setminus T_1$ and consider the remaining moduli space with
an additional tangency condition to $Y$ imposed at the new marked
point $z_{\beta\alpha}$. The new evaluation map (still denoted by
$\Ev$) contains the $1$-jet evaluation map at $z_{\alpha\beta}$.  
This modification is performed for all ghost
trees that are mapped to $Y$. Note that this decreases the number
of marked points mapped to $Y$, but we keep denoting it by $\ell$. 

(2) The tree $T$ need not be stable, and by construction the
perturbations in $\KK(\bar\MM_{1+\ell+1},TX^*;Y,x)$ are domain
independent on unstable components. Now by the discussion above,
unstable components can only correspond to cylinders in $T^*L$ with
two positive punctures (orbit cylinders can be removed from a stable
map). Thus no two such components can be
adjacent. Moreover, by Lemma~\ref{lem:torus}, they appear in regular
moduli spaces of the expected dimension $2(n-1)$ whose image under the
evaluation map is a covering of $T^*L$. So we can use surjectivity of
the linearized evaluation map at the adjacent components to achieve
transversality to the diagonal at the punctures of each such
component. 

In view of this discussion and Proposition~\ref{prop:surj}, after the
modification in (1), the linearized operator $L$ is transverse to the
manifold
$$
   \ZZ := \{0_\EE\}\times \prod_{\alpha E\beta}\Delta_{\alpha\beta}\times
   \{0_{\Gamma_1}\}\times \prod_{2\leq i\leq 1+\ell}Y_i.
$$
Here $\Delta_{\alpha\beta}$ denotes the diagonal in
$\Gamma_{\alpha\beta}\times \Gamma_{\beta\alpha}$ at the edge
$\alpha E\beta$, and $Y_i$ denotes $TY$ at the new marked points $z_i$
where we take the $1$-jet evaluation map, and $Y$ at the other ones. 

Now the proof is finished by the standard argument as in the proof 
of Corollary~\ref{cor:dim}: For each tree $T$
as above, the implicit function theorem implies that
$(\pb\times\Ev)^{-1}(\ZZ)$ is a Banach manifold. Pick
$K\in\KK(\bar\MM_{1+\ell+1},TX^*;Y,x)$ to be a regular value of the
projection onto the second factor for each of these spaces (this
exists by the Sard-Smale theorem). Then all components of broken
holomorphic spheres in the corresponding moduli space $\bar\MM$ appear 
in manifolds of the expected dimensions, and the evaluation maps at
the edges are transverse to the diagonal. Hence the total moduli space
(taking preimages of the diagonals under the edge evaluation maps) is
a smooth manifold. According to Corollary~\ref{cor:dim}, its dimension
is given by 
\begin{align*}
   \dim\bar\MM 
   &= (2n-6) + (2n+2) - (4n-4) + 2\ell -
   \sum_{i=2}^{\ell+1}\codim(Y_i) - 2m(T) \cr
   &= 2\ell - \sum_{i=2}^{\ell+1}\codim(Y_i) - 2m(T). 
\end{align*}
Here the term $(2n+2)$ comes from $c_1(\C\P^n)$ evaluated on $[\C
P^1]$, $(4n-4)$ from the tangency condition to $Z$ at $x$ (see
Section~\ref{sec:proofs-trans}), $\codim(Y_i)$ equals $2$ if
$Y_i=Y$ and $4$ if $Y_i=TY$, and $m(T)$ is the number of nodes (i.e.,
edges of $T$ over which the maps extend continuously). 
It follows that the moduli space can be nonempty (and thus the
dimension nonnegative) only if $m(T)=0$ and $Y_i=Y$ for all
$i$, so nodes and components mapped into $Y$ as discussed in (1) above
cannot occur. (Nodes are actually also excluded by indecomposability
of the homology class $[\C\P^1]$.)

The same discussion can be applied to the vertex containing the
marked point $z_1$ (which is mapped to $x$ with tangency to $Z$)
and the $m$ subtrees obtained by removing this vertex. Again by
Corollary~\ref{cor:dim}, their dimensions are given by  
\begin{align*}
   \dim\MM_{i} &= (n-3) - \CZ(\Gamma_i),\qquad i=1,\dots,m, \cr
   \dim\MM &= (n-3)(2-m) + (2n+2) + \sum_{\i=1}^{m}\bigl(\CZ(\Gamma_{i}) + \dim
   \Gamma_i\bigr) - (4n-4). 
\end{align*}
This concludes the proof of transversality for
Theorem~\ref{thm:Audin}(b). 
\end{proof}

\begin{proof}[Proof of transversality for
    Theorem~\ref{thm:uniruled-dim2}]
Here $L$ is a closed orientable surface with a metric of negative
curvature. Thus there are 
no holomorphic planes in $T^*L$ and $\R\times M$, and the only
holomorphic cylinders in $T^*L$ and $\R\times M$ are
\begin{enumerate}
\item orbit cylinders in $\R\times M$, and 
\item holomorphic cylinders in $T^*L$ projecting onto a closed
  geodesic $\gamma$ with two positive punctures asymptotic to
  $\pm\gamma$.   
\end{enumerate}
(This follows from a maximum principle argument similar to the proof
of Lem\-ma~\ref{lem:torus}.) 
Orbit cylinders can be removed in the Gromov--Hofer compactification. 
The cylinders in (ii) appear in regular moduli spaces
of the expected dimension $0$. We choose the point $x\in L$ so that it
does not meet these moduli spaces. 
According to Proposition~\ref{prop:domain-stable}, all other punctured
holomorphic spheres in $T^*L$, $\R\times M$ or $T^*L$ of energy $<E$
have a stable domain. Now the proof is concluded by the same arguments
as in the preceding proof of Theorem~\ref{thm:Audin}(b), which are
easier here because we have to consider neither Bott manifolds of
closed geodesics, nor a higher jet evaluation map at $x$.  
\end{proof}

%
%
%

\appendix

\section{The Chekanov torus}\label{app:Chekanov}

In this section we prove the ``No Go'' result mentioned in the
Introduction: {\em For the Chekanov torus $T_{\rm Chekanov}\subset \C\P^2$,
the boundary loops of the three disks produced in
Theorem~\ref{thm:disk} will in general represent multiples of the same
homology class and thus fail to generate the first homology of the
torus.} More precisely, we will see that this occurs whenever we choose
the almost complex structure in the proof of Theorem~\ref{thm:disk}
from a suitable nonempty open set. 

We begin by recalling the construction of the Chekanov torus $T_{\rm
  Chekanov}\subset \C\P^2$ following~\cite{CS}. Let
$\gamma\subset\{z\in \C \mid {\rm Re}\,z>0\}$ be a closed embedded
curve in the right half-plane, bounding a closed disk $D_\gamma$ of area
$\pi/3$ which contains $1\in\C$ in its interior. Denote their images under the
diagonal embedding $z\in\C\mapsto [1:z:z]\in\C\P^2$ by
$$
   \Gamma:=\{[1:z:z]\mid z\in\gamma\}\subset 
   D_\Gamma:=\{[1:z:z]\mid z\in D_\gamma\}\subset \C\P^2.  
$$
Consider the symplectic $S^1$-action on $\C\P^2$ given by
$e^{i\theta}\cdot [z_0:z_1:z_2] :=
[z_0:e^{i\theta}z_1:e^{-i\theta}z_2]$. The {\em Chekanov torus} is the
union of the orbits of this action through $\Gamma$,
$$
   L := T_{\rm Chekanov} := \{[1:e^{i\theta}z:e^{-i\theta}z]\mid
   z\in\gamma,\,e^{i\theta}\in S^1\}\subset \C\P^2.
$$
Note that $\p D_\Gamma=\Gamma\subset L$. 
We pick a point $v\in\gamma$ with $v\notin\R$ and set
$$
   \tau:=\{[1:e^{i\theta}v:e^{-i\theta}v]\mid
   e^{i\theta}\in S^1\}\subset  
   D_\tau:=\{[1:\zeta v:\bar\zeta v]\mid
   \zeta\in D\}\subset \C\P^2,
$$
where $D\subset \C$ denotes the closed unit disk. Thus $\tau$ is the
orbit of the point $[1:v:v]$ under the $S^1$-action and $\p
D_\tau=\tau\subset L$. We equip $D_\Gamma$ and $D_\tau$ with the
orientations induced from $D_\gamma$ resp.~$\Delta$ via the 
embeddings. Finally, let us introduce the projective lines $S_0,
S_1,S_2$ and the quadric $Q$ in $\C\P^2$,
$$
   S_i:=\{[z_0:z_1:z_2]\mid z_i=0\},\ i=0,1,2,\qquad  
   Q:=\{[z_0:z_1:z_2]\mid z_0^2=z_1z_2\}.
$$
It follows directly from the defining equations that
$S_0,S_1,S_2,Q\subset \C\P^2\setminus L$.
Since $H_1(\C\P^2)=0$ and $[L]=0\in H_2(\C\P^2)$, the long exact
sequence of the pair $(\C\P^2,L)$ shows that $H_2(\C\P^2,L)\cong\Z^3$
is generated by the relative homology classes represented by
$D_\Gamma$, $D_\tau$ and $S_0$ (which we will denote by the same
symbols). 

\begin{lemma}[Chekanov--Schlenk~\cite{CS}]\label{lem:CS}
(a) The Chekanov torus $L\subset \C\P^2$ is a monotone Lagrangian
$2$-torus in $\C\P^2$ with $\int_A\omega = \frac{\pi}{6}\mu(A)$ for
  the Maslov class $\mu$ and all $A\in H_2(\C\P^2,L)$. 

(b) Let $J$ be an almost complex structure on $\C\P^2\setminus L$ with
a negative cylindrical end for which $S_0,S_1,S_2$ and $Q$ are complex
submanifolds. Then the only relative homology classes in
$H_2(\C\P^2,L)$ which may contain punctured $J$-holomorphic
spheres of Maslov number $2$ are
$$
   D_\Gamma,\ S_0-2D_\Gamma, S_0-2D_\Gamma\pm D_\tau. 
$$
Apart from the doubles of these classes, the only classes which may
contain punctured $J$-holomorphic spheres of Maslov number $4$ are 
$$
   S_0-D_\Gamma,\ S_0-D_\Gamma\pm D_\tau,\ 2S_0-4D_\Gamma \pm D_\tau.
$$
\end{lemma}

\begin{proof}
Part (a) is proved in~\cite{CS}. Part (b) is carried out in~\cite{CS}
for $S^2\times S^2$ and the arguments easily carry over to $\C\P^2$;
the $\C\P^2$ case is also treated in~\cite{Aur}. For convenience, we
include the argument. The following table shows the relevant
intersection and Maslov numbers in $\C\P^2$. Here the Maslov numbers
can be taken from~\cite{CS} since they are the same in $S^2\times S^2$
and $\C\P^2$, and the intersection numbers follow easily from the
defining equations and inspection of orientations at intersection
points. 
$$
\begin{array}{c|c|c|c||c}
          & D_\Gamma & D_\tau & S_0 & A=a_\Gamma D_\Gamma + a_\tau D_\tau + bS_0\\
  \hline
  S_0 & 0                   & 0           & 1      & b\ge 0\\
  S_1 & 0                   & -1           & 1      & -a_\tau + b\ge 0\\
  S_2 & 0                   & 1          & 1      & a_\tau + b\ge 0\\
  Q     & 1                   & 0            & 2      & a_\Gamma + 2b\ge 0\\
  \hline
  \hline
  \mu  & 2                   & 0            & 6      & 2a_\Gamma + 6b = 2 \mbox{ or  }4
\end{array}
$$
Now suppose that a class $A=a_\Gamma D_\Gamma + a_\tau D_\tau +
bS_0\in H_2(\C\P^2,L)$ contains a punctured $J$-holomorphic sphere. 
Since $S_0,S_1,S_2$ and $Q$ are $J$-complex submanifolds, all the
intersection numbers in the last column of the table must be
nonnegative. 
Together with the condition on the Maslov number this yields the
result. 
%
\end{proof}

\begin{cor}\label{cor:CS}
Let $L=T_{\rm Chekanov}\subset \C\P^2$ be the Chekanov torus equipped
with a flat metric. 
Let $C_0,\dots,C_k\subset\C\P^2\setminus L$ be the punctured
$J$-holomorphic spheres resulting from degenerating holomorphic
spheres in the class $[\C\P^1]$ along the boundary of a tubular
neighbourhood of $L$, where $J$ is an
almost complex structure on $\C\P^2\setminus L$ satisfying the conditions in
Lemma~\ref{lem:CS}(b). Then the boundary loops of the $C_i$ all
represent multiples of the same homology class $[\Gamma]\in H_1(L)$. 
\end{cor}

\begin{proof}
Let $F=(F^{(1)},\dots,F^{(N)})$ be a broken holomorphic sphere in
$\C\P^2$ with $N$ obtained as the limit of holomorphic
spheres in the class $[\C\P^1]$ in the neck-stretching procedure. 
Thus the components $C_0,\dots,C_k$ of $F^{(N)}$ are punctured
$J$-holomorphic spheres in $\C\P^2\setminus L$, where we assume that
the almost complex structure $J$ on $\C\P^2\setminus L$ satisfies the
conditions in Lemma~\ref{lem:CS}(b). 
Note that all components $C_i$ have positive symplectic area and the
sum of the areas equals $\pi$. By monotonicity of $L$, the Maslov
number of each $C_i$ equals $6/\pi$ times its area. So all Maslov
numbers are positive even integers and their sum equals $6$, which
leaves only the combinations $(2,2,2)$, $(4,2)$ and $(2,4)$.  

Since the total intersection number of $F$ with $S_0$ equals $1$,
there is a unique component, say $C_0$, which intersects $S_0$. 
According to Lemma~\ref{lem:CS}, the components which do not intersect
$S_0$ are homologous to either $D_\Gamma$ or $2D_\Gamma$ in
$H_2(\C\P^2,L)$. Since the sum of the relative homology classes of all
the components of $F^{(N)}$ equals $S_0$, this leaves only the
following possibilities for the relative homology classes of
$C_0,\dots,C_k$: 
$$
   (S_0-2D_\Gamma,D_\Gamma,D_\Gamma)\quad \text{or}\quad 
   (S_0-2D_\Gamma,2D_\Gamma)\quad \text{or}\quad   
   (S_0-D_\Gamma,D_\Gamma).
$$
In particular, the boundary of each component $C_i$ is homologous to a
multiple of $\Gamma$ (where the boundary of $C_i$ is the sum of the asymptotic
geodesics at its punctures). Hence either all asymptotic geodesics 
of such a component are homologous to multiples of $\Gamma$, or at least 
two of them are not. The same holds for the components of the broken
holomorphic curve $F$ in
$T^*L$ and $\R\times S^*L$. Recall that one may assign to $F$ a tree
whose nodes correspond to the components of 
$F$, and whose edges correspond to common asymptotic geodesics between
two nodes resulting from the limit process.
Hence the ends of the tree correspond to components with exactly one
puncture which, by the preceding discussion, has to be asymptotic to a
multiple of $\Gamma$. Remove all ends together with
their adjacent edges. For each punctured holomorphic sphere
corresponding to a node of the remainig tree it is still true that the
asymptotics of the punctures corresponding to the remaining nodes are
either all homologous to multiples of $\Gamma$, or at least two of
them are not. (This holds since all edges we removed correspond to
punctures asymptotic to multiples of $\Gamma$.)  
By induction over the number of nodes in the tree (since it holds for
one node), it follows that all asymptotic geodesics are homologous to
multiples of $\Gamma$ and the corollary follows.
\end{proof}

\begin{remark}
(a) For the Chekanov torus, the proof of Theorem~\ref{thm:disk}
  (degenerating holomorphic spheres with a tangency condition)
  produces precisely three holomorphic planes $C_0,C_1,C_2$ in
  $\C\P^2\setminus L$, which by the proof of Corollary~\ref{cor:CS}
  must represent the relative homology classes
  $(S_0-2D_\Gamma,D_\Gamma,D_\Gamma)$. 

(b) Corollary~\ref{cor:CS} continuous to hold (with the same proof)
  for a tamed almost complex structure $J$ for which there exist
  complex submanifolds in $\C\P^2\setminus L$ that are homologous (in
  the complement of $L$ to $S_0,S_1,S_2$ and $Q$. The set of these
  almost complex structures is open. 

(c) Carrying over arguments from~\cite{CS} to the $\C\P^2$ case, it is
  shown in~\cite{Aur} that each of the relative homology classes of
  Maslov number $2$ listed in Lemma~\ref{lem:CS} contains a unique
  holomorphic disk passing through a given point on $L$. Neck
  stretching at the boundary of a tubular neighbourhood of $L$ yields
  holomorphic planes in $\C\P^2\setminus L$ representing these
  classes. This shows that there exist collections of holomorphic
  planes, e.g.~three planes representing the classes
  $(S_0-2D_\Gamma+D_\tau,S_0-2D_\Gamma-D_\tau,4D_\gamma)$, whose
  asymptotic geodesics do generate $H_1(L)$. According to
  Corollary~\ref{cor:CS}, such collections of planes do not arise in
  degenerations of holomorphic spheres in the class $[\C\P^1]$.  
\end{remark}


\section{A non-removable intersection}\label{app:rigid}

Here we prove the claim in Remark~\ref{rem:rigid-intro} from the
introduction, which we restate as follows. 

\begin{prop}
Let $(L_t)_{t\in[0,1]}$ be a Hamiltonian isotopy of closed Lagrangian
submanifolds in the closed unit ball $B\subset\C^n$. If $L_0\subset\p
B$, then $L_t\subset\p B$ for all $t\in[0,1]$.
\end{prop}

\begin{proof}
The set $\{t\in[0,1] \mid L_t\subset\p B\}$ is clearly nonempty and
closed, so we only need to show it is open. So let $L:=L_t\subset\p
B$. Let $v$ be the vector field on $\p B$ generating the Hopf
fibration (which is the characteristic foliation). Then $\om(v,w)=0$
for all $w\in TL\subset T(\p B)$, so $v\in (TL)^{\perp_\om}=TL$. This
shows that $L$ is invariant under the Hopf fibration, hence it
descends to a Lagrangian submanifold $\bar L\subset\C
P^{n-1}$. By the Lagrangian neighbourhood theorem and symplectic
reduction, a neighbourhood of $L$ in $\C^n$ is symplectomorphic
to a neighbourhood $U$ of the zero section in
$$
   (T^*\bar L\times T^*S^1,dp\wedge dq).
$$
Here $(q_1,\dots,q_{n-1},p_1,\dots,p_{n-1})$ are canonical coordinates
on $T^*\bar L$ and $(q_n,p_n)$ on $T^*S^1$. Moreover, the variable
$p_n$ corresponds to the moment map $z\mapsto\pi(|z|^2-1)$ of the
standard circle action on $\C^n$, so $p_n> 0$ outside $B$ and
$p_n\leq 0$ in $B$.

For $s$ sufficiently close to $t$ we can write $L_s$ as the graph of
an exact 1-form in $U$, i.e.,
$$
   L_s=\{(q,p) \mid p_i=\frac{\p f}{\p q_i}\}
$$
for a function $f:\bar L\times S^1\to\R$.  Fix a point $\bar
q=(q_1,\dots,q_{n-1})$. Since $L_s$ is contained in $B$, we have
$\frac{\p f}{\p q_n}(\bar q,q_n)=p_n\leq 0$ for all $q_n$. On the
other hand,
$$
   \int_0^{2\pi}p_n\,dq_n = \int_0^{2\pi}\frac{\p f}{\p q_n}(\bar
   q,q_n)dq_n=0.
$$
So $p_n$ must vanish identically, which means that $L_s\subset\p B$.
\end{proof}

\section{The embedding capacity for the flat torus}\label{app:emb-cap}

Let $T^n=(\R/2\pi\Z)^n$ be the standard flat torus, normalized such
that each $S^1$-factor has length $2\pi$. Equip $\C\P^n$ with the
standard symplectic form such that a complex line has area $\pi$, so
$\vol(\C\P^n)=\vol B^{2n}(1)$. Then the volume
estimate~\eqref{eq:volume} becomes
\begin{equation}\label{eq:volume2}
   c^{\C\P^n}(T^n) \geq \sqrt[n]{\frac{(2\pi)^n\vol
   B^n(1)}{\vol B^{2n}(1)}} =: C_n, 
\end{equation}
whereas the estimate from Theorem~\ref{thm:emb-cap} reads
\begin{equation}\label{eq:emb-cap2}
   c^{\C\P^n}(T^n)\geq 2(n+1).
\end{equation}
In view of the well-known formulae
$$
   \vol B^{2k}(1)=\frac{\pi^k}{k!},\qquad \vol
   B^{2k+1}(1)=\frac{2^{2k+1}\pi^kk!}{(2k+1)!}
$$
the right hand side of equation~\eqref{eq:volume2} can be written as 
$$
   C_{2k} = 2\sqrt{\pi}\sqrt[2k]{\frac{(2k)!}{k!}}, \qquad
   C_{2k+1} = 4\sqrt[2k+1]{\pi^kk!}.
$$
Note that the estimates~\eqref{eq:volume2} and~\eqref{eq:emb-cap2} agree
for $n=1$. 

\begin{lemma}
For every integer $n\geq 2$ we have $C_n<2(n+1)$, so the
estimate~\eqref{eq:emb-cap2} is better than the volume
estimate~\eqref{eq:volume2}.  
\end{lemma}

\begin{proof}
For $n=2k\geq 2$ even we use $2\pi k<(2k+1)^2$ to estimate
$$  
   C_{2k} \leq 2\sqrt{\pi}\sqrt[2k]{(2k)^k} = 2\sqrt{2\pi k} < 2(2k+1) =
   2(n+1). 
$$
Similarly, for $n=2k+1\geq 3$ odd we use $\pi k<(k+1)^2$ to estimate
$$  
   C_{2k+1} = 4\sqrt[2k]{\pi^kk^k} = 4\sqrt{\pi k} < 4(k+1) = 2(n+1). 
$$
\end{proof}

Finally, we wish to compare the estimates to the values realized by
the obvious symplectic embeddings 
$$
   \phi:\Bigl((-1,1)\times(\R/2\pi\R)\Bigr)^n\into\C^n,\qquad
   (s_j,t_j)\mapsto z_j:=\sqrt{2+2s_j}e^{it_j}. 
$$

\begin{lemma}
The embedding $\phi$ yields the upper bound 
$$
   c^{\C\P^n}(T^n) \leq 2(n+\sqrt{n}),
$$
which agrees with the lower bound~\eqref{eq:emb-cap2} iff $n=1$. 
\end{lemma}

\begin{proof}
In the unit codisk bundle $D^*(T^n)$ we have $\sum_{j=1}^ns_j^2\leq
1$. The maximum of the squared norm $|z|^2=\sum_{j=1}^n(2+2s_j)$ on
the image of $\phi$ under this constraint is attained for
$s_1=\dots=s_n=1/\sqrt{n}$ and given by 
$$
   \max|z^2| = \sum_{j=1}^n(2+2/\sqrt{n}) = 2(n+\sqrt{n}). 
$$  
Rescaling the symplectic form by this factor yields the upper bound. 
\end{proof}


\end{document}